\documentclass{svjour3} 
\usepackage{needspace}
\usepackage{etoolbox}

\BeforeBeginEnvironment{lemma}{\Needspace{2\baselineskip}}
\BeforeBeginEnvironment{theorem}{\Needspace{2\baselineskip}}
\BeforeBeginEnvironment{proposition}{\Needspace{2\baselineskip}}
\BeforeBeginEnvironment{corollary}{\Needspace{2\baselineskip}}

\usepackage{graphicx}
\usepackage{caption}
\usepackage{subcaption}
\usepackage{amssymb,amsmath}

\usepackage{algorithm}
\usepackage{algpseudocode}
\algrenewcommand{\algorithmicrequire}{\textbf{Input:}}
\algrenewcommand{\algorithmicensure}{\textbf{Output:}}
\newcommand{\Input}{\Require}
\newcommand{\Output}{\Ensure}
\usepackage{enumitem}

\usepackage{xcolor}
\usepackage{pgf,tikz,pgfplots}
\usepackage{tkz-euclide}
\usepackage{graphicx}
\usepackage{color}
\usepackage{hyperref}
\usepackage[all]{xy}
\usepackage{mathrsfs}
\usepackage{tikz,tkz-tab}
\usepackage{parallel}
\usepackage{array}
\usepackage{xcolor}
\usepackage{verbatim}
\usepackage{tikz-3dplot}
\usepackage[title]{appendix}
\usepackage{tikz-cd}
\usepackage{multirow}
\usepackage{makecell}
\usepackage{comment}

\renewcommand\labelenumi{(\alph{enumi})}
\renewcommand\theenumi\labelenumi

\title{Last two pieces of the puzzle for unsolvability of a system of two quadratic (in)equalities}

\author{Min-Chi Wang \and
        Ruey-Lin Sheu \and
        Huu-Quang Nguyen
}

\institute{Min-Chi Wang \at
              Department of Mathematics, National Cheng Kung University, Tainan, Taiwan\\
              \email{sdhwwc@yahoo.com.tw} 
           \and
           Ruey-Lin Sheu \at
              Department of Mathematics, National Cheng Kung University, Tainan, Taiwan\\
              \email{rsheu@mail.ncku.edu.tw}
              \and
           Huu-Quang Nguyen \at
           Department of Mathematics, Vinh
           University, Vinh, Nghe An, Vietnam
}

\date{Received: date / Accepted: date}

\begin{document}
\maketitle

\begin{abstract}
Given two quadratic functions \( f(x) = x^T Ax + 2a^T x + a_0 \) and \( g(x) = x^T Bx + 2b^T x + b_0 ,\) each associated with either the strict inequality ($<0$); non-strict inequality ($\leq 0$); or the equality ($=0$), it is a fundamental question to ask whether or not the joint system has a solution. For homogeneous quadratic systems ($a=b=0,~a_0=b_0=0$), starting from Finsler's lemma in 1936 until Yuan's alternative lemma in 1990, all combinations of the unsolvability for $\{x\in \mathbb{R} ^n\mid x^T Ax \mathbin{\star } 0\}\cap \{x\in \mathbb{R} ^n\mid x^T Bx \mathbin{\#} 0\}\subset \{0\}$, where $\star $ and $\#$ can be any of $\{<,\leq ,=\}$, have been shown to possess either a positive definite or a positive semi-definite matrix pencil of $A$ and $B.$ Extensions to nonhomogeneous quadratic systems
$\{x\in \mathbb{R} ^n\mid f(x) \mathbin{\star } 0\}\cap \{x\in \mathbb{R} ^n\mid g(x) \mathbin{\#} 0\}=\emptyset $ have been done for several cases already. Two challenging cases remain open: the
nonhomogeneous Calabi Theorem which determines when $\{f(x)=0\}\cap \{g(x)=0\}=\emptyset $;
and the nonhomogeneous (strict) Finsler lemma to determine whether $\{f(x)\leq 0\}\cap \{g(x)=0\}=\emptyset .$ The paper provides the answers to both, in theorems and algorithms.

\end{abstract}

\section{Introduction}

Let $A$ and $B$ be two $n\times n$ symmetric matrices. In 1964, Eugenio Calabi \cite{calabi1964linear}
studied the pair of quadratic forms $x^T A x$ and $x^T B x$ over $\mathbb{R} ^n$ with their joint values $(x^T Ax,x^T Bx)$ in $\mathbb{R} ^2$. Such pairs naturally occur through the second fundamental form of classical differential geometry, 
and also appear in the study of certain matrix differential equations. Calabi's theorem states that
\begin{eqnarray*}
[{\rm Calabi}] &(n\geq 3)& \{x\in \mathbb{R} ^n\mid x^T Ax=0\}\cap \{x\in \mathbb{R} ^n\mid x^T Bx=0\}=\{0\} \\
  (1964) &\Longleftrightarrow & \exists (\lambda ,\mu )\neq (0,0) \text{ s.t. } \lambda A+ \mu B\succ 0.
\end{eqnarray*}
Namely, two homogeneous quadratic equations have no non-trivial solution other than $0$ if and only if two matrices $A$ and $B$ possess a positive definite linear combination. 

However, Uhlig \cite{Uhlig1979} and Hiriart-Urruty \cite{hiriart2007potpourri} noted that Calabi's theorem was a rediscovery of Finsler's 1936 lemma \cite{finsler1936vorkommen}. The statement of (strict) Finsler lemma is quoted as follows.
\begin{eqnarray*}
[{\rm Str.\text{-}Finsler}] && \{x\in \mathbb{R} ^n\mid x^T Ax\leq 0\}\cap \{x\in \mathbb{R} ^n\mid x^T Bx=0\}=\{0\} \\
  (1936) &\Longleftrightarrow & \exists \xi \in \mathbb{R} \text{ s.t. } A+\xi B\succ 0.
\end{eqnarray*}
The term ``strict'' is used because in \cite[1936]{finsler1936vorkommen}, the condition
\[
\{x\in \mathbb{R} ^n\mid x^T Ax\leq 0\}\cap \{x\in \mathbb{R} ^n\mid x^T Bx=0\}=\{0\}
\]
was posed as $x^T Bx=0\Rightarrow x^T Ax>0$ for every \( x\neq 0 \). Thus, ``strict'' refers to the requirement that $x^T Ax$ is strictly positive.

The (strict) Finsler lemma has a non-strict version, independently studied by many researchers but is generally attributed to Finsler. It states that
\begin{eqnarray*}
[{\rm NStr.\text{-}Finsler}] &(B:~{\rm indefinite})& \{x\in \mathbb{R} ^n\mid x^T Ax<0\}\cap \{x\in \mathbb{R} ^n\mid x^T Bx=0\}=\emptyset \\
  (1936) &\Longleftrightarrow & \exists \xi \in \mathbb{R} \text{ s.t. } A+\xi B\succeq 0.
\end{eqnarray*}
Similarly, the term ``non-strict'' is used because the condition 
\[
\{x\in \mathbb{R} ^n\mid x^T Ax<0\}\cap \{x\in \mathbb{R} ^n\mid x^T Bx=0\}=\emptyset 
\]
can be written as $x^T Bx=0\Rightarrow x^T Ax\geq 0$. Thus, ``non-strict'' indicates that $x^T Ax$ is required to be non-strictly positive.

There are two other famous alternative theorems of similar flavor: $\mathcal{S}$-lemma by Yakubovich in \cite[1973]{yakubovich1973minimization}\cite[1971]{yakubovich1971s}\cite[1977]{yakubovich1977s} and also Yuan's alternative theorem \cite[1990]{yuan1990subproblem}:
\begin{eqnarray*}
[\mathcal{S}{\rm\text{-}lemma}] &(\exists \bar{x},\bar{x}^T B\bar{x}<0)& \{x\in \mathbb{R} ^n\mid x^T Ax<0\}\cap \{x\in \mathbb{R} ^n\mid x^T Bx\leq 0\}=\emptyset \\
  (1973) &\Longleftrightarrow & \exists \xi \geq 0 \text{ s.t. } A+\xi B\succeq 0.
\end{eqnarray*}
\begin{eqnarray*}
[{\rm Yuan}] && \{x\in \mathbb{R} ^n\mid x^T Ax<0\}\cap \{x\in \mathbb{R} ^n\mid x^T Bx<0\}=\emptyset \\
  (1990) &\Longleftrightarrow & \exists \xi \in [0,1] \text{ s.t. } (1-\xi )A+\xi B\succeq 0.
\end{eqnarray*}
As the five alternative theorems [Calabi], [Str.-Finsler], [NStr.-Finsler], [$\mathcal{S}$-lemma] and [Yuan] are all direct consequences of the joint numerical range 
$\{(x^T Ax, x^T Bx)\mid x \in \mathbb{R} ^n\}$ being a convex set\footnote{It is known as Dines theorem.} \cite[1941]{dines1941mapping} and the separating hyperplane theorem, they are equivalent to each other. Please refer to \cite[2010]{Yan-Guo} for the proof of equivalence. 

The following summarizes all combinations of a system of two homogeneous quadratic (in)equalities, regarding when they have only the trivial solution \( 0 \) or no solution at all.

\begin{align*}
 [{\rm Yuan}]~~ \{x\in \mathbb{R} ^n\mid x^T Ax < 0\}\cap \{x\in \mathbb{R} ^n\mid x^T Bx < 0\} &= \emptyset \\
 [\mathcal{S}{\rm\text{-}lemma}]({\rm homo.~ version})~~ \{x\in \mathbb{R} ^n\mid x^T Ax < 0\}\cap \{x\in \mathbb{R} ^n\mid x^T Bx \leq 0\} &= \emptyset \\
  [{\rm NStr.\text{-}Finsler}]~~ \{x\in \mathbb{R} ^n\mid x^T Ax < 0\}\cap \{x\in \mathbb{R} ^n\mid x^T Bx = 0\} &= \emptyset \\
  [{\rm Non\text{-}Available}]~~ \{x\in \mathbb{R} ^n\mid x^T Ax \leq 0\}\cap \{x\in \mathbb{R} ^n\mid x^T Bx \leq 0\} &= \{0\} \\
 [{\rm Str.\text{-}Finsler}]~~ \{x\in \mathbb{R} ^n\mid x^T Ax \leq 0\}\cap \{x\in \mathbb{R} ^n\mid x^T Bx = 0\} &= \{0\} \\
[{\rm Calabi}]~~ \{x\in \mathbb{R} ^n\mid x^T Ax = 0\}\cap \{x\in \mathbb{R} ^n\mid x^T Bx = 0\} &= \{0\}
\end{align*}
Notice that [{\rm Non-Available}] can be obtained from [Str.\text{-}Finsler] and $[\mathcal{S}{\rm\text{-}lemma}]$, by
\begin{eqnarray*}
 && \{x\in \mathbb{R} ^n\mid x^T Ax \leq 0\}\cap \{x\in \mathbb{R} ^n\mid x^T Bx \leq 0\}=\{0\} \\
  &\Leftrightarrow & \underbrace {\{x\in \mathbb{R} ^n\mid x^T Ax \leq 0,x^T Bx = 0\}=\{0\}}_{\text{[{\rm Str.\text{-}Finsler}]}} \text{ and } \underbrace {\{x\in \mathbb{R} ^n\mid x^T Ax \leq 0,x^T Bx < 0\}=\emptyset }_{[\mathcal{S}{\rm\text{-}lemma}]}.
\end{eqnarray*}

\noindent $\bullet $ {\bf The nonhomogeneous extensions.}

All five alternative theorems have their own nonhomogeneous extensions, as summarized below, where \( f(x) = x^T Ax + 2a^T x + a_0 \) and \( g(x) = x^T Bx + 2b^T x + b_0 \) are real \( n \)-variable quadratic functions; \( A,B\in \mathcal S^n \) are symmetric matrices; vectors \( a,b\in \mathbb{R} ^n \); and \( a_0,b_0\in \mathbb{R} \). 
\begin{align*}
 [{\rm Yuan}]({\rm nonhomo.~ version})~~ \{x\in \mathbb{R} ^n\mid f(x) < 0\}\cap \{x\in \mathbb{R} ^n\mid g(x) < 0\} &= \emptyset \\
 [\mathcal{S}{\rm\text{-}lemma}]({\rm nonhomo.~ version})~~ \{x\in \mathbb{R} ^n\mid f(x) < 0\}\cap \{x\in \mathbb{R} ^n\mid g(x) \leq 0\} &= \emptyset \\
  [{\rm NStr.\text{-}Finsler}]({\rm nonhomo.~ version})~~ \{x\in \mathbb{R} ^n\mid f(x) < 0\}\cap \{x\in \mathbb{R} ^n\mid g(x) = 0\} &= \emptyset \\
  [{\rm Non\text{-}Available}]({\rm nonhomo.~ version})~~ \{x\in \mathbb{R} ^n\mid f(x) \leq 0\}\cap \{x\in \mathbb{R} ^n\mid g(x) \leq 0\} &= \emptyset \\
 [{\rm Str.\text{-}Finsler}]({\rm nonhomo.~ version})~~ \{x\in \mathbb{R} ^n\mid f(x) \leq 0\}\cap \{x\in \mathbb{R} ^n\mid g(x) = 0\} &= \emptyset \\
[{\rm Calabi}]({\rm nonhomo.~ version})~~ \{x\in \mathbb{R} ^n\mid f(x) = 0\}\cap \{x\in \mathbb{R} ^n\mid g(x) = 0\} &= \emptyset 
\end{align*}
Due to the lack of convexity for the joint numerical range\footnote{For a complete characterization about the convexity of the joint numerical range$\{(f(x),g(x))\mid x\in \mathbb{R} ^n\}\subset \mathbb{R} ^2$, please refer to 
\cite[2022]{nguyen2022convexity}.} $\{(f(x),g(x))\mid x\in \mathbb{R} ^n\}\subset \mathbb{R} ^2$, the above six statements are neither equivalent to each other, nor do they necessarily imply a non-negative (positive) Lagrangian function:
$$ (\exists (\lambda ,\mu )\neq (0,0))\quad \lambda f(x) + \mu g(x) \geq (>)\, 0,\quad \forall x \in \mathbb{R} ^n.$$
except for the two cases $[\mathcal{S}{\rm\text{-}lemma}]({\rm nonhomo.~ version})$ and $[{\rm Yuan}]({\rm nonhomo.~ version})$.

\begin{theorem}
($\mathcal{S}$-lemma,~Yakubovich\cite[1973]{yakubovich1973minimization}) Suppose that there is a $\bar{x}\in \mathbb{R} ^n$ such that $g(\bar x)<0$. Then, the following two statements are equivalent.\\
(i) $\{x\in \mathbb{R} ^n\mid f(x) < 0\}\cap \{x\in \mathbb{R} ^n\mid g(x) \leq 0\}=\emptyset ;$\\
(ii) $(\exists \xi \geq 0)\quad f(x) + \xi g(x) \geq 0,~\forall x \in \mathbb{R} ^n$.
\end{theorem}



\begin{theorem}
(nonhomogeneous Yuan's lemma,~Jeyakumar, Lee and Li in \cite[2009]{jeyakumar2009alternative}) The following two statements are equivalent.\\
(i) $\{x\in \mathbb{R} ^n\mid f(x) < 0\}\cap \{x\in \mathbb{R} ^n\mid g(x) < 0\}=\emptyset ;$\\
(ii) $\exists \lambda ,\mu \geq 0$ such that $(\lambda ,\mu )\neq (0,0)$ and $\lambda f(x) + \mu g(x) \geq 0$ for all $x \in \mathbb{R} ^n$.
\end{theorem}

As for [{\rm NStr.\text{-}Finsler}]({\rm nonhomo version}), two results are available in the literature. The initial contribution was made by Xia, Wang, and Sheu \cite[2016]{xia2016s}. It was later refined by Nguyen and Sheu \cite[2019]{nguyen2019geometric}, who introduced the new concept the separation property (see Sect. 2).

\begin{theorem}
([{\rm NStr.\text{-}Finsler}]({\rm nonhomo version});~$\mathcal{S}$-lemma with equality, Xia, Wang, and Sheu \cite[2016]{xia2016s})
Suppose $g(x)$ takes both positive and negative values. Then, $\{x\in \mathbb{R} ^n\mid f(x) < 0\}\cap \{x\in \mathbb{R} ^n\mid g(x) = 0\}=\emptyset $ fails to imply the existence of a $\xi \in \mathbb{R} $ such that $f(x)+\xi g(x)\geq 0$ for all $x\in \mathbb{R} ^n$ if and only if $A$ has exactly one negative eigenvalue, $B=0$, $b\neq 0$, and
\[
\begin{bmatrix}
V^T AV & V^T (Ax_0+a)\\
(x_0^T A+a^T )V & f(x_0)
\end{bmatrix}
\succeq 0,
\]
where $x_0=-\frac {b_0}{2b^T b} b$ and $V\in \mathbb{R} ^{n\times (n-1)}$ is the matrix basis of $\mathcal N(b^T )$.
\end{theorem}

\begin{theorem}
\label{Geometric S-lemma}
([{\rm NStr.\text{-}Finsler}]({\rm nonhomo version}); Nguyen and Sheu \cite[2019]{nguyen2019geometric})
Suppose $g(x)$ takes both positive and negative values. Then, $\{x\in \mathbb{R} ^n\mid f(x) < 0\}\cap \{x\in \mathbb{R} ^n\mid g(x) = 0\}=\emptyset $ fails to imply the existence of a $\xi \in \mathbb{R} $ such that $f(x)+\xi g(x)\geq 0$ for all $x\in \mathbb{R} ^n$ if and only if the lower level set $\{f<0\}$ is separated by the $0$-level set $\{g=0\}$.
\end{theorem}

The statement ``$\{f<0\}$ is separated by $\{g=0\}$'' in Theorem \ref{Geometric S-lemma} is called the ``separation property.'' It describes the situation in which the set
\[
\{f<0\} = \{f<0\}^- \cup \{f<0\}^+
\]
is disconnected, consisting of exactly two connected components \( \{f<0\}^- \) and \( \{f<0\}^+ \). Moreover, one of the two components lies entirely in \( \{g>0\} \), while the other lies entirely in \( \{g<0\} \). In this case, the joint numerical range $\{(f(x),g(x))\mid x\in \mathbb{R} ^n\}\subset \mathbb{R} ^2$ is ``almost'' non-convex\footnote{with only one exception: $\{f<0\}$ is separated by $\{g=0\}$ but $\{f=0\}\cap \{g=0\}\ne \emptyset .$} so that the well-known ``separating hyperplane theorem'' cannot be applied to yield a non-negative Lagrangian function.

The main contribution of this paper is to incorporate the separation property and establish the following necessary and sufficient conditions for [{\rm Calabi}]({\rm nonhomo.\ version}) as well as [{\rm Str.\text{-}Finsler}] ({\rm nonhomo.\ version}), thereby completing the last two pieces of the puzzle concerning the unsolvability of a system of two quadratic (in)equalities.

\begin{theorem}[nonhomogeneous Calabi Theorem]
\label{NHFC}
The statement
\begin{itemize}
\item[(C0)] \( \{f=0\}\cap \{g=0\}=\emptyset \)
\end{itemize}
holds if and only if exactly one of the following three mutually exclusive statements (C1), (C2), and (C3) holds.
\begin{itemize}
\item[(C1)] There exists \( (\lambda ,\mu )\neq (0,0) \) such that \( \lambda f(x)+\mu g(x)>0 \) for all \( x\in \mathbb{R} ^n \).
\item[(C2)]
\begin{enumerate}
    \item There exists \( (\lambda ,\mu )\neq (0,0) \), unique up to multiplication by a positive constant, such that the Lagrangian function \( \lambda f(x)+\mu g(x)\geq 0,~\forall x\in \mathbb{R} ^n. \) Moreover, \( \lambda f(x)+\mu g(x) \) attains its minimal value 0.
    \item On the minimum solution set of \( \lambda f(x)+\mu g(x) \), namely
\[
\mathcal W=\{\lambda f+\mu g=0\}=\{\lambda \nabla f+\mu \nabla g=0\},
\]
there is either \( (\mu f-\lambda g)|_{\mathcal W}>0 \) or \( (\mu f-\lambda g)|_{\mathcal W}<0 \).
\end{enumerate}
\item[(C3)] \( f \) and \( g \) have the separation property. (to be elaborated in Section 2)
\end{itemize}
\end{theorem}


\begin{theorem}[nonhomogeneous Strict Finsler Lemma]
\label{NHSF}
Suppose that \( g(x) \) takes both positive and negative values. Then,
\begin{itemize}
    \item[(F0)] \( \{f\leq 0\}\cap \{g=0\}=\emptyset \)
\end{itemize}
holds if and only if exactly one of the following three mutually exclusive statements (F1), (F2), and (F3) holds.
\begin{itemize}
\item[(F1)] There exists \( \xi \in \mathbb{R} \) such that \( f(x)+\xi g(x)>0 \) for all \( x\in \mathbb{R} ^n \).
\item[(F2)]
\begin{enumerate}
    \item There exists a unique \( \xi \in \mathbb{R} \) such that the Lagrangian function \( f(x)+\xi g(x)\geq 0 \) for all \( x\in \mathbb{R} ^n \). Also, the Lagrangian function \( f(x)+\xi g(x) \) attains its minimal value 0.
    \item On the minimum solution set of \( f(x)+\xi g(x) \), namely
\[
\mathcal V=\{f+\xi g=0\}=\{\nabla f+\xi \nabla g=0\},
\]
there is either \( g|_{\mathcal V}>0 \) or \( g|_{\mathcal V}<0 \).
\end{enumerate}
\item[(F3)] \( \{g=0\} \) separates \( \{f\leq 0\} \).
\end{itemize}
\end{theorem}

It is worth mentioning that there was an algorithmic approach to solving [{\rm Calabi}]({\rm nonhomo.\ version}) based on examining the optimal values and their attainability of four ``single-constraint quadratic optimization problems'', labeled (\( \mathcal P 1\))-(\( \mathcal P 4\)) below.

\begin{theorem}\label{thm:P0iff}
(Algorithmic approach for [{\rm Calabi}]({\rm nonhomo.~version}), Nguyen, Lin, Sheu, and Xia \cite[2023]{nguyen2023Po4})
Assume that $f(x)$ and $g(x)$ both satisfy two-sided Slater's condition and their quadratic coefficient matrices $\{A,B\}$ is linearly independent. Then,
\begin{equation*}
    \{x\in \mathbb{R} ^n\mid f(x) = 0\}\cap \{x\in \mathbb{R} ^n\mid g(x) = 0\}=\emptyset 
\end{equation*}
if and only if, for the following four (QP1QC)'s:
\begin{equation*}
\begin{array}{ll}
(\mathcal P 1) & \inf\{f(x)~|~g(x)\leq 0\}; \\
(\mathcal P 2) & \inf\{f(x)~|~-g(x)\leq 0\}; \\
(\mathcal P 3) & \inf\{-f(x)~|~g(x)\leq 0\}; \\
(\mathcal P 4) & \inf\{-f(x)~|~-g(x)\leq 0\},
\end{array}
\end{equation*}
exactly one of the following two statements (\romannumeral 1), (\romannumeral 2) holds.
\begin{itemize}
\item[(\romannumeral 1)] Exactly one of $(\mathcal P1),~(\mathcal P2),~(\mathcal P3)$ and $(\mathcal P4)$ has a strictly positive optimal value, while the optimal values of the other three are strictly negative.
\item[(\romannumeral 2)] Exactly one of $(\mathcal P1),~(\mathcal P2),~(\mathcal P3)$ and $(\mathcal P4)$ has an unattainable optimal value $0$, while the remaining three are unbounded from below.
\end{itemize}
\end{theorem}
Through numerical experiments in Section \ref{num exp}, we will see that the algorithm implementing Theorem \ref{NHFC} is computationally more efficient than the (QP1QC) approach proposed in Theorem \ref{thm:P0iff}.

To present our proposals and schemes for resolving the last two pieces of the puzzle: [{\rm Calabi}]({\rm nonhomo version}) and [{\rm Str.-Finsler}]({\rm nonhomo version}), we organize the paper as follows. Section 2 presents preliminary results on the separation property of two quadratic (sub)level sets. Section 3 establishes Theorem \ref{NHFC}, the nonhomogeneous version of [{\rm Calabi}], and Section 4 proves Theorem \ref{NHSF}, the nonhomogeneous counterpart of [{\rm Str.-Finsler}]. Section 5 provides illustrative examples. Section 6 develops matrix pencil conditions and their algorithms for both theorems with the computation algorithms of matrix pencils detailed in Section 7. Numerical results are discussed in Section 8.

Before ending the section, let us remark that both [Calabi] and [Str.-Finsler] are related to the study of ``strict'' inequalities, as opposed to the ``non-strict'' inequality, provided we write: $\{f=0\}\cap \{g=0\}=\emptyset $ as $g=0~\Rightarrow f\ne 0$; $\{f\leq 0\}\cap \{g=0\}=\emptyset $ as $g=0~\Rightarrow f>0$. Important applications include: to answer whether or not the optimal value $\alpha =\inf \{f(x) \mid g=0\}$ is unattainable, namely, \( \{f-\alpha =0\}\cap \{g=0\}=\emptyset \); and to solve quadratic fractional programming $\beta =\inf \frac{f(x)}{g(x)}$ with a natural constraint $g>0$ (or $g\neq 0$).

\noindent $\bullet $ {\bf Notations.}

Throughout this paper, we use the following notation. Let \( \mathcal S^n \) denote the set of all real symmetric \( n\times n \) matrices. For a function \( h(x) \), we write the level set 
\[
\{x\in \mathbb{R} ^n \mid h(x)  \mathbin{\star } 0\}
\]
in the short-hand form \( \{h \mathbin{\star } 0\} \), where \( \star \in \{\leq ,<, =\} \). For a matrix \( Q \), we denote by \( \mathcal{R} (Q) \), \( \mathcal{N}(Q) \), and \( Q^\dagger \) its range space, null space, and the Moore--Penrose inverse of \( Q \), respectively. 

For a symmetric matrix \( P \) we write \( P\preceq 0 \), \( P\prec 0 \), \( P\succeq 0 \), and \( P\succ 0 \) to indicate that \( P \) is negative semi-definite, negative definite, positive semi-definite, and positive definite, respectively.

For a set \( \mathcal T\subset \mathbb{R} ^n \), we denote its closure by \( \operatorname{cl}(\mathcal T) \), its convex hull by \( \operatorname{conv}(\mathcal T) \), and its linear span set by \( \operatorname{span}(\mathcal T) \). For a mathematical programming (P), we denote its optimal value by \( v(\text{P}) \).

Finally, the joint numerical range of \( f(x) \) and \( g(x) \), denoted by
\[
(f,g)(\mathbb{R} ^n)=\{(f(x),g(x)) \mid x\in \mathbb{R} ^n\}\subset \mathbb{R} ^2,
\]
is defined to be the image of the joint function \( (f,g):\mathbb{R} ^n\rightarrow \mathbb{R} ^2 \).

\section{Preliminaries: the separation property of two quadratic (sub)level-sets}

The core technical tool of the paper is the separation of quadratic (sub)level-sets. We adopt the following formal definition.
\begin{definition}[Separation of quadratic (sub)level-sets](Nguyen and Sheu \cite[2019]{nguyen2019geometric})
\label{def sep}
The surface \( \{g = 0\} \) is said to separate the set \( \{f \mathbin{\star }0\} \), where \( \star \in \{<,\leq ,=\} \), if there are non-empty subsets \( L^- \) and \( L^+ \) of \( \{f \mathbin{\star }0\} \) such that
\begin{align*}
& L^- \cup L^+ =\{f \mathbin{\star } 0\};\\
& g(u)<0<g(v), \;\forall \, u \in L^- ; \;\forall \, v \in L^+.
\end{align*}
\end{definition}
Alternatively, Lemma \ref{eq sep} provides a more useful, yet equivalent, formulation of the original Definition \ref{def sep}, which will facilitate our analysis later on.

\begin{lemma}
\label{eq sep}
(Lemma 7 in \cite{nguyen2024separating})\\
The 0-level set \( \{g=0\} \) separates \( \{f \mathbin{\star } 0\} \) if and only if
\begin{align}
& \{f\mathbin{\star } 0\} \cap \{g=0\}=\emptyset \text{ and} \label{f=g=0}\\
& g(u)<0<g(v) \text{ for some } u,v\in \{f \mathbin{\star } 0\},
\end{align}
where \( \star \in \{<,\leq ,=\} \).
\end{lemma}

Note that, by the continuity of $g$ and the Intermediate Value Theorem, if \( \{f\mathbin{\star } 0\} \) is separated by \( \{g=0\} \), then \( \{f\mathbin{\star }0\} \) must be disconnected. In this case, it was shown in \cite[2019]{nguyen2019geometric} that the set \( \{f\mathbin{\star } 0\} \) consists of exactly two connected components, denoted by $L^-$ and $L^+$ in the definition.

Several important results on the separation of two quadratic (sub)level sets have been developed from Definition \ref{def sep} and Lemma \ref{eq sep}. For example, please refer to Nguyen, Chu and Sheu \cite[2022]{nguyen2022convexity} and \cite[2024]{nguyen2024separating}. In particular, let us mention the following result.
\begin{itemize}
\item[$\bullet $][Theorem 3.1 in Nguyen, Chu and Sheu \cite[2022]{nguyen2022convexity}] The joint numerical range $(f,g)(\mathbb{R} ^n)$ is non-convex if and only if there exist $\alpha ,~\beta \in \mathbb{R} $ such that either $\{g=\beta \}$ separates $\{f=\alpha \}$ or $\{f=\alpha \}$ separates $\{g=\beta \}$.
\end{itemize}

In this paper, since we characterize [Calabi] (nonhomo.\ version) and [Str.-Finsler] (nonhomo. version) through the lens of the separation of two quadratic (sub)level sets, related properties on the separation of such sets must first be developed.

\begin{definition}[The separation property]
\label{def sep prop}
Two quadratic functions \( f \) and \( g \) are said to have the separation property if either \( \{f=0\} \) separates \( \{g=0\} \), or \( \{g=0\} \) separates \( \{f=0\} \), or possibly both.
\end{definition}

\begin{lemma}
\label{h!ap}
If \( \{g=0\} \) separates \( \{f\mathbin{\star } 0\} \), where \( \star \in \{<,\leq ,=\} \), then
\begin{enumerate}
    \item \( (s,0)\notin (f,g)(\mathbb{R} ^n) \) for all \( s\mathbin{\star }0 \);\label{for all s}
    \item \( (r,0)\in \operatorname{conv}((f,g)(\mathbb{R} ^n)) \) for some \( r\mathbin{\star }0 \).\label{for some r}
\end{enumerate}
In particular, we have
\begin{itemize}
  \item[(i)] (for the case that ``$\star $'' is ``$=$'') if \( \{g=0\} \) separates \( \{f = 0\} \), then \( (0,0)\in \operatorname{conv}((f,g)(\mathbb{R} ^n))\setminus (f,g)(\mathbb{R} ^n) \). Due to the symmetry, if, reversely, \( \{f=0\} \) separates \( \{g = 0\} \), there also is \( (0,0)\in \operatorname{conv}((f,g)(\mathbb{R} ^n))\setminus (f,g)(\mathbb{R} ^n) \).
  \item[(ii)] (for the case that ``$\star $'' is ``$\leq $'') if \( \{g=0\} \) separates \( \{f \leq 0\} \), then \( ((-\infty ,0]\times \{0\}) \cap (f,g)(\mathbb{R} ^n) = \emptyset \) but \( ((-\infty ,0]\times \{0\}) \cap \operatorname{conv}((f,g)(\mathbb{R} ^n)) \neq \emptyset \).
  \item[(iii)] (for the case that ``$\star $'' is ``$<$'') If \( \{g=0\} \) separates \( \{f < 0\} \), then \( ((-\infty ,0)\times \{0\}) \cap (f,g)(\mathbb{R} ^n) = \emptyset \) but \( ((-\infty ,0)\times \{0\}) \cap \operatorname{conv}((f,g)(\mathbb{R} ^n)) \neq \emptyset \).
\end{itemize}
\end{lemma}

\begin{proof}
For \ref{for all s}, it is just a rephrase of \eqref{f=g=0}. Since \( \{f\mathbin{\star } 0\} \cap \{g=0\}=\emptyset , \) \( g(x)=0 \) implies that \( f(x)\mathbin{\star } 0 \) fails. Therefore, $(\forall s\mathbin{\star }0)$ \( (s,0)\notin (f,g)(\mathbb{R} ^n) \).

For \ref{for some r}, let \( u,v\in \mathbb{R} ^n \) be such that \( f(u),f(v)\mathbin{\star } 0 \) with \( g(u)<0<g(v) \). Let us write the value 0 as a convex combination of \( g(u) \) and \( g(v) \), and denote $(r,0)$ as
\[
(r,0)\triangleq \frac{g(v)}{g(v)-g(u)}(f(u),g(u))+\frac{-g(u)}{g(v)-g(u)}(f(v),g(v))\in \operatorname{conv}((f,g)(\mathbb{R} ^n)).
\]
Since \( f(u),f(v)\mathbin{\star } 0\), it follows that \( r\mathbin{\star } 0 \) and Statement \ref{for some r} is proved.
\end{proof}

In the following, we quote some preliminary results to be used later.

\begin{lemma}[Proposition 2 in \cite{nguyen2022convexity}]
\label{cor 2.9 ng}
Suppose that the level set \( \{\sigma f + \tau g = 0\} \) separates the level set \( \{\eta f + \theta g = 0\} \) for some real numbers \( \eta , \theta , \sigma , \tau \) with \( \eta \tau - \theta \sigma \neq 0 \). Then, exactly one of the following two statements holds.
\begin{enumerate}
    \item if both \( f \) and \( g \) are quadratic functions, \( \{f = 0\} \) and \( \{g = 0\} \) mutually separate each other;
    \item if one of \( f \) and \( g \) is affine, then the other must be a quadratic function and the 0-level set of the affine function separates the 0-level set of the quadratic function.
\end{enumerate}
\end{lemma}

It has the following immediate consequence.

\begin{lemma}
\label{linear comb sep}
Two quadratic functions \( f \) and \( g \) have the separation property if and only if \( \{\sigma f + \tau g = 0\} \) separates \( \{\eta f + \theta g = 0\} \) for some real numbers \( \eta , \theta , \sigma , \tau \) with \( \eta \tau - \theta \sigma \neq 0 \).
\end{lemma}

\begin{proof}
The sufficient part holds by Lemma \ref{cor 2.9 ng}. For the necessary part, if \( \{f=0\} \) separates \( \{g=0\} \), let us take \( \eta =\tau =0 \) and \( \theta =\sigma =1 \). If \( \{g=0\} \) separates \( \{f=0\} \), we take \( \eta =\tau =1 \) and \( \theta =\sigma =0 \).
\end{proof}

\begin{lemma}
\label{Tsu prop 1}
(Theorem 1 in \cite{nguyen2019geometric}, Theorem 2 and Lemma 6 in \cite{nguyen2024separating})\\
If \( \{g=0\} \) separates \( \{f\mathbin{\star } 0\} \), where \( \star \in \{<,\leq \} \), then \( g \) must be a non-constant affine function, while \( f \) must be quadratic.
\end{lemma}

We will also repeatedly use the following convex-geometric result.

\begin{lemma}[Fenchel--Eggleston theorem (Theorem 3.5) in \cite{danzer1963helly}]
\label{fenchel-eggleston}
Let \( \mathcal T\subseteq \mathbb R^m \) be the union of at most \( m \) connected sets. Then every point of \( \operatorname{conv}(\mathcal T) \) can be expressed as a convex combination of at most \( m \) points of \( \mathcal T \). In particular, since \( (f,g):\mathbb R^n\rightarrow \mathbb R^2 \) is continuous and \( \mathbb R^n \) is connected, every point of \( \operatorname{conv}((f,g)(\mathbb R^n)) \) is a convex combination of at most two points of \( (f,g)(\mathbb R^n) \).
\end{lemma}

\section{Proof for the nonhomogeneous Calabi theorem, Theorem \ref{NHFC}}
\label{pf:NHFC}

Recall that the notation \( \operatorname{conv}((f,g)(\mathbb{R} ^n)) \) stands for the convex hull of $(f,g)(\mathbb{R} ^n),$ whereas $\operatorname{cl}(\operatorname{conv}((f,g)(\mathbb{R} ^n)))$ is the closure of the convex hull of the joint numerical range of $f$ and $g.$ Obviously, there is 
\[
\underbrace {\operatorname{cl}(\operatorname{conv}((f,g)(\mathbb{R} ^n)))}_{\text{convex \& closed}}\supset \underbrace {\operatorname{conv}((f,g)(\mathbb{R} ^n))}_{\text{convex}}\supset (f,g)(\mathbb{R} ^n).
\]

To characterize the nonhomogeneous Calabi Theorem (Theorem \ref{NHFC}) that \( \{f=0\}\cap \{g=0\}=\emptyset \), namely, $(0,0)\notin (f,g)(\mathbb{R} ^n)$, our program is to divide
$(0,0)\notin (f,g)(\mathbb{R} ^n)$ into the following three exclusive cases.
\begin{itemize}
  \item[$\dagger $] (C1) in Theorem \ref{NHFC}: \( (0,0)\notin \operatorname{cl}(\operatorname{conv}((f,g)(\mathbb{R} ^n))) \). See Proposition \ref{0-t_i-gu_a} below.
  \item[$\dagger $] (C2) in Theorem \ref{NHFC}: $(0,0)\in \operatorname{cl}(\operatorname{conv}((f,g)(\mathbb{R} ^n)))\setminus \operatorname{conv}((f,g)(\mathbb{R} ^n))$. See Proposition \ref{iff boundary} below.
  \item[$\dagger $] (C3) in Theorem \ref{NHFC}: $(0,0)\in \operatorname{conv}((f,g)(\mathbb{R} ^n))\setminus (f,g)(\mathbb{R} ^n)$. See Proposition \ref{C2 prop} below.
\end{itemize}

In the examples in Section 5, we will show that (C1) and (C2) typically arise when the joint numerical range \( (f,g)(\mathbb{R} ^n) \) is convex, whereas (C3) can occur only when \( (f,g)(\mathbb{R} ^n) \) is nonconvex. Based on our experience, the nonconvex case is actually easier to characterize, although not necessarily easier to prove. For this reason we treat (C3) first and then prove (C1) and (C2).

\subsection{Proof of Theorem \ref{NHFC} (C3): $(0,0)\in \operatorname{conv}((f,g)(\mathbb{R} ^n))\setminus (f,g)(\mathbb{R} ^n)$}

According to Lemma \ref{h!ap} (i), if either \( \{g=0\} \) separates \( \{f=0\} \) or \( \{f=0\} \) separates \( \{g=0\} \), then \( (0,0)\in \operatorname{conv}((f,g)(\mathbb{R} ^n))\setminus (f,g)(\mathbb{R} ^n) \). Proposition \ref{C2 prop} below asserts that the converse is also true.

\begin{proposition}[Theorem \ref{NHFC} (C3)]
\label{C2 prop}
Two quadratic functions \( f \) and \( g \) have the separation property if and only if \( (0,0) \in \operatorname{conv}((f,g)(\mathbb{R} ^n)) \setminus (f,g)(\mathbb{R} ^n) \).
\end{proposition}

\begin{proof}
We only have to prove the sufficient part. Assume that \( (0,0)\in \operatorname{conv}((f,g)(\mathbb R^n)) \) and \( (0,0)\notin (f,g)(\mathbb R^n) \). By Fenchel--Eggleston theorem (Lemma \ref{fenchel-eggleston}), \( (0,0) \) is a convex combination of \( (\sigma ,\tau ),(-t\sigma ,-t\tau )\in (f,g)(\mathbb{R} ^n) \) for some \( (\sigma ,\tau )\neq (0,0) \) and \( t>0 \). Moreover, there exist \( u,v \in \mathbb{R} ^n \) such that 
\begin{equation}\label{tau-}
(f(u),g(u))=(\sigma ,\tau )~ {\rm and}~(f(v),g(v))=(-t\sigma ,-t\tau ).
\end{equation}
 
Since \( (\sigma ,\tau )\neq (0,0) \), 
\[
\det
\begin{bmatrix}
\sigma & \tau \\
\tau & -\sigma 
\end{bmatrix}
=-(\sigma ^2+\tau ^2)\neq 0.
\]
Then, the following system of simultaneous linear equations in variables $(r,s)$
\begin{equation*}
    \begin{cases}
      \sigma r + \tau s=0\\
      \tau r - \sigma s = 0
    \end{cases}
\end{equation*}
has a unique solution $(r,s)=(0,0).$ Due to \( (0,0) \notin (f,g)(\mathbb{R} ^n) \), there exists no \( x \in \mathbb{R} ^n \) such that \( f(x)=g(x)=0 \). Therefore,
\begin{equation}
\label{phian sep}
\{\sigma f + \tau g = 0\} \cap \{\tau f - \sigma g = 0\}=\emptyset .
\end{equation}

By \eqref{tau-}, it is easy to check that
\begin{equation}
\label{mix sep}
u,v\in \{\tau f-\sigma g=0\} \text {and} (\sigma f + \tau g)(u)\cdot (\sigma f + \tau g)(v) = -t(\sigma ^2+\tau ^2)^2<0.
\end{equation}
According to Lemma \ref{eq sep}, \eqref{phian sep} and \eqref{mix sep} together imply that \( \{\sigma f+\tau g=0\} \) separates \( \{\tau f-\sigma g=0\} \). Finally, by Lemma \ref{linear comb sep}, we conclude that the two quadratic functions \( f \) and \( g \) have the separation property. The proof is complete.
\end{proof}

\subsection{Proof of Theorem \ref{NHFC} (C1): \( (0,0)\notin \operatorname{cl}(\operatorname{conv}((f,g)(\mathbb{R} ^n))) \)}

We are going to show that, in this case, there exists a positive Lagrangian function \( \lambda f(x)+\mu g(x)>0 \) for all \( x\in \mathbb{R} ^n \) with \( (\lambda ,\mu )\neq (0,0) \). Let us first recall a lemma.

\begin{lemma}
\label{QP0}
(\cite{boyd2004convex} appendix A.5.4 and \cite{gallier2010schur})
Let \( h(x) =x^T Cx + 2c^T x +c_0 \) be an \( n \)-variable real quadratic function. Then, \( h(x) \) is bounded below if and only if \( C\succeq 0 \) and \( c\in \mathcal{R} (C) \). Moreover, when \( h(x) \) is bounded below, it always attains its minimum at some $\bar{x}\in \mathbb{R} ^n$ such that
\[
h(\bar{x})=\min_{x\in \mathbb{R} ^n} h(x)=c_0-c^T C^\dagger c.
\]
\end{lemma}

\begin{proposition}[Theorem \ref{NHFC} (C1)]
\label{0-t_i-gu_a}
Let \( f,g \) be two real quadratic functions. The following statements are equivalent.
\begin{enumerate}
    \item[(a)] There exists \( (\lambda ,\mu ) \neq (0,0) \) such that \( \lambda f(x)+\mu g(x)>0 \) for all \( x\in \mathbb{R} ^n \).
    \item[(b)] There exists \( (\lambda ,\mu ) \neq (0,0) \) and \( \nu >0 \) such that \( \lambda f(x)+\mu g(x) \geq \nu \) for all \( x\in \mathbb{R} ^n \).
    \item[(c)] \( (0,0)\notin \operatorname{cl}(\operatorname{conv}((f,g)(\mathbb{R} ^n))) \). That is, \( \{(0,0)\}\cap \operatorname{cl}(\operatorname{conv}((f,g)(\mathbb{R} ^n)))=\emptyset \).
\end{enumerate}
\end{proposition}

\begin{proof}
\noindent\textbf{(a) \(\Rightarrow \) (b).} Suppose $(a)$ holds with \( \lambda f(x)+\mu g(x)>0 \) for all \( x\in \mathbb{R} ^n \). It implies that the Lagrangian function \( \lambda f(x)+\mu g(x)\) is quadratic and bounded below by $0.$
By Lemma \ref{QP0}, there exists some \( \bar x\in \mathbb{R} ^n \) such that
\[
\nu :=\lambda f(\bar x)+\mu g(\bar x)=\inf_{x\in \mathbb{R} ^n} \lambda f(x)+\mu g(x)>0,
\]
which is exactly $(b)$.

\noindent\textbf{(b) \(\Rightarrow \) (c).} Suppose $(b)$ holds, that is,
\[
(f,g)(\mathbb{R} ^n) \subset \{(r,s)\in \mathbb{R} ^2 \mid \lambda r+\mu s \geq \nu >0 \},
\]
Since \( \{(r,s)\in \mathbb{R} ^2 \mid \lambda r+\mu s\geq \nu \} \) is closed and convex, and the set does not contain $(0,0)$ since $\nu >0,$ it follows that
\[
\operatorname{cl}(\operatorname{conv}((f,g)(\mathbb{R} ^n))) \subset \{(r,s)\in \mathbb{R} ^2 \mid \lambda r+\mu s \geq \nu \}.
\]
and \( \{(0,0)\}\cap \operatorname{cl}(\operatorname{conv}((f,g)(\mathbb{R} ^n)))=\emptyset \).

\noindent\textbf{(c) \(\Rightarrow \) (a).} Suppose $(c)$ holds. Then, the two closed convex sets \( \{(0,0)\} \) and \( \operatorname{cl}(\operatorname{conv}((f,g)(\mathbb{R} ^n))) \) are disjoint. As \( \{(0,0)\} \) is compact, the two can be strictly separated (\cite{boyd2004convex} Subsection 2.5.1) by a hyperplane. There exists \( (\lambda ,\mu )\neq (0,0) \) and \( \nu \in \mathbb{R} \) such that
\begin{align}
\operatorname{cl}(\operatorname{conv}((f,g)(\mathbb{R} ^n))) & \subset \{(r,s)\in \mathbb{R} ^2 \mid \lambda r+\mu s > \nu \};\label{0 ext pt up}\\
\{(0,0)\} & \subset \{(r,s)\in \mathbb{R} ^2 \mid \lambda r+\mu s < \nu \}.\label{0 ext pt}
\end{align}
From \eqref{0 ext pt}, \( \nu >0 \). As a result, \eqref{0 ext pt up} implies that
\[
(f,g)(\mathbb{R} ^n) \subset \{(r,s)\in \mathbb{R} ^2 \mid \lambda r+\mu s > \nu >0\},
\]
which is $(a)$ of this proposition.
\end{proof}

\subsection{Proof of Theorem \ref{NHFC} (C2): \( (0,0)\in \operatorname{cl}(\operatorname{conv}((f,g)(\mathbb{R} ^n)))\setminus \operatorname{conv}((f,g)(\mathbb{R} ^n)) \)}

This is the most interesting but subtle case in Theorem \ref{NHFC}. When it happens the Lagrangian function \( \lambda f(x)+\mu g(x)\geq 0 \) attains the lower bound value $0$. On the minimum solution set, \( \mathcal W=\{x\in \mathbb{R} ^n\mid \lambda f(x)+\mu g(x)=0\}\), there exists a ``conjugate'' Lagrangian function which is either strictly positive or strictly negative.

We first need Lemma \ref{siang=>giam} and Lemma \ref{nonclosed} to establish the Lagrange multiplier \( (\lambda ,\mu )\ne (0,0) \) such that
\(
\lambda f(x)+\mu g(x) \geq 0, \;\forall \, x\in \mathbb{R} ^n.
\)

\begin{lemma}\label{siang=>giam}
Suppose that \( (0,0)\notin (f,g)(\mathbb{R} ^n) \) and the two row vectors \( (\lambda ,\mu ), (\lambda ^\prime,\mu ^\prime)\in \mathbb{R} ^2 \) are linearly independent. If there is
\begin{equation}\label{siang giap}
\lambda f(x)+\mu g(x)\geq 0~{\rm and}~\lambda ^\prime f(x)+\mu ^\prime g(x)\geq 0, \; \;\forall \, x\in \mathbb{R} ^n,
\end{equation}
then the sum of the two inequalities in \eqref{siang giap} becomes strictly positive everywhere:
\[
(\lambda +\lambda ^\prime)f(x)+(\mu +\mu ^\prime)g(x)>0, \;\forall \, x\in \mathbb{R} ^n.
\]
\end{lemma}

\begin{proof}
By \( (0,0)\notin (f,g)(\mathbb{R} ^n) \), we have \( (f(x),g(x)) \neq (0,0) \) for all \( x \in \mathbb{R} ^n \). Since the two row vectors \( (\lambda ,\mu ) \) and \( (\lambda ^\prime,\mu ^\prime) \) are linearly independent,
\[
\begin{bmatrix}
\lambda f(x)+\mu g(x) \\ \lambda ^\prime f(x) + \mu ^\prime g(x)
\end{bmatrix}
=
\begin{bmatrix}
\lambda & \mu \\ \lambda ^\prime & \mu ^\prime
\end{bmatrix}
\begin{bmatrix}
f(x) \\ g(x)
\end{bmatrix}
\neq 
\begin{bmatrix}
0 \\ 0
\end{bmatrix},
\; \;\forall \, x\in \mathbb{R} ^n.
\]
Thus, for all \(x\in \mathbb{R} ^n\), at least one of \( \lambda f(x)+\mu g(x) \) and \( \lambda ^\prime f(x)+\mu ^\prime g(x) \) is nonzero. Together with \eqref{siang giap}, we get
\[
(\lambda f(x)+\mu g(x)) +(\lambda ^\prime f(x)+\mu ^\prime g(x))>0, \; \;\forall \, x\in \mathbb{R} ^n.
\]
\end{proof}

\begin{lemma}\label{nonclosed}
If \( (0,0)\in \operatorname{cl}(\operatorname{conv}((f,g)(\mathbb{R} ^n)))\setminus \operatorname{conv}((f,g)(\mathbb{R} ^n)) \), then there exists \( (\lambda ,\mu )\neq (0,0) \) such that
\begin{equation}
\lambda f(x)+\mu g(x)\geq 0,\; \;\forall \, x\in \mathbb{R} ^n.\label{exist lambda}
\end{equation}
Moreover, the existence of \( (\lambda ,\mu ) \) is unique up to multiplication by a positive constant. That is, if \( (\lambda ^\prime,\mu ^\prime)\neq (0,0) \) also satisfies \eqref{exist lambda}, then \( (\lambda ^\prime,\mu ^\prime)=(t\lambda ,t\mu ) \) for some \( t>0 \).
\end{lemma}
\begin{proof}
Since \( \{(0,0)\} \) and \( \operatorname{conv}((f,g)(\mathbb{R} ^n)) \) are both convex and disjoint, by the separating hyperplane theorem, there exists \( (\lambda ,\mu )\neq (0,0) \) and \( \nu \in \mathbb{R} \) such that
\begin{align}
\operatorname{conv}((f,g)(\mathbb{R} ^n)) &\subset \{(r,s)\in \mathbb{R} ^2 \mid \lambda r+\mu s \geq \nu \};\label{exist lambda mu}\\
\{(0,0)\} &\subset \{(r,s)\in \mathbb{R} ^2 \mid \lambda r+\mu s \leq \nu \}.\label{0<mu}
\end{align}
By \eqref{0<mu}, \( \nu \geq 0 \), which, together with \eqref{exist lambda mu}, gives
\[
(f,g)(\mathbb{R} ^n)\subset \operatorname{conv}((f,g)(\mathbb{R} ^n))\subset \{(r,s)\in \mathbb{R} ^2 \mid \lambda r+\mu s \geq \nu \geq 0\}.
\]
It is exactly \eqref{exist lambda}.

For uniqueness, suppose that \( (\lambda ^\prime,\mu ^\prime)\neq (0,0) \) also satisfies \eqref{exist lambda} but there is no \( t\in \mathbb{R} \) such that \( (\lambda ^\prime,\mu ^\prime)=(t\lambda ,t\mu ) \). By Lemma \ref{siang=>giam},
\[
(f,g)(\mathbb{R} ^n) \subset \{(r,s)\in \mathbb{R} ^2 \mid (\lambda +\lambda ^\prime)r+(\mu +\mu ^\prime) s> 0\},
\]
which also implies that $(\lambda +\lambda ^\prime)$ and $(\mu +\mu ^\prime)$ cannot be both zero. By the equivalence of Proposition \ref{0-t_i-gu_a} $(a)$ and $(c)$, we know that \( (0,0)\notin \operatorname{cl}(\operatorname{conv}((f,g)(\mathbb{R} ^n))) \), a contradiction to the assumption of this lemma. So, $(\lambda ,\mu )$ and \( (\lambda ^\prime,\mu ^\prime)\) must be linearly dependent.

Secondly, suppose that, for some $t>0,$ \( (\lambda ^\prime,\mu ^\prime)=(-t\lambda ,-t\mu )\neq (0,0) \) also satisfies 
$$\lambda ^\prime f(x)+\mu ^\prime g(x)=-t\lambda f(x)-t\mu g(x)\geq 0,\; \;\forall \, x\in \mathbb{R} ^n.$$
It implies that $\lambda f(x)+\mu g(x)\leq 0,~\forall x\in \mathbb{R} ^n,$ and forces that \( \lambda f(x)+\mu g(x)=0 \) for all \( x\in \mathbb{R} ^n \). We claim that \( \mu f(x)-\lambda g(x)\neq 0 \) for all \( x\in \mathbb{R} ^n \).

Otherwise, let \( \bar x\in \mathbb{R} ^n \) be such that \( \lambda f(\bar x)+\mu g(\bar x)=\mu f(\bar x)-\lambda g(\bar x)=0 \). Then,
\[
\begin{bmatrix}
\lambda & \mu \\
\mu & -\lambda 
\end{bmatrix}
\begin{bmatrix}
f(\bar x)\\
g(\bar x)
\end{bmatrix}
=
\begin{bmatrix}
0 \\ 0
\end{bmatrix}.
\]
By \( \det\begin{bmatrix}\lambda & \mu \\\mu & -\lambda \end{bmatrix}=-\lambda ^2-\mu ^2<0 \), it follows that \( f(\bar x)=g(\bar x)=0 \) and \( (0,0)\in (f,g)(\mathbb{R} ^n) \), which contradicts the assumption that \( (0,0)\notin \operatorname{conv}((f,g)(\mathbb{R} ^n)) \).

Now, we know that \( \mu f(x)-\lambda g(x)\neq 0,~\forall x\in \mathbb{R} ^n \), provided that \( (\lambda ',\mu ')=(-t\lambda ,-t\mu )\neq (0,0) \) for some \( t>0 \) satisfying \eqref{exist lambda}. Since \( \mu f(x)-\lambda g(x) \) is a continuous function of \( x \), by Intermediate Value Theorem, \( \mu f(x)-\lambda g(x) \) must be either strictly positive or strictly negative. In either case, Proposition \ref{0-t_i-gu_a} $(a)$ leads to the fact that \( (0,0)\notin \operatorname{cl}(\operatorname{conv}((f,g)(\mathbb{R} ^n))) \), which again contradicts the assumption of this lemma that \( (0,0)\in \operatorname{cl}(\operatorname{conv}((f,g)(\mathbb{R} ^n)))\setminus \operatorname{conv}((f,g)(\mathbb{R} ^n)) \). The lemma is thus proved.
\end{proof}

Finally, Theorem \ref{NHFC} (C2) is stated and proved as follows.

\begin{proposition}[Theorem \ref{NHFC} (C2)]\label{iff boundary}
\( (0,0)\in \operatorname{cl}(\operatorname{conv}((f,g)(\mathbb{R} ^n)))\setminus \operatorname{conv}((f,g)(\mathbb{R} ^n)) \) if and only if the following two statements hold.
\begin{enumerate}
    \item There exists \( (\lambda ,\mu )\neq (0,0) \), unique up to multiplication by a positive constant, such that the Lagrangian function \( \lambda f(x)+\mu g(x)\geq 0,~\forall x\in \mathbb{R} ^n. \) Moreover, \( \lambda f(x)+\mu g(x) \) attains its minimal value 0.\label{gu^an 2}
    \item On the minimum solution set of \( \lambda f(x)+\mu g(x) \),
\[
\mathcal W=\{\lambda f+\mu g=0\}=\{\lambda \nabla f+\mu \nabla g=0\},
\]
there is either \( (\mu f-\lambda g)|_{\mathcal W}>0 \) or \( (\mu f-\lambda g)|_{\mathcal W}<0 \).\label{gu^an 3}
\end{enumerate}
\end{proposition}

\begin{proof}
\textbf{(Necessity)} Suppose that \( (0,0)\in \operatorname{cl}(\operatorname{conv}((f,g)(\mathbb{R} ^n)))\setminus \operatorname{conv}((f,g)(\mathbb{R} ^n)) \). By Lemma \ref{nonclosed}, the Lagrangian function \( \lambda f(x)+\mu g(x) \) must be non-negative. However, from the equivalence of statements~(c) and~(a) in
Proposition~\ref{0-t_i-gu_a}, the assumption that \( (0,0)\in \operatorname{cl}(\operatorname{conv}((f,g)(\mathbb{R} ^n))) \) implies that the Lagrangian function \( \lambda f(x)+\mu g(x) \) cannot be strictly positive. There exists at least one \( \bar x\in \mathbb{R} ^n \) such that \( \lambda f(\bar x)+\mu g(\bar x)=0 \). The point \( \bar x \) is a minimum solution of \( \lambda f+\mu g \) on \( \mathbb{R} ^n \).

For the statement $\ref{gu^an 3}$, Suppose, to the contrary, that there would exist another minimizer \( \hat x\in \mathcal W \) of the Lagrangian function $\lambda f(x)+\mu g(x)$ such that \( -\mu f(\hat x)+\lambda g(\hat x)=0 \). Then,
\[
\left\{
\begin{aligned}
 \lambda f(\hat x)+\mu g(\hat x) &=0;\\
-\mu f(\hat x)+\lambda g(\hat x) &=0.
\end{aligned}
\right .
\]
Since \( \lambda ^2+\mu ^2\neq 0 \), it implies that
\( f(\hat x)=g(\hat x)=0 \), which contradicts the assumption that \( (0,0)\notin \operatorname{conv}((f,g)(\mathbb{R} ^n)) \). The image of the restriction of \( -\mu f+\lambda g \) to \( \mathcal W=\{x \mid \lambda f(x)+\mu g(x)=0\} \) does not contain $0.$

However, the minimum solution set $\mathcal W$ of the Lagrangian function \( \lambda f(x)+\mu g(x) \) is affine, and thus is connected. So is the image set \( (-\mu f+\lambda g)(\mathcal W) \). As \( 0\notin (-\mu f+\lambda g)(\mathcal W) \), the values of \(-\mu f+\lambda g\) on \( \mathcal W \) must be either strictly positive or strictly negative.

\noindent \textbf{(Sufficiency)} Let us prove by contradiction. Suppose, to the contrary, that \( (0,0) \notin \operatorname{cl}(\operatorname{conv}((f,g)(\mathbb{R} ^n)))\setminus \operatorname{conv}((f,g)(\mathbb{R} ^n)) \). Namely, either \( (0,0)\in \operatorname{conv}((f,g)(\mathbb{R} ^n)) \) or \( (0,0)\notin \operatorname{cl}(\operatorname{conv}((f,g)(\mathbb{R} ^n))) \).

If \( (0,0)\in \operatorname{conv}((f,g)(\mathbb{R} ^n)) \), there exists \( (r,s),(-tr,-ts) \in (f,g)(\mathbb{R} ^n) \), \( t>0 \), and \( u,v\in \mathbb{R} ^n \) such that
\begin{equation}\label{(f,g)(u,v)}
(f(u),g(u))=(r,s), \, (f(v),g(v))=(-tr,-ts),~t>0
\end{equation}
and $(0,0)$ is a convex combination of $(r,s)$ and $(-tr,-ts).$
By $\ref{gu^an 2}$, the Lagrangian function \( \lambda f(x)+\mu g(x) \) is non-negative, so that
\[
\lambda r+\mu s\geq 0, \lambda (-tr)+\mu (-ts)\geq 0,~t>0.
\]
As a result, \( \lambda r+\mu s=0 \) and also $\lambda (-tr)+\mu (-ts)= 0$ so that
\[
\lambda f(u)+\mu g(u)=\lambda f(v)+\mu g(v)=0,
\]
which implies that \( u,v\in \mathcal W \). Moreover, by \eqref{(f,g)(u,v)},
\begin{align*}
(-\mu f+\lambda g)(u)\cdot (-\mu f+\lambda g)(v) &= (-\mu r+\lambda s)\cdot (-\mu (-tr)+\lambda (-ts))\\
&= -t(-\mu r+\lambda s)^2 \leq 0,
\end{align*}
which contradicts the statement $(b)$ of this proposition.

On the other hand, suppose that \( (0,0)\notin \operatorname{cl}(\operatorname{conv}((f,g)(\mathbb{R} ^n))) \). By Proposition \ref{0-t_i-gu_a}, there exists \( (\lambda ',\mu ')\neq (0,0) \) such that \( \lambda ' f(x)+\mu ' g(x)>0 \) for all \( x\in \mathbb{R} ^n \). Since the Lagrange multipliers are assumed to be unique up to multiplication by a positive constant, all possible Lagrangian functions \( \lambda f(x)+\mu g(x) \) must be strictly positive on \( x\in \mathbb{R} ^n \). It, however, contradicts the assumption that there exists a Lagrangian function \( \lambda f(x)+\mu g(x) \) which attains its minimum value 0. The proof for the sufficiency of this proposition is thus complete.
\end{proof}



\section{Proof for the nonhomogeneous Strict Finsler Lemma, Theorem \ref{NHSF}}
\label{pf:NHSF}

Our second proposal in this paper is to resolve the nonhomogeneous version of the strict Finsler lemma [Str.-Finsler], which asks to characterize
\[
\{f\leq 0\}\cap \{g=0\}=\emptyset ,~\text{or equivalently}~((-\infty ,0]\times \{0\}) \cap (f,g)(\mathbb{R} ^n) = \emptyset .
\]
Following the same scheme to analyze the nonhomogeneous version of [Calabi], we divide $((-\infty ,0]\times \{0\}) \cap (f,g)(\mathbb{R} ^n) = \emptyset $ into three exclusive cases.
\begin{itemize}
  \item[$\dagger $] (F1) in Theorem \ref{NHSF}: \(((-\infty ,0]\times \{0\}) \cap \operatorname{cl}(\operatorname{conv}((f,g)(\mathbb{R} ^n))) = \emptyset \). See Proposition \ref{su`ann gu_a} below.
  \item[$\dagger $] (F2) in Theorem \ref{NHSF}: 
\[
((-\infty ,0]\times \{0\}) \cap \operatorname{cl}(\operatorname{conv}((f,g)(\mathbb{R} ^n))) \neq \emptyset ; ((-\infty ,0]\times \{0\}) \cap \operatorname{conv}((f,g)(\mathbb{R} ^n)) = \emptyset .
\]
See Proposition \ref{tsu^an kh`o} below.
  \item[$\dagger $] (F3) in Theorem \ref{NHSF}: 
\[
((-\infty ,0]\times \{0\}) \cap \operatorname{conv}((f,g)(\mathbb{R} ^n)) \neq \emptyset ; ((-\infty ,0]\times \{0\}) \cap (f,g)(\mathbb{R} ^n) = \emptyset .
\]
See Proposition \ref{F2 prop} below.
\end{itemize}

In the same manner, we begin by proving (F3), before turning to (F1) and (F2).

\subsection{Proof of Theorem \ref{NHSF} (F3): $((-\infty ,0]\times \{0\}) \cap \operatorname{conv}((f,g)(\mathbb{R} ^n)) \ne \emptyset $ but $((-\infty ,0]\times \{0\}) \cap (f,g)(\mathbb{R} ^n) =\emptyset $}

The main focus of this subsection is to show that this case occurs if and only if the $0$-level set \( \{g=0\} \) separates the lower level set \( \{f\leq 0\} \). To facilitate the proof, we need the following lemma.

\begin{lemma}
\label{sep with aff}
If \( \{g=0\} \) separates \( \{f\mathbin{\star }0\} \), where \( \star \in \{<,\leq \} \), then \( \{g=0\} \) separates \( \{f-\zeta g\mathbin{\star }0\} \) for all \( \zeta \in \mathbb{R} \).
\end{lemma}

\begin{proof}
Since \( \{g=0\} \) separates \( \{f\mathbin{\star }0\} \) with \( \star \in \{<,\leq \} \) (note that $\star $ cannot be $=$ in this lemma), there are the following properties:
\begin{eqnarray}
&& \text{by \eqref{f=g=0} in Lemma \ref{eq sep}}, \{f\mathbin{\star } 0\}\cap \{g=0\}=\emptyset ;\\
&& \text{by Lemma \ref{h!ap}~(a) and (b)}, (\exists r\mathbin{\star }0)~(r,0)\in \operatorname{conv}((f,g)(\mathbb{R} ^n))\setminus (f,g)
(\mathbb{R} ^n).\label{sep-3}\\
&& \text{by Lemma \ref{Tsu prop 1}}, g~\text{is a non-constant affine function;}\label{sep-1}
\end{eqnarray}
From \eqref{sep-3}, a translation in $\mathbb{R} ^2$ by $(-r,0)$ gives
\[
(0,0)\in \operatorname{conv}((f-r,g)(\mathbb{R} ^n))\setminus (f-r,g)(\mathbb{R} ^n),
\]
which implies, by Proposition \ref{C2 prop}, that \( f-r \) and \( g \) have the separation property. Since $g$ is affine by \eqref{sep-1}, there must be \( \{g=0\} \) to separate \( \{f-r=0\} \) but not the other way around.

Choose an arbitrary \( \zeta \in \mathbb{R} \). We can rewrite \[ \{g=0\} =\{0\cdot (f-\zeta g-r)+g=0\} \] and \[ \{f-r=0\}= \{(f-\zeta g-r)+\zeta g=0\}. \] Thus the statement that \( \{g=0\} \) separates \( \{f-r=0\} \) becomes that
\[ \{0\cdot (f-\zeta g-r)+g=0\} ~{\rm separates} ~\{(f-\zeta g-r)+\zeta g=0\}. \] By Lemma
\ref{cor 2.9 ng} (b) and the fact that $g$ is affine, the 0-level set $\{g=0\}$ separates $\{f-\zeta g-r=0\}.$
By Lemma \ref{eq sep}, 
\[
(\exists u,v\in \mathbb{R} ^n)~ u,v\in \{f-\zeta g-r=0\}\subset \{f-\zeta g\mathbin{\star } 0\} \text{ such that } g(u)<0<g(v).
\]
Moreover, \( \{f-\zeta g\mathbin{\star } 0\}\cap \{g=0\}=\{f\mathbin{\star } 0\}\cap \{g=0\}=\emptyset \). It shows that \( \{g=0\} \) separates \( \{f-\zeta g\mathbin{\star }0\} \) for all \( \zeta \in \mathbb{R} \).
\end{proof}

\begin{proposition}
\label{F2 prop}
The $0$-level set \( \{g=0\} \) separates \( \{f\leq 0\} \) if and only if
\[
 ((-\infty ,0]\times \{0\})\cap \operatorname{conv}((f,g)(\mathbb{R} ^n))\neq \emptyset \text{ but }((-\infty ,0]\times \{0\})\cap (f,g)(\mathbb{R} ^n)=\emptyset .
\]
\end{proposition}

\begin{proof}
The necessary part holds by Lemma \ref{h!ap} (ii). For the sufficient part, let us assume that \( ((-\infty ,0]\times \{0\})\cap (f,g)(\mathbb{R} ^n)=\emptyset \) and \( (r,0)\in \operatorname{conv}((f,g)(\mathbb{R} ^n))\neq \emptyset \) for some \( r\leq 0 \). By the connectedness of \( (f,g)(\mathbb{R} ^n) \) and Fenchel--Eggleston theorem (Lemma \ref{fenchel-eggleston}), let \( u,v\in \mathbb{R} ^n \) such that
\begin{equation}
\label{(r,0)}
(r,0)=t(f(u),g(u))+(1-t)(f(v),g(v)) \text{ for some } 0<t<1.
\end{equation}
Since \( tf(u)+(1-t)f(v)=r\leq 0 \) and $t\in (0,1)$, the two values $f(u)$ and $f(v)$ cannot be both positive. Without loss of generality, let \( f(u)\leq 0 \). Thus, \( g(u)\neq 0 \) because \(((-\infty ,0]\times \{0\})\cap (f,g)(\mathbb{R} ^n)=\emptyset \). 

On the other hand, \eqref{(r,0)} also implies that \( 0=tg(u)+(1-t)g(v) \). Then, 
\begin{equation}\label{g(u)g(v)<0}
g(v)=-\frac t{1-t} g(u) \text{ and } g(u)g(v)=-\frac t{1-t} g(u)^2<0. 
\end{equation}

Define \( \zeta =\frac{f(u)-f(v)}{g(u)-g(v)} \). Then,
\begin{align*}
f(u)-\zeta g(u)
&=f(v)-\zeta g(v)\\
&=t(f(u)-\zeta g(u))+(1-t)(f(v)-\zeta g(v))\\
&=tf(u)+(1-t)f(v)-\zeta (tg(u)+(1-t)g(v))\\
&=r\leq 0.
\end{align*}
Therefore, \( u,v\in \{f-\zeta g\leq 0\} \) and \( g(u)g(v)<0 \) by \eqref{g(u)g(v)<0}. Also,
\[
\{f-\zeta g\leq 0\}\cap \{g=0\}=\{f\leq 0\}\cap \{g=0\}=\emptyset .
\]
By Lemma \ref{eq sep}, \( \{g=0\} \) separates \( \{f-\zeta g\leq 0\} \). By Lemma \ref{sep with aff}, \( \{g=0\} \) separates \( \{(f-\zeta g)+\zeta g\leq 0\}=\{f\leq 0\} \). The proof is complete.
\end{proof}

\subsection{Proof of Theorem \ref{NHSF} (F1): \(((-\infty ,0]\times \{0\}) \cap \operatorname{cl}(\operatorname{conv}((f,g)(\mathbb{R} ^n))) = \emptyset \)}
\begin{proposition}
\label{su`ann gu_a}
Let \( f,g \) be two real quadratic functions. Suppose that $g$ takes both positive and negative values. Then, the following statements are equivalent.
\begin{enumerate}
    \item There exists \( \xi \in \mathbb{R} \) such that \( f(x)+\xi g(x)>0 \) for all \( x\in \mathbb{R} ^n \).\label{su`ann it}
    \item There exists \( \xi \in \mathbb{R} \) and \( \nu >0 \) such that \( f(x)+\xi g(x) \geq \nu \) for all \( x\in \mathbb{R} ^n \). \label{su`ann j_i}
    \item The two sets, \( (-\infty ,0]\times \{0\} \) and \( \operatorname{cl}(\operatorname{conv}((f,g)(\mathbb{R} ^n))) \), are disjoint. \label{su`ann sann}
\end{enumerate}
\end{proposition}

\begin{proof}
The proofs of \( \ref{su`ann it}\Longrightarrow \ref{su`ann j_i} \) and \( \ref{su`ann j_i}\Longrightarrow \ref{su`ann sann} \) of Proposition \ref{su`ann gu_a} are similar to those of Proposition \ref{0-t_i-gu_a}, and hence are omitted. We only focus on the part \( \ref{su`ann sann}\Longrightarrow \ref{su`ann it} \).

Suppose that $\ref{su`ann sann}$ holds. In particular, \( (0,0)\notin \operatorname{cl}(\operatorname{conv}((f,g)(\mathbb{R} ^n))) \). By Proposition \ref{0-t_i-gu_a}, there exists \( (\lambda ,\mu )\neq (0,0) \) such that
\begin{equation}\label{gh-}
\lambda f(x)+\mu g(x)>0,~\forall x\in \mathbb{R} ^n.
\end{equation}
Since $g$ takes both positive and negative values, there exists some \( x_0\in \mathbb{R} ^n \) such that \( g(x_0)=0 \) and, by \ref{su`ann sann}, \( (f(x_0),g(x_0))\notin (-\infty ,0]\times \{0\} \). It implies that \( f(x_0)>0 \). Hence,
\[
\lambda f(x_0)+\mu g(x_0)=\lambda f(x_0)>0\Longrightarrow \lambda >0.
\]
Define \( \xi =\frac{\mu }{\lambda } \). Then, \eqref{gh-} renders the statement $\ref{su`ann it}$.
\end{proof}

\subsection{Proof of Theorem \ref{NHSF} (F2): \(((-\infty ,0]\times \{0\}) \cap \operatorname{cl}(\operatorname{conv}((f,g)(\mathbb{R} ^n))) \ne \emptyset \) but
$((-\infty ,0]\times \{0\}) \cap \operatorname{conv}((f,g)(\mathbb{R} ^n)) = \emptyset $}

\begin{proposition}
\label{tsu^an kh`o}
Suppose that \( g(x) \) takes both positive and negative values. Then, \( ((-\infty ,0]\times \{0\})\cap \operatorname{conv}((f,g)(\mathbb{R} ^n))=\emptyset \) but \( ((-\infty ,0]\times \{0\})\cap \operatorname{cl}(\operatorname{conv}((f,g)(\mathbb{R} ^n)))\neq \emptyset \) if and only if the following two statements hold.
\begin{enumerate}
    \item There exists a unique \( \xi \in \mathbb{R} \) such that the Lagrangian function \( f(x)+\xi g(x)\geq 0 \) for all \( x\in \mathbb{R} ^n \). Also, the Lagrangian function \( f(x)+\xi g(x) \) attains its minimal value 0.\label{tsu^an 2}
    \item On the minimum solution set of \( f(x)+\xi g(x) \),
\[
\mathcal V=\{f+\xi g=0\}=\{\nabla f+\xi \nabla g=0\},
\]
there is either \( g|_{\mathcal V}>0 \) or \( g|_{\mathcal V}<0 \). \label{tsu^an 3}
\end{enumerate}
\end{proposition}

\begin{proof}
\textbf{(Necessity)}
Since \( (-\infty ,0]\times \{0\} \) and \( \operatorname{conv}((f,g)(\mathbb{R} ^n)) \) are disjoint and both sets are convex, by the separating hyperplane theorem, there exists a separation hyperplane \( \{(r,s)\in \mathbb{R} ^2 \mid \lambda r+\mu s=\nu \} \) with \( (\lambda ,\mu )\neq (0,0) \) and \( \nu \in \mathbb{R} \) such that
\begin{align}
(-\infty ,0]\times \{0\} &\subset \{(r,s)\in \mathbb{R} ^2 \mid \lambda r+\mu s\leq \nu \};\label{khui-tshiat}\\
\operatorname{conv}((f,g)(\mathbb{R} ^n)) &\subset \{(r,s)\in \mathbb{R} ^2 \mid \lambda r+\mu s\geq \nu \}.\label{JNR-tshiat}
\end{align}
From \eqref{khui-tshiat}, $(0,0)\in \{(r,s)\in \mathbb{R} ^2 \mid \lambda r+\mu s\leq \nu \}$ so that \( \nu \geq 0 \). Moreover, for all \( r<0 \) and $s=0$, there is \( \lambda r+\mu s=\lambda r\leq \nu \).
\begin{equation}\label{1fsf--}
(\forall r<0)~ \lambda \geq \frac{\nu }{r}, \text{ and hence } \lambda \geq 0.
\end{equation}
It now follows from \eqref{JNR-tshiat} that
\begin{equation}\label{1fsf}
(f,g)(\mathbb{R} ^n) \subset \operatorname{conv}((f,g)(\mathbb{R} ^n)) \subset \{(r,s)\in \mathbb{R} ^2 \mid \lambda r+\mu s\geq \nu \geq 0\},
\end{equation}
It implies that \( \lambda f(x)+\mu g(x)\geq 0 \) for all \( x\in \mathbb{R} ^n \).

We claim that, in \eqref{1fsf--}, \( \lambda >0 \) indeed. Suppose on the contrary that \( \lambda =0 \). Due to \( (\lambda ,\mu )\neq (0,0) \) for the separation hyperplane, we know \( \mu \neq 0 \) so that, \eqref{1fsf} gives $\mu g(x)\geq 0$, which violates the two-sided Slater condition that \( g(x) \) takes both positive and negative values.
As \( \lambda >0 \), we can let \( \xi =\frac{\mu }{\lambda } \) to obtain 
\[
f(x)+\xi g(x)\geq 0,~ \forall x\in \mathbb{R} ^n .
\]

However, \( f(x)+\xi g(x) \) cannot be strictly positive on $\mathbb{R} ^n$. Otherwise, by Lemma \ref{QP0}, there would exist a minimum point $\bar{x}\in \mathbb{R} ^n$ such that 
\[
f(x)+\xi g(x)\geq f(\bar{x})+\xi g(\bar{x}) >0, \; \;\forall \, x\in \mathbb{R} ^n.
\]
That is, \( (f,g)(\mathbb{R} ^n)\subset \{(r,s)\in \mathbb{R} ^2 \mid r+\xi s\geq f(\bar{x})+\xi g(\bar{x}) \} \). Taking the convex hull and the closure on both sides, we have
\begin{equation}
\label{SF cl conv}
\underbrace {\operatorname{cl}(\operatorname{conv}((f,g)(\mathbb{R} ^n)))}_{W}\subset \underbrace {\{(r,s)\in \mathbb{R} ^2 \mid r+\xi s\geq f(\bar{x})+\xi g(\bar{x}) >0\}}_{S}.
\end{equation}
From \eqref{SF cl conv}, we notice that 
\begin{equation}\label{U cap S}
\underbrace {(-\infty ,0]\times \{0\}}_{U} \cap \underbrace {\{(r,s)\in \mathbb{R} ^2 \mid r+\xi s\geq f(\bar{x})+\xi g(\bar{x}) >0\}}_{S}=\emptyset ,
\end{equation}
and thus
\begin{equation}\label{U cap W}
\underbrace {(-\infty ,0]\times \{0\}}_{U} \cap \underbrace {\operatorname{cl}(\operatorname{conv}((f,g)(\mathbb{R} ^n)))}_{W} = \emptyset ,
\end{equation}
which is a contradiction to the assumption. It implies that \( f(x)+\xi g(x)\geq 0 \) for all \( x\in \mathbb{R} ^n \) and there exists some \( x^* \) such that \( f(x^*)+\xi g(x^*)=0 \) and \(x^* \) is a minimizer of \( f(x)+\xi g(x) \) on \( \mathbb{R} ^n \).

We now argue that the Lagrange multiplier \( \xi \) is indeed unique.
Suppose, to the contrary, that there exists another \( \xi '\neq \xi \) that also satisfies \( f(x)+\xi 'g(x)\geq 0 \) for all \( x\in \mathbb{R} ^n \). By Lemma \ref{siang=>giam}, the sum of the two Lagrangian functions must be strictly positive; that is
\[
2f(x)+(\xi +\xi ')g(x)>0,\; \;\forall \, x\in \mathbb{R} ^n,
\]
which again leads to \eqref{SF cl conv}, \eqref{U cap S}, \eqref{U cap W}, and a contradiction. Therefore, the Lagrange multiplier \( \xi \) is unique.

Finally, by the assumption \( ((-\infty ,0]\times \{0\})\cap \operatorname{conv}((f,g)(\mathbb{R} ^n))=\emptyset \), we see that
\[
(0,0) \notin (f,g)(\mathbb{R} ^n). \text{ That is, } \{f=0\}\cap \{g=0\}=\emptyset .
\]
If we consider the $0$-level set of the function $g$ restricted to the minimum solution set of \( f(x)+\xi g(x) \), i.e., $\mathcal V=\{f+\xi g=0\}$, we find that 
\[
\{g|_{\mathcal V}=0\}=\{f+\xi g=0\}\cap \{g=0\}=\{f=0\}\cap \{g=0\}=\emptyset .
\]
Since the affine subspace \( \mathcal V \) is connected, the image set of \( g|_\mathcal V \) must be a connected interval that does not include the value $0.$ As a consequence, \( g|_\mathcal V \) is either strictly positive or strictly negative. The necessity of this proposition is thus proved.

\noindent \textbf{(Sufficiency)} Let us prove by contradiction. If \(( (-\infty ,0]\times \{0\})\cap \operatorname{cl}(\operatorname{conv}((f,g)(\mathbb{R} ^n)))\) were assumed to be empty, by the equivalence of Proposition \ref{su`ann gu_a} $(a)$ and $(c)$, there would exist \( \xi '\in \mathbb{R} \) such that \( f(x)+\xi 'g(x)>0 \) for all \( x\in \mathbb{R} ^n \), which contradicts \ref{tsu^an 2} of this proposition. Therefore,
\(( (-\infty ,0]\times \{0\})\cap \operatorname{cl}(\operatorname{conv}((f,g)(\mathbb{R} ^n)))\ne \emptyset .\)

Secondly, suppose that \( ((-\infty ,0]\times \{0\})\cap \operatorname{conv}((f,g)(\mathbb{R} ^n))\neq \emptyset \). By FenchelEggleston theorem (Lemma \ref{fenchel-eggleston}), there would exist \( u,v \in \mathbb{R} ^n \) and some \( 0\leq \bar t\leq 1 \) such that
\begin{equation}\label{tree-4}
\bar t(f(u),g(u))+(1-\bar t)(f(v),g(v))\in (-\infty ,0]\times \{0\}.
\end{equation}
That is, we have
\begin{equation}
\label{conv SF}
\bar tf(u)+(1-\bar t)f(v)\leq 0,\, \bar tg(u)+(1-\bar t)g(v)=0.
\end{equation}
By the non-negativity of \( f(x)+\xi g(x) \), there also are
 \[
 f(u)+\xi g(u)\geq 0 \text{ and } f(v)+\xi g(v)\geq 0.
 \] 
Hence,
\begin{eqnarray*}
&& \bar t\underbrace {(f(u)+\xi g(u))}_{\geq 0}+(1-\bar t)\underbrace {(f(v)+\xi g(v))}_{\geq 0}\\
&=& \underbrace {(\bar tf(u)+(1-\bar t)f(v))}_{\leq 0~\rm{by}~\eqref{conv SF}}+\xi \underbrace {(\bar tg(u)+(1-\bar t)g(v))}_{=0~\rm{by}~\eqref{conv SF}} \leq 0.
\end{eqnarray*}
It implies that 
\begin{equation}\label{tree-3}
 \bar t(f(u)+\xi g(u))=0 \text{ and } (1-\bar t)(f(v)+\xi g(v))=0 .
\end{equation}

If \( 0<\bar t<1 \), \eqref{tree-3} forces \( f(u)+\xi g(u)=0 \) and \( f(v)+\xi g(v)=0 \). Then, both \( u \) and \( v \) attain the minimum value 0 of the Lagrangian function. That is, \( u,v\in \mathcal V \). 

However, from \eqref{conv SF},
\[
g(u)=-\frac{1-\bar t}{\bar t} g(v) \text{ so that } g(u)g(v)=-\frac{1-\bar t}{\bar t} g(v)^2 \leq 0,
\]
which contradicts \ref{tsu^an 3} that \( g|_\mathcal V \) is either a positive or a negative function.

If \( \bar t=0 \) in \eqref{conv SF}, there is \( f(v)\leq 0 \) and \( g(v)=0 \) so that \( f(v)+\xi g(v)\leq 0 \). Since the Lagrangian function is non-negative, \( f(v)+\xi g(v)=0 \) and \( v\in \mathcal V \), which is impossible since \( g(v)=0 \), contradicting either \( g|_\mathcal V>0 \) or \( g|_\mathcal V <0 \).

If \( \bar t=1 \) in \eqref{conv SF}, \( f(u)\leq 0 \) and \( g(u)=0 \) will lead to the same contradiction as above.

Therefore, neither \eqref{conv SF} nor \eqref{tree-4} can hold. That is, \( ((-\infty ,0]\times \{0\})\cap \operatorname{conv}((f,g)(\mathbb{R} ^n))\neq \emptyset \) is a false assumption.
The sufficiency of this proposition is thus proved.
\end{proof}

\section{Illustrative examples for [Calabi](nonhomo version) and [Str.-Finsler]({nonhomo version})}

In this section, we use several examples in \( \mathbb{R} ^2 \) to illustrate conditions (C1)-(C3) in Theorem \ref{NHFC} for [Calabi](nonhomo.\ version), as well as conditions (F1)-(F3) in Theorem \ref{NHSF} for [Str.-Finsler]({nonhomo version}). In each example, the mapping $(f,g)$ maps $(x_1,x_2)\in \mathbb{R} ^2$ to their joint numerical range $(r,s)\in (f,g)(\mathbb{R} ^2)$ with $r=f(x_1,x_2),~s=g(x_1,x_2).$ Representing [Calabi](nonhomo.\ version) and [Str.-Finsler](nonhomo.\ version) in the joint numerical range setting, we have
$$\{f=0\}\cap \{g=0\}=\emptyset \Longleftrightarrow (0,0)\notin (f,g)(\mathbb{R} ^2)$$
and 
$$\{f\leq 0\}\cap \{g=0\}=\emptyset \Longleftrightarrow \left((-\infty ,0]\times \{0\}\right)\cap (f,g)(\mathbb{R} ^2)=\emptyset .$$
Also recall that 
$$\operatorname{cl}(\operatorname{conv}((f,g)(\mathbb{R} ^2)))\supset \operatorname{conv}((f,g)(\mathbb{R} ^2))\supset (f,g)(\mathbb{R} ^2).$$

Example \ref{eg:CF1} is the case when $(0,0)\notin \operatorname{cl}(\operatorname{conv}((f,g)(\mathbb{R} ^2))),$ indicating that $(0,0)$ lies away from $(f,g)(\mathbb{R} ^2)$ with a positive margin to the closure of its convex hull.

Example \ref{eg:F2} illustrates 
$(0,0)\notin \operatorname{conv}((f,g)(\mathbb{R} ^2))$ but $(0,0)\in \operatorname{cl}(\operatorname{conv}((f,g)(\mathbb{R} ^2))).$ Although \( (0,0) \) does not belong to the convex hull of \( (f,g)(\mathbb{R} ^2) \), it belongs to the boundary of the convex hull as a limit point.

Finally, Example \ref{eg:CF3} shows a case where $(f,g)(\mathbb{R} ^2)$ is not convex; $(0,0)\notin (f,g)(\mathbb{R} ^2)$; yet the convex hull contains $(0,0)$.

Here we would like to draw the reader’s special attention to the fact that in Example \ref{eg:CF1} and Example \ref{eg:F2}, the joint numerical range $(f,g)(\mathbb{R} ^2)$ is a convex set, whereas in Example \ref{eg:CF3}, $(f,g)(\mathbb{R} ^2)$ is nonconvex.


\begin{example}
\label{eg:CF1}
(for (C1) in Theorem \ref{NHFC} and (F1) in Theorem \ref{NHSF})

Let \( f(x_1,x_2)=-x_1^2+x_2^2+2 \), \( g(x_1,x_2)=x_1^2 \). When \( g(x_1,x_2)=0 \), we have \( x_1=0,~x_2\in \mathbb{R} \) and hence \( f(0,x_2)=x_2^2+2\geq 2 \). Therefore, 
\[
\{f=0\}\cap \{g=0\}=\emptyset =\{f\leq 0\}\cap \{g=0\}.
\]
The joint numerical range in the example is
\begin{align*}
(f,g)(\mathbb{R} ^2) & =\{(-x_1^2+x_2^2+2,x_1^2) \mid x _1,x_2\in \mathbb{R} \}\\
& =\{(-y_1+y_2+2,y_1) \mid y_1\geq 0,y_2\geq 0\}~(y_1=x_1^2,~y_2=x_2^2)\\
& = \{(r,s) \mid r+s\geq 2,~s\geq 0\}~(r=-y_1+y_2+2,~s=y_1),
\end{align*}
which is convex and closed. Hence $\operatorname{cl}(\operatorname{conv}((f,g)(\mathbb{R} ^2)))=(f,g)(\mathbb{R} ^2).$

Note that $(0,0)\notin (f,g)(\mathbb{R} ^2)$ and $\left((-\infty ,0]\times \{0\}\right)\cap (f,g)(\mathbb{R} ^2)=\emptyset .$ In both cases, the hyperplane \( \mathcal{H} =\{(r,s) \mid r + s=1\} \) separates strictly the joint numerical range $(f,g)(\mathbb{R} ^2)$ from $\{(0,0)\}$ (Fig.\ \ref{figC1})
and also from $\left((-\infty ,0]\times \{0\}\right)$ (Fig.\ \ref{figF1}). Since $(f,g)(\mathbb{R} ^2)$ lies in the upper half-space $\{(r,s) \mid r + s>1\},$ there is
$$f(x)+g(x)>1,~\forall x\in \mathbb{R} ^2,$$
which validates (C1) in Theorem \ref{NHFC} and (F1) in Theorem \ref{NHSF}.

\begin{figure}[bth]
\centering
\begin{minipage}[b]{0.45\textwidth}
\centering
\begin{tikzpicture}[scale=0.65]
    \draw[->] (-4,0) -- (4,0) node[right] {$r$};
    \draw[->] (0,-4) -- (0,4) node[above] {$s$};
    \node[green!70!black, below left] at (3.5,-2.5) {\(\mathcal{H} : r+s=1\)};
    \node[blue] at (2,3) {\((f,g)(\mathbb{R} ^2)\)};
    \node[below] at (2,0) {\( (2,0) \)};
    \node[left] at (0,2) {\( (0,2) \)};

     \fill[blue!20,opacity=0.5] (2,0) -- (4,0) -- (4,4) -- (-2,4) -- cycle;
    \draw[thick,blue] (-2,4) -- (2,0);
    \draw[thick,blue] (2,0) -- (4,0);
    \draw[draw=red, fill=red] (0,0) circle (2pt) node[red, below left] {\(\{(0,0)\}\)};
    \draw[thick,green!70!black] (-3,4) -- (4,-3);
\end{tikzpicture}
\caption{Example \ref{eg:CF1} shows (C1) with $\lambda f(x)+\mu g(x)>0$ and $\lambda =\mu =1.$}
\label{figC1}
\end{minipage}
\hspace{0.06\textwidth}
\begin{minipage}[b]{0.45\textwidth}
\centering
\begin{tikzpicture}[scale=0.65]
    \draw[->] (-4,0) -- (4,0) node[right] {$r$};
    \draw[->] (0,-4) -- (0,4) node[above] {$s$};
    \node[red, below] at (-2,0) {\((-\infty ,0]\times \{0\}\)};
    \node[green!70!black, below left] at (3.5,-2.5) {\(\mathcal{H} : r+s=1\)};
    \node[blue] at (2,3) {\((f,g)(\mathbb{R} ^2)\)};
    \node[below] at (2,0) {\( (2,0) \)};
    \node[left] at (0,2) {\( (0,2) \)};

     \fill[blue!20,opacity=0.5] (2,0) -- (4,0) -- (4,4) -- (-2,4) -- cycle;
    \draw[thick,blue] (-2,4) -- (2,0);
    \draw[thick,blue] (2,0) -- (4,0);
    \draw[thick,red] (-4,0) -- (0,0);
    \draw[draw=red, fill=red] (0,0) circle (2pt);
    \draw[thick,green!70!black] (-3,4) -- (4,-3);
\end{tikzpicture}
\caption{Example \ref{eg:CF1} shows (F1) with $f(x)+\xi g(x)>0$ and $\xi =1.$}
\label{figF1}
\end{minipage}
\end{figure}
\end{example}

\begin{example}
\label{eg:F2}
(for (C2) in Theorem \ref{NHFC} and (F2) in Theorem \ref{NHSF})

Consider \( f(x_1,x_2)=x_1^2-x_1x_2+1 \) and \( g(x_1,x_2)=x_1x_2-1 \). Observe that 
$$g(x_1,x_2)=0~\Rightarrow ~x_1x_2=1,~x_1\ne 0~\Rightarrow ~f(x_1,x_2)=x_1^2-x_1x_2+1=x_1^2>0.$$
Therefore, 
\[
\{f=0\}\cap \{g=0\}=\emptyset =\{f\leq 0\}\cap \{g=0\}.
\]
As can be seen in Fig.\ \ref{figC2'}, the two level sets \( \{f=0\} \) and \( \{g=0\} \) approach each other asymptotically but never intersect. It can be justified by computing the vertical difference of the two level sets:
 \[
\left|\left(x_1+\frac{1}{x_1}\right)-\frac{1}{x_1}\right|
=|x_1|\longrightarrow0
\qquad\text{as }x_1\to0.
\]

\begin{figure}[bth]
\centering
\begin{minipage}[b]{0.45\textwidth}
\centering
\begin{tikzpicture}[scale=0.65]
    \draw[->] (-4,0) -- (4,0) node[right] {$x_1$};
    \draw[->] (0,-4) -- (0,4) node[above] {$x_2$};

  \draw[thick,cyan] plot[domain=0.25:1 ,samples=25] (\x,{1/\x}) --
  plot[domain=1:4 ,samples=25] (\x,{1/\x});
  \draw[thick,cyan] plot[domain=-0.25:-1 ,samples=25] (\x,{1/\x}) --
  plot[domain=-1:-4 ,samples=25] (\x,{1/\x});
\draw[thick,magenta] plot[domain={2-sqrt(3)}:1 ,samples=25] (\x,{\x + 1/\x}) --
  plot[domain=1:{2+sqrt(3)} ,samples=25] (\x,{\x + 1/\x});
  \draw[thick,magenta] plot[domain={-2+sqrt(3)}:-1 ,samples=25] (\x,{\x + 1/\x}) --
  plot[domain=-1:{-2-sqrt(3)} ,samples=25] (\x,{\x + 1/\x});

  \node[cyan,above right] at (2,0.5) {\(\{g=0\}\)};
  \node[cyan,below left] at (-1,-1) {\(x_1x_2-1=0\)};
  \node[magenta,below right] at (2,2.5) {\(\{f=0\}\)};
  \node[magenta, right] at (-0.5,-2.5) {\(x_1^2-x_1x_2+1=0\)};
\end{tikzpicture}
\caption{In Example \ref{eg:F2}, $\{f=0\}=\{x_1^2-x_1x_2+1=0\}$ and $\{g=0\}=\{x_1x_2-1=0\}$ have no intersection.}
\label{figC2'}
\end{minipage}
\hspace{0.06\textwidth}
\begin{minipage}[b]{0.45\textwidth}
\centering
\begin{tikzpicture}[scale=0.65]
    \draw[->] (-4,0) -- (4,0) node[right] {$x_1$};
    \draw[->] (0,-4) -- (0,4) node[above] {$x_2$};

  \draw[thick,cyan] plot[domain=0.25:1 ,samples=25] (\x,{1/\x}) --
  plot[domain=1:4 ,samples=25] (\x,{1/\x});
  \draw[thick,cyan] plot[domain=-0.25:-1 ,samples=25] (\x,{1/\x}) --
  plot[domain=-1:-4 ,samples=25] (\x,{1/\x});
\draw[thick,magenta] plot[domain={2-sqrt(3)}:1 ,samples=25] (\x,{\x + 1/\x}) --
  plot[domain=1:{2+sqrt(3)} ,samples=25] (\x,{\x + 1/\x});
  \draw[thick,magenta] plot[domain={-2+sqrt(3)}:-1 ,samples=25] (\x,{\x + 1/\x}) --
  plot[domain=-1:{-2-sqrt(3)} ,samples=25] (\x,{\x + 1/\x});
  \fill[magenta!20,opacity=0.5] plot[domain={-2+sqrt(3)}:-1 ,samples=25] (\x,{\x + 1/\x}) -- 
  plot[domain=-1:{-2-sqrt(3)} ,samples=25] (\x,{\x + 1/\x}) -- cycle;
\fill[magenta!20,opacity=0.5] plot[domain={2-sqrt(3)}:1 ,samples=25] (\x,{\x + 1/\x}) -- 
  plot[domain=1:{2+sqrt(3)} ,samples=25] (\x,{\x + 1/\x}) -- cycle;

  \node[cyan,above right] at (2,0.5) {\(\{g=0\}\)};
  \node[cyan,below left] at (-1,-1) {\(x_1x_2-1=0\)};
  \node[magenta,below right] at (0.5,3.5) {\(\{f\leq 0\}\)};
  \node[magenta,right] at (-0.5,-2.5) {\(x_1^2-x_1x_2+1=0\)};
\end{tikzpicture}
\caption{In Example \ref{eg:F2}, $\{f\leq 0\}=\{x_1^2-x_1x_2+1\leq 0\}$ and $\{g=0\}=\{x_1x_2-1=0\}$ have no intersection.}
\label{figF2'}
\end{minipage}
\end{figure}

As for the joint numerical range $(f,g)(\mathbb{R} ^2),$ we first see that 
\begin{align*}
(f+g,g)(\mathbb{R} ^2) &=\{(x_1^2,x_1x_2-1) \mid x_1,x_2\in \mathbb{R} \} \\
&=\bigcup _{t\neq 0}\{(t^2,tx_2-1) \mid x_2\in \mathbb{R} \}\cup \{(0,-1)\} \\
&=\{(\gamma ,\sigma ) \mid \gamma >0,~\sigma \in \mathbb{R} \}\cup \{(0,-1)\}.\\
\end{align*}
By the fact that 
$$(\gamma ,\sigma )\in (f+g,g)(\mathbb{R} ^2)~\Longleftrightarrow ~(\gamma -\sigma ,\sigma )\in (f,g)(\mathbb{R} ^2),$$
the joint numerical range $(f,g)(\mathbb{R} ^2)$ in the example can be expressed as 
\begin{eqnarray}
(f,g)(\mathbb{R} ^2)&=&\{(r,s)\mid ~r=\gamma -\sigma ,~s=\sigma ,~(\gamma ,\sigma )\in \{(\gamma ,\sigma ) \mid \gamma >0\}\cup \{(0,-1)\}\}\nonumber\\
&=&\{(r,s) \mid r+s>0\}\cup \{(1,-1)\},\label{eq:JNR2}
\end{eqnarray}
which is a convex set as shown in Fig.\ \ref{fig2_0}. Note that 
$$(0,0)\notin (f,g)(\mathbb{R} ^2) \text{ and } \left((-\infty ,0]\times \{0\}\right)\cap (f,g)(\mathbb{R} ^2)=\emptyset .$$
However, after taking the closure of the set in \eqref{eq:JNR2} (see Fig.\ \ref{fig2_0}), we obtain
$$(0,0)\in \operatorname{cl}((f,g)(\mathbb{R} ^2))\text{ and } \left((-\infty ,0]\times \{0\}\right)\cap \operatorname{cl}((f,g)(\mathbb{R} ^2))\ne \emptyset ,$$
which is the case (C2) in Theorem \ref{NHFC} and the case (F2) in Theorem \ref{NHSF}, respectively.

\begin{figure}[bth]
\begin{minipage}[b]{0.45\textwidth}
\centering
\begin{tikzpicture}[scale=0.65]
    \draw[->] (-4,0) -- (4,0) node[right] {$r$};
    \draw[->] (0,-4) -- (0,4) node[above] {$s$};
    \node[blue] at (2,3) {\((f,g)(\mathbb{R} ^2)\)};
    \node[below left] at (-2,2) {\( r+s=0 \)};
    \node[red, above] at (-2,0) {\((-\infty ,0]\times \{0\}\)};

     \draw[draw=blue, fill=blue] (1,-1) circle (2pt) node[below left] {\( (1,-1) \)};
     \fill[blue!20,opacity=0.5] (4,-4) -- (-4,4) -- (4,4) -- cycle;
    \draw[blue,dashed,thick] (-4,4) -- (4,-4);
    \draw[thick,red] (-4,0) -- (0,0);
    \draw[draw=red, fill=red] (0,0) circle (2pt);

\end{tikzpicture}
\caption{The convex joint numerical range in Example \ref{eg:F2} has no intersection with either $\{(0,0)\}$ or \((-\infty ,0]\times \{0\} \).}
\label{fig2_0}
\end{minipage}
\hspace{0.06\textwidth}
\begin{minipage}[b]{0.45\textwidth}
\centering
\begin{tikzpicture}[scale=0.65]
    \draw[->] (-4,0) -- (4,0) node[right] {$r$};
    \draw[->] (0,-4) -- (0,4) node[above] {$s$};
    \node[red, above] at (-2,0) {\((-\infty ,0]\times \{0\}\)};
    \node[green!70!black, below left] at (3,-3) {\(\mathcal{H} : r+s=0\)};
    \node[blue] at (2,3) {\((f,g)(\mathbb{R} ^2)\)};

     \draw[draw=blue, fill=blue] (1,-1) circle (2pt) node[below left] {\((1,-1)\)};
     \fill[blue!20,opacity=0.5] (4,-4) -- (-4,4) -- (4,4) -- cycle;
     \draw[very thick,green!70!black] (-4,4) -- (4,-4);
    \draw[blue,dashed] (-4,4) -- (4,-4);
    \draw[thick,red] (-4,0) -- (0,0);
    \draw[draw=red, fill=red] (0,0) circle (2pt);
\end{tikzpicture}
\caption{In Example \ref{eg:F2},~\( \mathcal{H} =\{(r,s) \mid r+s=0\} \) is the only separating hyperplane but the separation is not strict.}
\label{figF2.}
\end{minipage}
\end{figure}

\begin{figure}[bth]
\begin{minipage}[b]{0.45\textwidth}
\centering
\begin{tikzpicture}[scale=0.65]
    \draw[->] (-4,0) -- (4,0) node[right] {$r$};
    \draw[->] (0,-4) -- (0,4) node[above] {$s$};
    \node[green!70!black, below left] at (3,-3) {\(\mathcal{H} : r+s=0\)};
    \node[blue] at (2,3) {\((f,g)(\mathbb{R} ^2)\)};
    \node[gray,above left] at (-1,-3) {\( r-s=2 \)};

      \fill[blue!20,opacity=0.5] (4,-4) -- (-4,4) -- (4,4) -- cycle;
     \draw[draw=blue, fill=blue] (1,-1) circle (2pt) node[below left] {\((1,-1)\)};
     \draw[thick,gray,dashed] (-2,-4) -- (4,2);
     \draw[very thick,green!70!black] (-4,4) -- (4,-4);
    \draw[blue,dashed] (-4,4) -- (4,-4);
    \draw[draw=red, fill=red] (0,0) circle (2pt) node[red, below left] {\(\{(0,0)\}\)};

\end{tikzpicture}
\caption{Non-strict separation hyperplane}
\label{figF2}
\end{minipage}
\hspace{0.06\textwidth}
\begin{minipage}[b]{0.45\textwidth}
\centering
\begin{tikzpicture}[scale=0.65]
    \draw[->] (-4,0) -- (4,0) node[right] {$r$};
    \draw[->] (0,-4) -- (0,4) node[above] {$s$};
    \node[red, above] at (-2,0) {\((-\infty ,0]\times \{0\}\)};
    \node[green!70!black, below left] at (3,-3) {\(\mathcal{H} : r+s=0\)};
    \node[blue] at (2,3) {\((f,g)(\mathbb{R} ^2)\)};
    \node[gray,below] at (-2.5,-1) {\( s=-1 \)};

     \draw[draw=blue, fill=blue] (1,-1) circle (2pt) node[below left] {\((1,-1)\)};
     \fill[blue!20,opacity=0.5] (4,-4) -- (-4,4) -- (4,4) -- cycle;
     \draw[thick,gray,dashed] (-4,-1) -- (4,-1);
     \draw[very thick,green!70!black] (-4,4) -- (4,-4);
    \draw[blue,dashed] (-4,4) -- (4,-4);
    \draw[thick,red] (-4,0) -- (0,0);
    \draw[draw=red, fill=red] (0,0) circle (2pt);
\end{tikzpicture}
\caption{Non-strict separation hyperplane for Example \ref{eg:F2}}
\label{figF2-}
\end{minipage}
\end{figure}

From Fig.\ \ref{figF2.}, observe that \( \mathcal{H} =\{(r,s) \mid r+s=0\} \) is the only hyperplane that can separate the convex joint numerical range in \eqref{eq:JNR2} from $\{(0,0)\}$ (or from the larger convex set \((-\infty ,0]\times \{0\} \)). The separation by \( \mathcal{H} \) is not strict because it touches $(0,0)\in (-\infty ,0]\times \{0\}$ and $(1,-1)\in (f,g)(\mathbb{R} ^2)$.
As such, \( \mathcal{H} =\{(r,s) \mid r+s=0\} \) alone is not sufficient to determine whether \( (f,g)(\mathbb{R} ^2) \) in \eqref{eq:JNR2} intersects $\{(0,0)\}$ or \((-\infty ,0]\times \{0\} \).

Then, (C2) in Theorem \ref{NHFC} and (F2) in Theorem \ref{NHSF} require the touching point $(1,-1)\in (f,g)(\mathbb{R} ^2)$ by $\mathcal{H} $ to be further separated from $(0,0)$ (in Theorem \ref{NHFC}) or from $(-\infty ,0]\times \{0\}$ (in Theorem \ref{NHSF}), should we want to conclude that
$\{f=0\}\cap \{g=0\}=\emptyset $ (or $\{f\leq 0\}\cap \{g=0\}=\emptyset $). 
The central issue is to provide a general representation for the touching point $(1,-1)\in (f,g)(\mathbb{R} ^2)$.

The touching point $(1,-1)\in (f,g)(\mathbb{R} ^2)$ by \( \mathcal{H} =\{(r,s) \mid r+s=0\} \) can be written as the minimum solution $(r^*,s^*)=(1,-1)$ that minimizes $k=r+s$ over the joint numerical range $(r,s)\in (f,g)(\mathbb{R} ^2).$ Specifically, 
\begin{eqnarray*}
0=r^*+s^*&=&\min\limits_{(r,s)}\{r+s=k\,\vert~r=f(x_1,x_2),s=g(x_1,x_2)\}\\
&=&\min\limits_{(x_1,x_2)}\{f(x_1,x_2)+g(x_1,x_2)=x_1^2~\vert~ (x_1,x_2)\in \mathbb{R} ^2\},
\end{eqnarray*}
where the minimal value \( 0 \) is attained (in the $x_1$-$x_2$ space) on 
\[
{\mathcal W}={\mathcal V}=\{(0,x_2) \mid x_2\in \mathbb{R} \}~~(\mathcal W\text{ for (C2)};~\mathcal V\text{ for (F2)}).
\]
The set ${\mathcal W}$ (or ${\mathcal V}$) is the pre-image of the touching point {\small $(1,-1)$} under the joint mapping $(f,g).$

If the touching point $(1,-1)$ is to be separated from $(0,0)$ by a line $\mathcal{L} =\{(r,s) \mid \sigma r+\tau s=k\}$, then $\mathcal{L} $ must be non-parallel to $\mathcal{H} $ with $k\neq 0$ (since $k=0$ would make $\mathcal{L} $ pass through $(0,0)$). Given the pre-image ${\mathcal W}$, it amounts to checking whether $\sigma f(x_1,x_2)+\tau g(x_1,x_2)$ takes strictly positive or strictly negative values on ${\mathcal W}$. In Fig.\ \ref{figF2}, we choose $\mathcal{L} $ to be $r-s=2$ which passes through $(1,-1)$ with $k=2$ and it is perpendicular to \( \mathcal{H} =\{(r,s) \mid r+s=0\} \). 

On the other hand, if the touching point $(1,-1)$ is to be separated from the ray \((-\infty ,0]\times \{0\} \), we can choose a horizontal line $\mathcal{L} =\{(r,s) \mid s=k\}$ with $k\ne 0$. In Fig.\ \ref{figF2-}, the line $s=-1$ serves the purpose. In the pre-image, we see that \( g(x_1,x_2)=x_1x_2-1 \) indeed takes the constant value $-1<0$ on ${\mathcal V}=\{(0,x_2) \mid x_2\in \mathbb{R} \}.$
\end{example}

\begin{example}
\label{eg:CF3}
(for (C3) in Theorem \ref{NHFC} and (F3) in Theorem \ref{NHSF})
Consider \( f(x_1,x_2)=-x_1^2+x_2^2+1 \), \( g(x_1,x_2)=2x_1-x_2 \). Observe that
\[
g(x_1,x_2)=0 \Rightarrow x_2= 2x_1 \Rightarrow f(x_1,x_2)=3x_1^2+1>0.
\]
Therefore,
\[
\{f=0\}\cap \{g=0\}=\emptyset =\{f\leq 0\}\cap \{g=0\}.
\]
Notice that the $0$-level set \( \{f=0\} \) is a hyperbola of two branches:
\begin{align*}
\{f=0\}^+&=\{(x_1,x_2) \mid x_1^2-x_2^2=1, x_1>0\},\\
\{f=0\}^-&=\{(x_1,x_2) \mid x_1^2-x_2^2=1, x_1<0\}.
\end{align*}
Likewise, the sub-level set \( \{f\leq 0\} \) contains two connected components:
\begin{align*}
\{f\leq 0\}^+ &=\{(x_1,x_2) \mid x_1^2-x_2^2\geq 1, x_1>0\},\\
\{f\leq 0\}^- &=\{(x_1,x_2) \mid x_1^2-x_2^2\geq 1, x_1<0\}.
\end{align*}
See Fig.\ \ref{fig:C3'} and Fig.\ \ref{fig:F3'} respectively. 
As \( (1,0)\in \{f=0\}^+ \) and \( g(1,0)=2>0 \), while \( (-1,0)\in \{f=0\}^- \) and \( g(-1,0)=-2<0 \), Lemma \ref{eq sep} confirms that \( \{g=0\} \) separates \( \{f=0\} \) and also \( \{g=0\} \) separates \( \{f\leq 0\}. \) This example validates (C3) in Theorem \ref{NHFC} and (F3) in Theorem \ref{NHSF}.

\begin{figure}[bth]
\centering
\begin{minipage}[b]{0.45\textwidth}
\begin{tikzpicture}[scale=0.65]

  \draw[thick,magenta] plot[domain={-sqrt(15)}:{sqrt(15)},samples=50]
      ({sqrt(1+(\x)^2)},\x);
  \draw[thick,magenta] plot[domain={-sqrt(15)}:{sqrt(15)},samples=50]
      ({-sqrt(1+(\x)^2)},\x);
  \draw[thick,cyan] (2,4) -- (-2,-4);

    \draw[->] (-4,0) -- (4,0) node[right] {$x_1$};
    \draw[->] (0,-4) -- (0,4) node[above] {$x_2$};

  \node[cyan,above left] at (1,2) {\(\{g=0\}\)};
  \node[magenta,below right] at ({sqrt(2)},1) {\(\{f=0\}^+\)};
  \node[magenta,below left] at ({-sqrt(2)},1) {\(\{f=0\}^-\)};
  \node[below] at (1,0) {\( (1,0) \)};
  \node[below] at (-1,0) {\( (-1,0) \)};
  \node[cyan,below right] at (-1,-2) {\(2x_1-x_2=0\)};
  \node[magenta, above right] at ({-sqrt(10)},3) {\(-x_1^2+x_2^2+1=0\)};
  \node[magenta, below left] at ({sqrt(10)},-3) {\(-x_1^2+x_2^2+1=0\)};
\end{tikzpicture}
\caption{In Example \ref{eg:CF3}, \( \{g=0\}=\{2x_1-x_2=0\} \) separates \( \{f=0\}=\{-x_1^2+x_2^2+1=0\} \).}
\label{fig:C3'}
\end{minipage}
\hspace{0.06\textwidth}
\centering
\begin{minipage}[b]{0.45\textwidth}
\begin{tikzpicture}[scale=0.65]

    \fill[magenta!20] plot[domain={-sqrt(15)}:{sqrt(15)},samples=50]
      ({sqrt(1+(\x)^2)},\x) -- cycle;
    \fill[magenta!20] plot[domain={-sqrt(15)}:{sqrt(15)},samples=50]
      ({-sqrt(1+(\x)^2)},\x) -- cycle;
  \draw[thick,magenta] plot[domain={-sqrt(15)}:{sqrt(15)},samples=50]
      ({sqrt(1+(\x)^2)},\x);
  \draw[thick,magenta] plot[domain={-sqrt(15)}:{sqrt(15)},samples=50]
      ({-sqrt(1+(\x)^2)},\x);
  \draw[thick,cyan] (2,4) -- (-2,-4);

    \draw[->] (-4,0) -- (4,0) node[right] {$x_1$};
    \draw[->] (0,-4) -- (0,4) node[above] {$x_2$};

  \node[cyan,above left] at (1,2) {\(\{g=0\}\)};
  \node[magenta,below right] at ({sqrt(2)},1) {\(\{f\leq 0\}^+\)};
  \node[magenta,below left] at ({-sqrt(2)},1) {\(\{f\leq 0\}^-\)};
  \node[below] at (1,0) {\( (1,0) \)};
  \node[below] at (-1,0) {\( (-1,0) \)};
  \node[cyan,below right] at (-1,-2) {\(2x_1-x_2=0\)};
  \node[magenta, above right] at ({-sqrt(10)},3) {\(-x_1^2+x_2^2+1=0\)};
  \node[magenta, below left] at ({sqrt(10)},-3) {\(-x_1^2+x_2^2+1=0\)};
\end{tikzpicture}
\caption{In Example \ref{eg:CF3}, \( \{g=0\}=\{2x_1-x_2=0\} \) separates \( \{f\leq 0\}=\{-x_1^2+x_2^2+1\leq 0\} \).}
\label{fig:F3'}
\end{minipage}
\hspace{0.06\textwidth}
\begin{minipage}[b]{0.45\textwidth}
\centering
\begin{tikzpicture}[scale=0.65]
    \draw[->] (-4,0) -- (4,0) node[right] {$r$};
    \draw[->] (0,-4) -- (0,4) node[above] {$s$};
    \node[blue] at (2,2) {\((f,g)(\mathbb{R} ^2)\)};
    \node[below left, blue] at ({-1/3},2) {\( r=-\frac{s^2}{3} +1 \)};
    \node[below right] at (1,0) {\( 1 \)};

     \fill[blue!20,opacity=0.5] (-4,4) -- (4,4) -- (4,-4) -- (-4,-4) --
     plot[domain={-sqrt(15)}:{sqrt(15)} ,samples=25] ({-(\x)^2/3 +1},\x) -- cycle;
    \draw[blue,thick] plot[domain={-sqrt(15)}:{sqrt(15)} ,samples=25] ({-(\x)^2/3 +1},\x);
    \draw[draw=red, fill=red] (0,0) circle (3pt) node[red, below left] {\(\{(0,0)\}\)};
\end{tikzpicture}
\caption{The closed joint numerical range set of Example \ref{eg:CF3} has no intersection with \( \{(0,0)\} \).}
\label{fig:C3}
\end{minipage}
\hspace{0.06\textwidth}
\begin{minipage}[b]{0.45\textwidth}
\centering
\begin{tikzpicture}[scale=0.65]
    \draw[->] (-4,0) -- (4,0) node[right] {$r$};
    \draw[->] (0,-4) -- (0,4) node[above] {$s$};
    \node[red, below] at (-2,0) {\((-\infty ,0]\times \{0\}\)};
    \node[blue] at (2,2) {\((f,g)(\mathbb{R} ^2)\)};
    \node[below left, blue] at ({-1/3},2) {\( r=-\frac{s^2}{3} +1 \)};
    \node[below right] at (1,0) {\( 1 \)};

     \fill[blue!20,opacity=0.5] (-4,4) -- (4,4) -- (4,-4) -- (-4,-4) --
     plot[domain={-sqrt(15)}:{sqrt(15)} ,samples=25] ({-(\x)^2/3 +1},\x) -- cycle;
    \draw[blue,thick] plot[domain={-sqrt(15)}:{sqrt(15)} ,samples=25] ({-(\x)^2/3 +1},\x);
    \draw[thick,red] (-4,0) -- (0,0);
    \draw[draw=red, fill=red] (0,0) circle (2pt);
\end{tikzpicture}
\caption{The closed joint numerical range set of Example \ref{eg:CF3} has no intersection with \( (-\infty ,0]\times \{0\} \).}
\label{fig:F3}
\end{minipage}
\end{figure}
According to Proposition \ref{C2 prop}, two quadratic functions \( f \) and \( g \) have the separation property if and only if \( (0,0) \in \operatorname{conv}((f,g)(\mathbb{R} ^2)) \setminus (f,g)(\mathbb{R} ^2) \). In this example, the joint numerical range of \( f \) and \( g \) is
\begin{align*}
(f,g)(\mathbb{R} ^2)&=\{(-x_1^2+x_2^2+1,2x_1-x_2)|x_1,x_2\in \mathbb{R} \}\\
&=\{(-x_1^2+(2x_1-s)^2 +1,s) \mid x_1,s=2x_1-x_2\in \mathbb{R} \}\\
&=\{(3x_1^2-4sx_1+s^2+1,s) \mid x_1,s\in \mathbb{R} \}\\
&=\left \{(r,s) \, \middle | \, r\geq -\frac{s^2}{3} +1 \right \},
\end{align*}
as illustrated by the shadow areas in Fig.\ \ref{fig:C3} and Fig.\ \ref{fig:F3}, where \( (f,g)(\mathbb{R} ^2) \) is closed and
\begin{equation*}
(0,0)\notin \operatorname{cl} ((f,g)(\mathbb{R} ^2)) \text{ and }\operatorname{cl} ((f,g)(\mathbb{R} ^2)) \cap ((-\infty ,0]\times \{0\})=\emptyset .
\end{equation*}
Since \( \operatorname{conv}((f,g)(\mathbb{R} ^2))=\mathbb{R} ^2 \), we have
\begin{equation*}
(0,0)\in \operatorname{conv}(\operatorname{cl}((f,g)(\mathbb{R} ^2))) \text{ and} \operatorname{conv}(\operatorname{cl}((f,g)(\mathbb{R} ^2))) \cap ((-\infty ,0]\times \{0\})\neq \emptyset .
\end{equation*}
Proposition \ref{C2 prop} as well as Proposition \ref{F2 prop} are both justified in this example.
\end{example}

\section{Checkable conditions for [Calabi](nonhomo version) and [Str.-Finsler]({nonhomo version})}
\label{checkability}
In this section, we develop explicit formulas of Theorem \ref{NHFC} [Calabi](nonhomo version) and of Theorem \ref{NHSF} [Str.-Finsler]({nonhomo version}) that are amenable to computation. Specifically, conditions (C1), (C2), (C3) for checking \(
(\text{C0}):~ \{f=0\}\cap \{g=0\}=\emptyset 
\) are reduced to computationally implementable versions (C1-1), (C2-1), (C3-1) whereas conditions (F1), (F2), (F3) for checking \(
(\text{F0}):~ \{f\leq 0\}\cap \{g=0\}=\emptyset 
\) become (F1-1), (F2-1) and (F3-1).

In the following, we first address (C3‑1) and (F3‑1), both of which involve the separation property. Since the proof of (C3‑1) relies on properties of (F3‑1), we present (F3‑1) in Section 6.1, followed by (C3‑1) in Section 6.2. As for the pairs ``(C2‑1) and (F2‑1)'' and ``(C1‑1) and (F1‑1),'' we characterize them using matrix pencils, and these will be discussed in Sections 6.3 and 6.4, respectively.

\subsection{Checkable conditions (F3-1) for (F3): $\{g=0\}$ to separate $\{f \leq 0\}$}

Most of the results in this subsection rely on earlier developments by Nguyen, Chu and Sheu \cite{nguyen2024separating}, in which analytic conditions for the $0$-level set $\{g=0\}$ to separate $\{x\in \mathbb{R} ^n \mid f(x) \mathbin{\star } 0\}$ with $\star \in \{<, \leq , =, \geq , >\}$ were derived. In particular, we quote the result for $\{g=0\}$ to separate $\{x\in \mathbb{R} ^n \mid f(x) \leq 0\}$ in Theorem \ref{le sep}, and then provide an enhanced version of it in Theorem \ref{le sep+}.

\begin{theorem}
(Theorem 2 in \cite{nguyen2024separating})
\label{le sep}
The following are equivalent.
\begin{itemize}
\item[(F3)] The level set \( \{g = 0\} \) separates \( \{f \leq 0\} \).
\item[(F3-0)] \( B = 0 \), \( b \neq 0 \), \( A \) has exactly one negative eigenvalue, and
\[
 V^T AV\succeq 0, w\in \mathcal{R} (V^T AV),f(x_0)-w^T (V^T AV)^\dagger w >0,
\]
where \( w=V^T (Ax_0 + a) \), \( x_0=\frac{-b_0}{2b^T b}b \), \( V\in \mathbb{R} ^{n\times (n-1)} \) is the matrix basis of \( \mathcal{N} (b^T ) \).
\end{itemize}
\end{theorem}

In what follows, we replace the condition \( V^T AV \succeq 0 \) in (F3-0) of Theorem \ref{le sep} by 
\begin{equation}\label{no-VAV-1}
b\in \mathcal{R} (A) \text{ and } b^T A^\dagger b \leq 0.
\end{equation}
Notice that \eqref{no-VAV-1} is easier to verify than to check $V^T AV\succeq 0$ directly.

\begin{lemma}
\label{no VAV}
Suppose that \( A \) has exactly one negative eigenvalue and \( V\in \mathbb{R} ^{n\times (n-1)} \) is the matrix basis of \( \mathcal{N} (b^T ) \) with \( b\neq 0 \). Then, \( V^T A V \succeq 0 \) if and only if \( b\in \mathcal{R} (A) \) and \( b^T A^\dagger b \leq 0 \).
\end{lemma}

\begin{proof}
Since \( V \) is a matrix basis of \( \mathcal{N}(b^T )=\{y\in \mathbb{R} ^n \mid b^T y =0 \} \), we have
\begin{equation}
\label{VAV iff}
V^T AV\succeq 0 \Longleftrightarrow y^T Ay \geq 0, ~\forall y\in \mathcal{N}(b^T ).
\end{equation}

Let \( u_1,u_2,\ldots ,u_n \) be orthonormal eigenvectors of matrix $A$ with their respective eigenvalues \( \alpha _1,\alpha _2,\ldots ,\alpha _n \). Since $A$ has exactly one negative eigenvalue, we assume that \( \alpha _1 < 0 \), \( \alpha _i > 0 \) for \( 1<i\leq k \), and \( \alpha _i = 0 \) for \( k+1\leq i\leq n \). Then, the symmetric matrix $A$ has the following spectral decomposition:
\[
A=\sum _{i=1}^n \alpha _i u_i u_i^T =\sum _{i=1}^k \alpha _i u_i u_i^T .
\]
In the case where $k=n,$ $A$ is of full rank and it has no $0$-eigenvalue.

Using the orthonormal eigenvectors \( u_1,u_2,\ldots ,u_n \) as a basis of $\mathbb{R} ^n,$ the vector $b$ can be expressed as 
\[
b=\sum _{i=1}^n \beta _i u_i\ne 0,~\beta _i\in \mathbb{R} ,~i=1,2,\ldots ,n.
\]
With the setting, we first prove the necessary part of Lemma \ref{no VAV} by contradiction with the following two cases (i) and (ii).
\begin{itemize}
\item[(i)] Suppose that \( V^T AV\succeq 0 \) and \( b\notin \mathcal{R} (A)=\operatorname{span} \{u_1,u_2,\ldots ,u_k\} \). The case could occur only when $A$ is rank deficient, so that $k<n$, and at least one of the coefficients $\beta _{k+1}, \beta _{k+2},\ldots , \beta _{n}$ of vector $b$ corresponding to the zero eigenvalue of $A$ is nonzero. Let us assume that \( \beta _n\neq 0 \) and consider \( y = \beta _n u_1 - \beta _1 u_n \). Then,
\[
b^T y = \left (\sum _{i=1}^n \beta _i u_i \right )^T (\beta _n u_1 -\beta _1 u_n )= \beta _1 \beta _n - \beta _n \beta _1 = 0.
\]
Namely, $y\in \mathcal{N}(b^T )$. In addition,
\begin{align*}
y^T Ay &= (\beta _n u_1 - \beta _1 u_n)^T \left(\sum _{i=1}^n \alpha _i u_i u_i^T \right) (\beta _n u_1 - \beta _1 u_n)\\
&=(\beta _n u_1 - \beta _1 u_n)^T \left( \alpha _1\beta _nu_1-\alpha _n\beta _1 u_n \right)\\
&= \alpha _1\beta _n^2 + \alpha _n\beta _1^2< 0,\text{ since } \alpha _1<0, \beta _n\neq 0, \text{ and } \alpha _n=0.
\end{align*}
By \eqref{VAV iff}, \( V^T AV\nsucceq 0 \), which is a contradiction.

\item[(ii)] Suppose that \( V^T AV\succeq 0 \), \( b\in \mathcal{R} (A) \) but \( b^T A^\dagger b >0 \). Since \( u_1^T A u_1 = \alpha _1 <0 \), by \eqref{VAV iff}, we have \( b^T u_1\neq 0 \). Consider
\begin{equation}
\label{y with u1}
w = A^\dagger b - \frac{b^T A^\dagger b}{b^T u_1} u_1 \text{ so that } b^T w= b^T A^\dagger b - \frac{b^T A^\dagger b}{b^T u_1} b^T u_1 = 0.
\end{equation}
Then, $w\in \mathcal{N}(b^T )$. 
Notice that \( b\in \mathcal{R} (A) \) implies that \( AA^\dagger b=b \). By direct calculation, we have
\begin{eqnarray}
w^T Aw &=& \left (A^\dagger b - \frac{b^T A^\dagger b}{b^T u_1} u_1 \right )^T A \left (A^\dagger b - \frac{b^T A^\dagger b}{b^T u_1} u_1 \right )\nonumber\\
&=& b^T A^\dagger A A^ \dagger b - 2 \frac{b^T A^\dagger b}{b^T u_1} u_1^T AA^\dagger b + \left (\frac{b^T A^\dagger b}{b^T u_1}\right )^2 u_1^T A u_1\nonumber\\
&=& - b^T A^\dagger b + \left (\frac{b^T A^\dagger b}{b^T u_1}\right )^2 \alpha _1 \nonumber\\
&<& 0 \quad (\text{since } b^T A^\dagger b >0 \text{ by assumption, and } b^T u_1\neq 0, \alpha _1 < 0).\label{direct calc yAy}
\end{eqnarray}
Combining \eqref{VAV iff}, \eqref{y with u1} and \eqref{direct calc yAy} again leads to the contradiction \( V^T AV\nsucceq 0 \). The necessary part is thus completed.
\end{itemize}

Conversely, we prove the sufficiency of the theorem under the following assumptions:
\begin{equation}\label{tree-5}
0\ne b\in \mathcal{R} (A), ~b^T A^\dagger b \leq 0,
\end{equation}
and aim to show that
\[
\forall y\in \mathcal{N}(b^T )~\Rightarrow ~y^T Ay\geq 0.
\]

Denote any $y\in \mathcal{N}(b^T )$ as \( y=\sum _{i=1}^n \gamma _i u_i \).
Since $b\in \mathcal{R} (A)$ and $y\in \mathcal{N}(b^T ),$ we have
\begin{equation}\label{bTy=0}
b=\sum _{i=1}^k \beta _i u_i\text{ and } b^T y = \sum _{i=1}^k \beta _i \gamma _i = 0,~k\leq n.
\end{equation}
Notice that
\begin{eqnarray}
0 &\geq & b^T A^\dagger b\nonumber\\
&=& \left (\sum _{i=1}^k \beta _i u_i \right)^T \left (\sum _{i=1}^k \alpha _i^{-1} u_i u_i^T \right) \left (\sum _{i=1}^k \beta _i u_i \right) \quad (\text{since } \alpha _i=0, ~\forall \, k+1\leq i\leq n  )\nonumber\\
&=&\sum _{i=1}^k \frac{\beta ^2_i}{\alpha _i}.
\label{w ge ana}
\end{eqnarray}
Since $\alpha _1<0$ and \( \alpha _i>0 \) for \( 2\leq i\leq k \), \eqref{w ge ana} implies that \( \beta _1\neq 0 \). Otherwise, if $\beta _1= 0$, \eqref{w ge ana} would force \( \beta _i=0 \) for all \( 1\leq i\leq k \), a contradiction to the assumption in \eqref{tree-5} that \( b\neq 0 \).

Rearrange the term for $b^T y$ in \eqref{bTy=0} and take the square. We obtain
\begin{eqnarray}
(-\beta _1 \gamma _1)^2 & =& \left (\sum _{i=2}^k \beta _i \gamma _i \right )^2 \quad \nonumber\\
& \leq & \left(\sum _{i=2}^k \frac{\beta _i^2}{\alpha _i}\right)\left(\sum _{i=2}^k \alpha _i \gamma _i^2\right) \quad (\text{by Cauchy-Schwarz inequality})\nonumber\\
& \leq & \left(-\frac{\beta _1^2}{\alpha _1}\right) \left (\sum _{i=2}^k \alpha _i \gamma _i^2 \right ). \quad (\text{by }\eqref{w ge ana})\label{Cauchy-Sch}
\end{eqnarray}

Recall that \( \alpha _1<0 \) and $\beta _1\ne 0.$ We then cancel $\beta _1^2$ from both sides of \eqref{Cauchy-Sch} and swap $-\alpha _1$ to the left hand side to obtain
\begin{equation}
\label{both devide}
-\alpha _1 \gamma _1^2 \leq \sum _{i=2}^k \alpha _i \gamma _i^2~ \Rightarrow \sum _{i=1}^k \alpha _i \gamma _i^2~ \geq 0.
\end{equation}
Hence,
\begin{eqnarray*}
y^T Ay &=& \left ( \sum _{i=1}^n \gamma _i u_i^T \right ) \left ( \sum _{i=1}^n \alpha _i u_i u_i^T \right ) \left ( \sum _{i=1}^n \gamma _i u_i \right )\nonumber\\
&=&\left ( \sum _{i=1}^n \gamma _i u_i^T \right ) \left ( \sum _{i=1}^n \alpha _i\gamma _i u_i \right )\nonumber\\
&=& \sum _{i=1}^n \alpha _i \gamma _i^2=\sum _{i=1}^k \alpha _i \gamma _i^2\geq 0. \quad (\text{by } \eqref{both devide}~{\rm and}~\alpha _i=0,~i=k+1,\ldots ,n)
\end{eqnarray*}
The proof for the sufficiency as well as the entire Lemma \ref{no VAV} are thus complete.
\end{proof}

Incorporating Lemma \ref{no VAV} with Theorem \ref{le sep}, we have the following result.

\begin{theorem}
\label{le sep+}
The following are equivalent.
\begin{itemize}
\item[(F3)] The level set \( \{g = 0\} \) separates \( \{f \leq 0\} \).
\item[(F3-1)] \( B = 0 \), \( b \neq 0 \), \( A \) has exactly one negative eigenvalue,
\[
b\in \mathcal{R} (A), b^T A^\dagger b \leq 0, w\in \mathcal{R} (V^T AV),f(x_0)-w^T (V^T AV)^\dagger w >0,
\]
where \( w=V^T (Ax_0 + a) \), \( x_0=\frac{-b_0}{2b^T b}b \), \( V\in \mathbb{R} ^{n\times (n-1)} \) is the matrix basis of \( \mathcal{N} (b^T ) \).
\end{itemize}
\end{theorem}

Algorithm \ref{alg:F3} in Appendix \ref{'en-s'ng-hu8at} enumerates all the computational steps involved in verifying (F3-1).

\subsection{Checkable conditions (C3-1) for (C3): the separation property}

Recall that the separation property (C3) is defined as: either \( \{f = 0\} \) separates \( \{g = 0\} \), or \( \{g = 0\} \) separates \( \{f = 0\} \). According to Lemma \ref{= = sep} and Lemma \ref{= = aff} below, the separation property reduces to the existence of a hyperplane $\{h=0\}$ that separates either $\{f\leq 0\}$ or $\{f\geq 0\}$, or the existence of a hyperplane \( \{l=0\} \) that separates either \( \{g\leq 0\} \) or \( \{g\geq 0\} \).

\begin{lemma}
\label{= = sep}
(\cite{nguyen2024separating} Theorem 4)\\
Let \( f(x) = x^T Ax + 2 a^T x + a_0 \) and \( g(x) = x^T Bx + 2 b^T x + b_0 \). Then, \( \{g=0\} \) separates \( \{f=0\} \) if and only if there exists \( \zeta \in \mathbb{R} \) such that \( B = \zeta A \) and the hyperplane \( \{- \zeta f + g = 0\} \) separates \( \{f = 0\} \).
\end{lemma}

\begin{lemma}
\label{= = aff}
(\cite{nguyen2024separating} Theorem 3)\\
Let \( f(x) = x^T Ax + 2 a^T x + a_0 \) and \( h(x) = c^T x + c_0 \). The hyperplane \( \{h = 0\} \)
separates \( \{f = 0\} \) if and only if either (i) \( \{h = 0\} \) separates \( \{f \leq 0\} \); or (ii) \( \{h = 0\} \) separates \( \{-f \leq 0\} \).
\end{lemma}

Applying Lemma~\ref{= = sep} and Lemma~\ref{= = aff} to verify the separation property (C3) amounts to invoking Theorem~\ref{le sep+} four times, namely to check whether 
\begin{itemize}
\item[(i)] either \( \{h = 0\} \) separates \( \{f \leq 0\} \), or
\( \{h = 0\} \) separates \( \{-f \leq 0\} \);
\end{itemize}
or whether there exists a constant \( \tau \in \mathbb{R} \) such that \( A=\tau B \); \( l(x)=-\tau g + f \) and 
\begin{itemize}
\item[(ii)] either \( \{l = 0\} \) separates \( \{g \leq 0\} \), or
\( \{l = 0\} \) separates \( \{-g \leq 0\} \).
\end{itemize}
If any of the four possibilities in (i) and (ii) is satisfied, the quadratic pair 
$(f,g)$ possesses the separation property. The following lemma offers an alternative that checks all of them in a single step.

\begin{lemma}
\label{sep prop aff}
The two quadratic functions \( f \) and \( g \) have the separation property if and only if there exists \( (\sigma ,\tau )\neq (0,0) \) such that \( \sigma f + \tau g \) is an affine function and the hyperplane \( \{\sigma f+ \tau g=0\} \) separates \( \{\tau f - \sigma g \leq 0\} \).
\end{lemma}

\begin{proof}
\textbf{(Sufficiency)} Suppose that \((\sigma ,\tau )\neq (0,0)\) is a pair of constants so that \(\sigma f + \tau g\) is affine and the hyperplane \( \{\sigma f + \tau g = 0\} \) separates the set \( \{\tau f - \sigma g \leq 0\} \). According to Lemma \ref{= = aff}, the hyperplane \( \{\sigma f + \tau g = 0\} \) also separates the quadratic $0$-level set \( \{\tau f - \sigma g = 0\} \). By Lemma \ref{linear comb sep}, since $\tau \tau -\sigma (-\sigma )=(\tau ^2+\sigma ^2)\neq 0$, \( f \) and \( g \) have the separation property.

\noindent \textbf{(Necessity)} Assume that \( f \) and \( g \) have the separation property. By Definition \ref{def sep prop}, one of the level sets \(\{f=0\}\) or \(\{g=0\}\) separates the other. Without loss of generality, let us assume that \(\{g=0\}\) separates \(\{f=0\}\) as the other case when \(\{f=0\}\) separates \(\{g=0\}\) can be similarly proved. Then, by Lemma \ref{= = sep}, there exists a scalar \(\zeta \in \mathbb{R} \) such that the function \( -\zeta f + g \) is affine and the hyperplane \( \{-\zeta f + g = 0\} \) separates \(\{f=0\}\). By writing
\[
\{-\zeta f + g = 0\} = \{(-\zeta f + g)+0\cdot (f + \zeta g) = 0\}
\]
and
\[
\{f=0\}=\left \{-\frac{\zeta }{1+\zeta ^2}(-\zeta f+g)+\frac 1{1+\zeta ^2} (f+\zeta g) =0 \right \},
\]
and noting that
\[
-\frac{\zeta }{1+\zeta ^2}\cdot 0-\frac 1{1+\zeta ^2}\cdot 1 \neq 0,
\]
we know, by Lemma \ref{cor 2.9 ng}, that the hyperplane \( \{-\zeta f + g = 0\} \) also separates the quadratic sublevel set \(\{f+\zeta g=0\}\).
Applying Lemma \ref{= = aff} again, the hyperplane \( \{-\zeta f + g=0\} \) separates 
\[\text{either } \{f+\zeta g\leq 0\} \text{ or }\{-f-\zeta g\leq 0\}. \] 
In the former case, we can take \(\sigma =-\zeta \) and \(\tau =1\). As for the latter case, we can write the hyperplane \( \{-\zeta f + g=0\} \) as \( \{\zeta f-g=0\} \), which separates the quadratic level set \( \{-f-\zeta g\leq 0\} \). Then, take \(\sigma =\zeta \) and \(\tau =-1\) to complete the proof of the case.

\end{proof}

To incorporate Theorem \ref{le sep+} with Lemma \ref{sep prop aff}, we first denote 
\begin{align*}
\sigma f(x) + \tau g(x) &= x^T \underbrace {(\sigma A + \tau B)}_{=0\text{ since affine}} x + 2 \underbrace {(\sigma a + \tau b)^T }_{=c^T \neq 0} x + (\sigma a_0 + \tau b_0)\\
\tau f(x) - \sigma g(x) &= x^T \underbrace {(\tau A - \sigma B)}_{=D} x + 2 \underbrace {(\tau a - \sigma b)^T }_{=d^T } x + (\tau a_0 - \sigma b_0)
\end{align*}
to have Theorem \ref{= sep+} below for checking the separation property.

\begin{theorem}
\label{= sep+}
The following are equivalent.
\begin{itemize}
\item[(C3)] \( f(x) = x^T Ax + 2a^T x + a_0 \) and \( g(x) = x^T Bx + 2b^T x + b_0 \) have the separation property.
\item[(C3-1)] There exists \( (\sigma ,\tau )\neq (0,0) \) such that \( \sigma A + \tau B = 0 \), \( c:=\sigma a + \tau b \neq 0 \) and the matrix \( D:=\tau A - \sigma B \) has exactly one negative eigenvalue. Moreover,
{\small
\begin{equation}\label{C3-1-equations}
 c\in \mathcal{R} (D), ~c^T D^\dagger c \leq 0,~ w\in \mathcal{R} (V^T DV),~\tau f(x_0)-\sigma g(x_0)-w^T (V^T DV)^\dagger w >0,
\end{equation}}
where \( d=\tau a-\sigma b \), \( w=V^T (Dx_0 + d) \), \( x_0=\frac{-\sigma a_0 - \tau b_0}{2c^T c}c \) and \( V\in \mathbb{R} ^{n\times (n-1)} \) is the matrix basis of \( \mathcal{N} (c^T ) \). 
\end{itemize}
\end{theorem}

To further analyze what \( (\sigma , \tau )\neq (0,0) \) in (C3-1) may satisfy \eqref{C3-1-equations}, we notice that \( \sigma A + \tau B = 0 \) implies that \( A \) and \( B \) are linearly dependent, while \( D = \tau A - \sigma B \) having exactly one negative eigenvalue indicates that \( A \) and \( B \) cannot be both zero matrices. Therefore, we divide the discussion into two cases: (i) \( A = 0 \) and (ii) \( A \neq 0 \). 

\begin{itemize}
\item[(i)] The matrix $A=0.$ In this case, $B\ne 0.$ Due to \( \sigma A + \tau B = 0 \), all possible \( (\sigma ,\tau )\neq (0,0) \) satisfying (C3-1) must fall inside the set $\mathcal{M} =\{(\sigma ,\tau ):~\sigma \neq 0,~\tau =0\}$. That is, $\mathcal{M} $ is the union of two open rays originating at $(0,0)$:
\begin{equation}\label{two-open-rays-1}
\mathcal{M} =\{r(1,0):r>0\}\cup \{s(-1,0):s>0\}.
\end{equation}
\item[(ii)] The matrix $A\ne 0.$ Pick an arbitrary index \( (i_0,j_0) \) such that \( A_{i_0 j_0}\neq 0 \) and define \( \zeta =\frac{B_{i_0 j_0}}{A_{i_0 j_0}} \). If there exists \( B_{ij}\neq \zeta A_{ij} \) for some index \( (i,j) \), (C3-1) fails and thus \( f \) and \( g \) do not have the separation property. Otherwise, \( B_{ij}=\zeta A_{ij} \) for all indices \( (i,j) \) and \( B=\zeta A \). Namely, $-\zeta A+B=0.$ Then, the set of all possible \( (\sigma ,\tau )\neq (0,0) \) satisfying (C3-1) falls inside 
\begin{equation}\label{two-open-rays-2}
\mathcal{M} =\{r(-\zeta ,1):r>0\}\cup \{s(\zeta ,-1):s>0\},
\end{equation}
which is again the union of two open rays originated from $(0,0)$.
\end{itemize}

\begin{proposition}[positive homogeneity]\label{ray-invariant} If (C3-1) in Theorem \ref{= sep+} holds for \( (\sigma ,\tau )\neq (0,0) \), it also holds for \( (t\sigma , t\tau ),~\forall t>0 \).
\end{proposition}

\begin{proof}
For \( t>0 \), let us denote \( {\bar c}:=(t\sigma ) a + (t\tau ) b \) and \( {\bar D}:=(t\tau ) A - (t\sigma ) B \) so that \( {\bar c}=tc,~{\bar D}=tD\). Clearly, \( \sigma A + \tau B = 0 \) implies that \( (t\sigma ) A + (t\tau ) B = 0 \). Since $c\ne 0$ and $D$ has exactly one negative eigenvalue, it follows that ${\bar c}=tc\ne 0$ and ${\bar D}$ also has exactly one negative eigenvalue. 

As for \eqref{C3-1-equations}, since $\mathcal{R} (D)=\mathcal{R} (tD)=\mathcal{R} (\bar D)$ is a subspace, 
$$c\in \mathcal{R} (D) \Rightarrow {\bar c}=tc\in \mathcal{R} (D)=\mathcal{R} (\bar D).$$
In addition, $c^T D^\dagger c \leq 0$ implies that 
$$~{\bar c}^T {\bar D}^\dagger {\bar c}=(tc)^T \frac{1}{t}D^\dagger (tc)=t\left(c^T D^\dagger c\right) \leq 0.$$
Notice that the null space \( \mathcal{N} (c^T ) \) coincides with \( \mathcal{N} ({\bar c}^T ) \). Moreover, $x_0$ remains unchanged, since
$$ \frac{-(t\sigma ) a_0 - (t\tau ) b_0}{2{\bar c}^T {\bar c}}{\bar c}=\frac{-(t\sigma ) a_0 - (t\tau ) b_0}{2t^2c^T c}t c=\frac{-\sigma a_0 - \tau b_0}{2c^T c}c=x_0. $$
Therefore, we can define 
$${\bar d}=(t\tau ) a-(t\sigma ) b=td \text{ and }\bar w=V^T ({\bar D}x_0 + {\bar d})=t\left(V^T (Dx_0 + d)\right)=tw,$$
so that $\bar w\in \mathcal{R} (V^T {\bar D}V)$, due to
$$w\in \mathcal{R} (V^T DV)~\Rightarrow ~ \bar w=tw\in \mathcal{R} (V^T DV)=
\mathcal{R} (V^T (tD)V)=\mathcal{R} (V^T {\bar D}V).$$ Finally, we verify that 
\begin{eqnarray*}
&&(t\tau ) f(x_0)-(t\sigma ) g(x_0)-(tw)^T (V^T (tD)V)^\dagger (tw)\\
&=&t\left(\tau f(x_0)-\sigma g(x_0)-w^T (V^T DV)^\dagger w\right)>0.
\end{eqnarray*}
The proof for Proposition \ref{ray-invariant} is complete.
\end{proof}

Incorporating \eqref{two-open-rays-1} and \eqref{two-open-rays-2} with Proposition \ref{ray-invariant}, we find that there are only two possible candidates \( (\sigma ,\tau )\neq (0,0) \) which may satisfy (C3-1), depending on \( A=0 \) or \( A\neq 0 \).

\begin{proposition}
The verification of \( (\sigma ,\tau )\neq (0,0) \) satisfying (C3-1) reduces to checking only $(1,0)$ and $(-1,0)$ when $A=0,~B\ne 0;$ while checking only $(-\zeta ,1), (\zeta ,-1)$
when $A\ne 0$ satisfies $-\zeta A+B=0.$
\end{proposition}

Then, if any of the candidates satisfies \eqref{C3-1-equations}, \( f \) and \( g \) have the separation property. Otherwise, the separation property fails for \( f \) and \( g \). See Algorithm \ref{alg:C3} in Appendix \ref{'en-s'ng-hu8at} for the entire computational procedure to verify (C3).

\subsection{Matrix pencil conditions (C2-1) and (F2-1) for (C2) and (F2)}

The key to verifying (C2) is to determine multipliers
 \( (\lambda ,\mu )\neq (0,0) \) such that \( \lambda f(x)+\mu g(x)\geq 0, \forall x\in \mathbb{R} ^n \). In contrast, verifying (F2) requires finding a scalar \( \xi \in \mathbb{R} \) such that \( f(x)+\xi g(x)\geq 0, \forall x\in \mathbb{R} ^n\). In both cases, the Lagrangian function attains the minimum value \( 0 \). Moreover, the ``conjugate'' Lagrangian function \( (\mu f-\lambda g)\) (in (F2), simply \( g \)) is either strictly positive or strictly negative on the minimum solution set of the respective Lagrangian function.

In this subsection, following \cite{hsia2014revisit}, we characterize the multipliers satisfying these conditions by means of matrix pencils. The computation of the resulting matrix pencils is discussed in Section~\ref{comp II(F,G)}.
 
Let \( f(x) = x^T Ax + 2a^T x + a_0 \) and \( g(x) = x^T Bx + 2b^T x + b_0 \). Define $(n+1)\times (n+1)$ symmetric matrices:
\begin{equation}
\label{defn F,G}
F =
\begin{bmatrix}
A & a\\
a^T & a_0
\end{bmatrix},~
G =
\begin{bmatrix}
B & b\\
b^T & b_0
\end{bmatrix}.
\end{equation}
According to the Schur complement (see, e.g., \cite{boyd2004convex} A.5.4),
\begin{equation}\label{Schur}
f(x)\geq 0,~\forall x\in \mathbb{R} ^n~\Longleftrightarrow ~F\succeq 0.
\end{equation}
Moreover,
\begin{equation}\label{min-f}
\min_{x\in \mathbb{R} ^n} f(x) =
\begin{cases}
a_0 - a^T A^\dagger a, & \text{if } A\succeq 0 \text{ and } a\in \mathcal{R} (A);\\
-\infty , & \text{otherwise},
\end{cases}
\end{equation}
so that $f(x)\geq 0,~\forall x\in \mathbb{R} ^n$ if and only if \( A\succeq 0 \), \( a\in \mathcal{R} (A) \), and \(a_0 - a^T A^\dagger a \geq 0 \); and $f(x)>0,~\forall x\in \mathbb{R} ^n$ if and only if \( A\succeq 0 \), \( a\in \mathcal{R} (A) \), and \(a_0 - a^T A^\dagger a > 0 \).

\begin{definition}
Let \( F , G \) be as in \eqref{defn F,G}. Define two kinds of matrix pencils of \( F \) and \( G \) as
\begin{align}
II_\succeq (F,G) &=\{(\lambda , \mu )\in \mathbb{R} ^2 \mid \lambda F +\mu G \succeq 0\};\label{II(A,B)}\\
I_\succeq (F,G) &=\{\xi \in \mathbb{R} \mid F + \xi G \succeq 0\}.\label{I(A,B)}
\end{align}
\end{definition}

Note that both \( II_\succeq (F,G) \) and \( I_\succeq (F,G) \) are convex sets, since they are pre-images of positive semi-definite cones under affine functions. In particular, \( II_\succeq (F,G) \) is a convex cone having non-empty relative interior, whereas the set \( I_\succeq (F,G) \) is a convex interval which might be unbounded or empty. See reference \cite{more1993generalizations}.

By \eqref{Schur}, matrix pencils \eqref{II(A,B)} and \eqref{I(A,B)} can be expressed as the non-negativity of Lagrangian functions as
\begin{align}
II_\succeq (F,G) &=\{(\lambda , \mu )\in \mathbb{R} ^2 \mid \lambda f(x) + \mu g(x) \geq 0, \;\forall \, x\in \mathbb{R} ^n\}; \label{II(f,g)} \\
I_\succeq (F,G) &=\{\xi \in \mathbb{R} \mid f(x) + \xi g(x) \geq 0, \;\forall \, x\in \mathbb{R} ^n\}.
\end{align}
According to Lemma \ref{nonclosed} in Subsection 3.3 , the multiplier \( (\lambda ,\mu )\neq (0,0) \) such that \( \lambda f(x)+\mu g(x)\geq 0,\,\forall x\in \mathbb{R} ^n \), if exists, is unique up to multiplication by a positive constant and thus the corresponding matrix pencil \( II_\succeq (F,G) \) is a ray satisfying (C2). On the other hand, Proposition \ref{tsu^an kh`o} in Subsection 4.2 asserts that the multiplier \( \xi \in \mathbb{R} \) such that \( f(x)+\xi g(x)\geq 0,\,\forall x\in \mathbb{R} ^n \), if exists, must be unique. In that case, the matrix pencil \( I_\succeq (F,G) \) satisfying (F2) is a singleton.

Moreover, since the Lagrangian function \( \lambda f(x)+\mu g(x)\geq 0\) attains its minimal value \( 0 \), for \( (\lambda ,\mu )\in II_\succeq (F,G) \), 
\[
\lambda a_0+\mu b_0 - (\lambda a+\mu b)^T (\lambda A+\mu B)^\dagger (\lambda a+\mu b) = 0.
\]
The the set of minimizers of \( \lambda f(x)+\mu g(x) \), due to convexity, collects all the points that satisfy the first order condition:
\[
\mathcal W = \{\nabla (\lambda f+\mu g)=0\}=\{x\in \mathbb{R} ^n\mid (\lambda A+\mu B)x+(\lambda a+\mu b)=0\}.
\]
In terms of the null space representation
\[
\mathcal W = u+\mathcal N(\lambda A+\mu B)=\{u+Uy\mid y\in \mathbb{R} ^k\},
\]
there is \( u:=-(\lambda A+\mu B)^\dagger (\lambda a+\mu b) \), and \( U\in \mathbb{R} ^{n\times k} \) is a matrix basis of \( \mathcal N(\lambda A+\mu B) \). As the conjugate Lagrangian function \(\mu f-\lambda g \), when restricted to \( \mathcal W \), must be either positive or negative, it follows that the quadratic function
\begin{align*}
(\mu f-\lambda g)|_{\mathcal W} (x)&= (u+Uy)^T C(u+Uy) + 2c^T (u+Uy) +c_0\\
& =y^T (U^T CU)y+2(Cu+c)^T U y+(\mu f(u)-\lambda g(u)),
\end{align*}
with \( C=\mu A-\lambda B \), \( c=\mu a-\lambda b \), and \( c_0=\mu a_0-\lambda b_0 \) must either satisfy \( U^T CU \succeq 0 \), \( U^T (Cu+c)\in \mathcal{R} (U^T CU) \), and 
\[
(\mu f(u)-\lambda g(u))- (Cu+c)^T U (U^T CU)^\dagger U^T (Cu+c) >0,
\]
or satisfy \( U^T CU \preceq 0 \), \( U^T (Cu+c)\in \mathcal{R} (U^T CU) \), and 
\[
(\mu f(u)-\lambda g(u))- (Cu+c)^T U (U^T CU)^\dagger U^T (Cu+c) <0.
\]
Theorem \ref{(C2-1)} below presents the matrix pencil version (C2-1) of (C2). 

\begin{theorem}
\label{(C2-1)}
(C2) has the following matrix pencil version (C2-1).
\begin{itemize}
\item[(C2-1)] \( II_\succeq (F,G) \) is a ray, namely \( II_\succeq (F,G)=\{t(\lambda ,\mu ) \mid t\geq 0\} \), with some \( (\lambda ,\mu )\neq (0,0) \) such that
\[
(\lambda a_0+\mu b_0)-(\lambda a+\mu b)^T (\lambda A+\mu B)^\dagger (\lambda a+\mu b) = 0.
\]
Moreover, either one of the following holds.
\begin{itemize}
\item \( U^T CU \succeq 0 \), \( U^T (Cu+c)\in \mathcal{R} (U^T CU) \), and
\[
(\mu f(u)-\lambda g(u))- (Cu+c)^T U(U^T CU)^\dagger U^T (Cu+c) >0;
\]
\item \( U^T CU \preceq 0 \), \( U^T (Cu+c)\in \mathcal{R} (U^T CU) \), and
\[
(\mu f(u)-\lambda g(u))- (Cu+c)^T U(U^T CU)^\dagger U^T (Cu+c) <0,
\]
\end{itemize}
where \( C=\mu A-\lambda B \), \( c=\mu a-\lambda b \), \( U \) is a matrix basis of \( \mathcal N(\lambda A+\mu B) \), and \( u=-(\lambda A+\mu B)^\dagger (\lambda a+\mu b) \).
\end{itemize}
\end{theorem}

The matrix pencil version (F2‑1) for (F2) in strict Finsler Lemma can be similarly formulated.

\begin{theorem}
\label{(F2-1)}
(F2) has the following matrix pencil version (F2-1).
\begin{itemize}
\item[(F2-1)] \( I_\succeq (F,G) \) is a singleton, namely \( I_\succeq (F,G)=\{\xi \} \), with \( \xi \in \mathbb{R} \) so that
\[
(a_0+\xi b_0)-(a+\xi b)^T (A+\xi B)^\dagger (a+\xi b) = 0.
\]
Moreover, either one of the following holds.
\begin{itemize}
\item \( U^T BU \succeq 0 \), \( U^T (Bu+b)\in \mathcal{R} (U^T BU) \), and
\[
g(u)- (Bu+b)^T U(U^T BU)^\dagger U^T (Bu+b) >0;
\]
\item \( U^T BU \preceq 0 \), \( U^T (Bu+b)\in \mathcal{R} (U^T BU) \), and
\[
g(u)- (Bu+b)^T U(U^T BU)^\dagger U^T (Bu+b) <0,
\]
\end{itemize}
where \( U \) is a matrix basis of \( \mathcal N(A+\xi B) \) and \( u=-(A+\xi B)^\dagger (a+\xi b) \).
\end{itemize}
\end{theorem}

\subsection{Matrix pencil versions (C1-1) and (F1-1) of (C1) and (F1)}

To derive the matrix pencil conditions for the existence of Lagrange multipliers \( (\lambda ,\mu )\in \mathbb{R} ^2 \) and a scalar \( \xi \in \mathbb{R} \) satisfying, respectively,
\[
(\text{C1})\quad \lambda f(x)+\mu g(x)>0 ,\, \forall x\in \mathbb{R} ^n,
\]
and
\[
(\text{F1})\quad f(x)+\xi g(x)>0 ,\, \forall x\in \mathbb{R} ^n,
\]
we prove in Lemma \ref{u_ann-tu_a} that it suffices to test an arbitrary \((\lambda ,\mu )\in \operatorname{relint}(II_\succeq (F,G))\) and \(\xi \in \operatorname{relint}(I_\succeq (F,G))\), where ``\( \operatorname{relint} \)'' stands for the relative interior. Certainly, \( (\lambda ,\mu ) \) satisfying (C1) must belong to \( II_\succeq (F,G) \), while \( \xi \) satisfying (F1) belongs to \( I_\succeq (F,G) \).

\begin{lemma}
\label{u_ann-tu_a} 
\mbox{}
\begin{enumerate}
    \item If (C1) holds for some \( (\lambda ,\mu )\in \mathbb{R} ^2 \), then (C1) must also hold for every \( (\lambda ',\mu ')\in \operatorname{relint} (II_\succeq (F,G)) \). On the other hand, if (C1) fails for every \( (\lambda ,\mu )\in \mathbb{R} ^2 \), there is certainly no \( (\lambda ',\mu ')\in \operatorname{relint} (II_\succeq (F,G)) \) that can make (C1) hold.
\item If (F1) holds for some \( \xi \in \mathbb{R} \), then (F1) must also hold for every \( \xi '\in \operatorname{relint} (I_\succeq (F,G)) \). If (F1) fails for all \( \xi \in \mathbb{R} \), no point \( \xi '\in \operatorname{relint} (I_\succeq (F,G)) \) can satisfy (F1).
\end{enumerate}

\end{lemma}

\begin{proof}
\leavevmode
\begin{enumerate}
\item
We prove by contradiction. Suppose that (C1) holds for some \( (\lambda ,\mu )\in \mathbb{R} ^2 \) but there exists \( (\lambda ',\mu ')\in \operatorname{relint} (II_\succeq (F,G)) \) such that \( \lambda ' f(x_0) +\mu ' g(x_0) \leq 0 \) for some \( x_0\in \mathbb{R} ^n \). Let \( V \) be the affine hull of \( II_\succeq (F,G) \) and \( U \) be a neighborhood of \( (\lambda ',\mu ') \) such that
\[
V\supset II_\succeq (F,G) \text{ and } U\cap V\subset II_\succeq (F,G).
\]
Consider the line in \( \mathbb{R} ^2 \) passing through \( (\lambda ,\mu ) \) and \( (\lambda ',\mu ') \) as
\[
\ell (t) = (r(t),s(t)):= (\lambda '+(\lambda -\lambda ')t,\mu '+(\mu -\mu ')t )
\]
with \( \ell (0)=(\lambda ',\mu ') \) and \( \ell (1)=(r(1),s(1))=(\lambda ,\mu ) \). Obviously, \( \ell (\mathbb{R} )\subset V \). Define an affine function 
\[
c(t) = r(t) f(x_0) +s(t) g(x_0)
\]
with \( c(0)=\lambda 'f(x_0)+\mu 'g(x_0)\leq 0 \). Since \( U \) is an open neighborhood of \( \ell (0)=(\lambda ',\mu ') \), there exists \( \varepsilon >0 \) such that the image of \( \ell \) on the interval \( (-\varepsilon ,\varepsilon ) \) is contained in \( U \). That is, for \( |t|<\varepsilon \),
\[
\ell (t)=(r(t),s(t))\in U\cap V\subset II_\succeq (F,G),
\]
and hence,
\[
\begin{aligned}
c(t) = r(t) f(x_0) +s(t) g(x_0)
&= r(t)
\begin{bmatrix}x_0^T &1\end{bmatrix} F \begin{bmatrix}x_0 \\ 1\end{bmatrix}
+ s(t)
\begin{bmatrix}x_0^T &1\end{bmatrix} G \begin{bmatrix}x_0 \\ 1\end{bmatrix}\\
&=
\begin{bmatrix}x_0^T &1\end{bmatrix} (r(t)F+s(t)G) \begin{bmatrix}x_0 \\ 1\end{bmatrix}
\geq 0, \quad |t|<\varepsilon .
\end{aligned}
\]
Therefore, \( c(0)=0 \) is the local minimum of the affine function \( c(t) \), indicating that \( c(t)=0 \) for all \( t\in \mathbb{R} \). In particular, \( 0=c(1)=\lambda f(x_0)+ \mu g(x_0) \), which contradicts the assumption that \( \lambda f(x) + \mu g(x) >0 \) for all \( x\in \mathbb{R} ^n \). The proof is complete.

\item
The proof is similar to that of (a), and therefore is omitted.
\end{enumerate}

\end{proof}

Incorporating \eqref{min-f} with Lemma \ref{u_ann-tu_a}, we have the equivalent matrix pencil conditions for (C1) and (F1).

\begin{theorem}
The following are equivalent.
\begin{itemize}
    \item[(C1)] There exists \( (\lambda ,\mu )\neq (0,0) \) such that \( \lambda f(x)+\mu g(x)>0 \) for all \( x\in \mathbb{R} ^n \).
    \item[(C1-1)] \( II_\succeq (F,G)\neq \{(0,0)\} \), and for any \( (\lambda ,\mu )\in \operatorname{relint}(II_\succeq (F,G)) \), there is
\[
\lambda a_0+\mu b_0-( \lambda a+\mu b)^T ( \lambda A+\mu B)^\dagger ( \lambda a+\mu b) > 0.
\]

\end{itemize}
\end{theorem}

\begin{theorem}
\label{(F1-1)}
The following are equivalent.
\begin{itemize}
    \item[(F1)] There exists \( \xi \in \mathbb{R} \) such that \( f(x)+\xi g(x)>0 \) for all \( x\in \mathbb{R} ^n \).
    \item[(F1-1)] \( I_\succeq (F,G)\neq \emptyset \), and for any \( \xi \in \operatorname{relint}(I_\succeq (F,G)) \), there is
\[
a_0+\xi b_0-(a+\xi b)^T (A+\xi B)^\dagger (a+\xi b) > 0.
\]

\end{itemize}
\end{theorem}

\begin{remark}
It is possible for a \( (\lambda ,\mu )\in \mathbb{R} ^2 \) that satisfies (C1) to fall on the relative boundary of \( II_\succeq (F,G) \).

Let \( f(x_1)=x_1^2 \) and \( g(x_1)=1 \). Then,
\[
II_\succeq (F,G)= \left\{(\lambda ,\mu ) \middle | \lambda 
\begin{bmatrix}
1 & 0 \\
0 & 0 
\end{bmatrix}
+\mu 
\begin{bmatrix}
0 & 0 \\
0 & 1 
\end{bmatrix}
\succeq 0 \right \} =[0,\infty )\times [0,\infty ).
\]
\begin{itemize}
    \item For \( \lambda =0 \), \( \mu >0 \), on the relative boundary \( \partial (II_\succeq (F,G)) \), there is \( \lambda f(x_1) + \mu g(x_1) = \mu >0 \) for all \( x_1\in \mathbb{R} \), in which case (C1) holds.
    \item For \( \lambda >0 \), \( \mu =0 \), on the relative boundary \( \partial (II_\succeq (F,G)) \), \( \lambda f(x_1) + \mu g(x_1) = \lambda x_1^2 =0 \) if \( x_1=0 \). Then, (C1) fails.
\end{itemize}
\end{remark}

For completeness, the procedures developed in the preceding two subsections are summarized as pseudocode in Appendix \ref{'en-s'ng-hu8at}. The procedures for checking conditions (C1) and (C2) are combined in Algorithm~\ref{alg:C12}, whereas those for checking conditions (F1) and (F2) are combined in Algorithm~\ref{alg:F12}.

\section{Computing \( II_\succeq (F,G) \) and \( I_\succeq (F,G) \)}
\label{comp II(F,G)}

In this subsection, we demonstrate how to calculate \( II_\succeq (F,G) \). Notice that \( II_\succeq (F,G) \) is a closed convex cone, since it is the dual cone of the joint numerical range set of \( F \) and \( G \) 
\[
\mathcal{E} = \{(z^T F z, z^T G z) \mid z\in \mathbb{R} ^{n+1}\},
\]
where \( F,G\in \mathcal S^{n+1} \) are defined in \eqref{defn F,G}. In fact,
\begin{align*}
II_\succeq (F,G) & =\{(\lambda ,\mu ) \mid \lambda F+\mu G\succeq 0\}\\
& = \{(\lambda ,\mu )\mid \lambda z^T F z + \mu z^T G z \geq 0, \;\forall \, z\in \mathbb{R} ^{n+1}\}\\
& = \{(\lambda ,\mu )\mid \lambda r + \mu s \geq 0, \;\forall \, (r,s) \in \mathcal{E} \}. 
\end{align*}
It can be easily seen that a closed convex cone in \( \mathbb{R} ^2 \) is a polyhedron with at most two extreme rays, whose union constitutes the boundary of \( II_\succeq (F,G) \). Moreover, \( II_\succeq (F,G) \) is pointed if and only if \( \{F,G\} \) is linearly independent due to the following argument
\begin{equation}
\label{pted<=>l.i.}
\begin{aligned}
&\nexists (\lambda ,\mu )\neq (0,0)\text{ s.t. } (\lambda ,\mu )\in II_\succeq (F,G),~ (-\lambda ,-\mu )\in II_\succeq (F,G)\\
\Longleftrightarrow {}&\nexists (\lambda ,\mu )\neq (0,0)\text{ s.t. } \lambda F+\mu G\succeq 0,~ -\lambda F-\mu G\succeq 0\\
\Longleftrightarrow {}&\nexists (\lambda ,\mu )\neq (0,0)\text{ s.t. } \lambda F+\mu G=0.
\end{aligned}
\end{equation}

\subsection{Computing \( II_\succeq (F,G) \) when \( \{F,G\} \) is linearly dependent}
\label{II_(F,G),F//G}

If \( F=G=0 \), the matrix pencil \( II_\succeq (F,G) \) is the whole space \( \mathbb{R} ^2 \). Suppose that \( F \) and \( G \) are not both zero. By their linear dependence, let \( (\tilde{\lambda },\tilde{\mu })\neq (0,0) \) satisfy \( \tilde{\lambda }F+\tilde{\mu }G=0 \). Then, the line spanned by \( (\tilde{\lambda }, \tilde{\mu }) \)
\[
L=\{(t\tilde{\lambda },t\tilde{\mu })\mid t\in \mathbb{R} \}=\{(r,s)\in \mathbb{R} ^2\mid \tilde{\mu }r-\tilde{\lambda }s=0\}
\]
is contained in \( II_\succeq (F,G) \) and \( L \) separates \( \mathbb{R} ^2 \) into two half-planes:
\[
L^+=\{(r,s)\in \mathbb{R} ^2\mid \tilde{\mu }r-\tilde{\lambda }s\geq 0\} \text{ and } L^-=\{(r,s)\in \mathbb{R} ^2\mid \tilde{\mu }r-\tilde{\lambda }s\leq 0\}
\]
By writing any \( rF+sG \) as a linear combination of \( \tilde{\lambda }F+\tilde{\mu }G \) and  \( \tilde{\mu }F-\tilde{\lambda }G \), we have
\begin{equation}
\label{rF+sG l.c.}
rF+sG=\frac{\tilde{\lambda }r+\tilde{\mu }s}{\tilde{\lambda }^2+\tilde{\mu }^2}\underbrace {(\tilde{\lambda }F+\tilde{\mu }G)}_{=0}+\frac{\tilde{\mu }r-\tilde{\lambda }s}{\tilde{\lambda }^2+\tilde{\mu }^2}(\tilde{\mu }F-\tilde{\lambda }G)=\frac{\tilde{\mu }r-\tilde{\lambda }s}{\tilde{\lambda }^2+\tilde{\mu }^2}(\tilde{\mu }F-\tilde{\lambda }G).
\end{equation}
Then, \( \tilde{\mu }F-\tilde{\lambda }G\neq 0 \). Otherwise, by \eqref{rF+sG l.c.}, there would be \( rF+sG=0 \), \( \forall (r,s)\in \mathbb{R} ^2 \), implying a contradiction \( F=G=0 \). As a result, when \( \{F,G\} \) is linearly dependent and \( (\tilde{\lambda },\tilde{\mu }) \) are such that \( \tilde{\lambda }F+\tilde{\mu }G=0 \) and \(  \tilde{\lambda }^2+ \tilde{\mu }^2>0 \), we have
\begin{align*}
(r,s)\in II_\succeq (F,G)&\Longleftrightarrow \frac{\tilde{\mu }r-\tilde{\lambda }s}{\tilde{\lambda }^2+\tilde{\mu }^2}(\tilde{\mu }F-\tilde{\lambda }G)\succeq 0\\
&\Longleftrightarrow 
\begin{cases}
\tilde{\mu }r-\tilde{\lambda }s\geq 0, & \text{if } \tilde{\mu }F-\tilde{\lambda }G\succeq 0;\\
\tilde{\mu }r-\tilde{\lambda }s\leq 0, & \text{if } \tilde{\mu }F-\tilde{\lambda }G\preceq 0;\\
\tilde{\mu }r-\tilde{\lambda }s=0, & \text{if } \tilde{\mu }F-\tilde{\lambda }G \text{ is indefinite}.
\end{cases}
\end{align*}
In summary,
\begin{itemize}
    \item If \( \tilde{\mu }F-\tilde{\lambda }G\succeq 0 \), then \( II_\succeq (F,G)=\{(r,s)\in \mathbb{R} ^2\mid \tilde{\mu }r-\tilde{\lambda }s\geq 0\} \) is a closed half-plane.
    \item If \( \tilde{\mu }F-\tilde{\lambda }G\preceq 0 \), then \( II_\succeq (F,G)=\{(r,s)\in \mathbb{R} ^2\mid \tilde{\mu }r-\tilde{\lambda }s\leq 0\} \).
    \item If \( \tilde{\mu }F-\tilde{\lambda }G \) is indefinite, then \( II_\succeq (F,G)=L=\{(r,s)\in \mathbb{R} ^2\mid \tilde{\mu }r-\tilde{\lambda }s=0\} \) is a line in \( \mathbb{R} ^2 \). 
\end{itemize}

\subsection{Computing \( II_\succeq (F,G) \) when \( \{F,G\} \) is linearly independent}
\label{II_(F,G),F/|G}

In this case, we determine the closed convex cone \( II_\succeq (F,G) \) by characterizing its boundary \( \partial (II_\succeq (F,G)) \). Let us first remove the common null space of \( F \) and \( G \) based on the idea in \cite{hsia2014revisit}. Let \( R\in \mathbb{R} ^{(n+1)\times (n+1-k)} \) and \( S\in \mathbb{R} ^{(n+1)\times k} \) be matrix bases of \( \mathcal N(F)\cap \mathcal N(G) \) and \( (\mathcal N(F)\cap \mathcal N(G))^\bot \) respectively, where \( \bot \) denotes the orthogonal complement, and \( k=\dim ((\mathcal N(F)\cap \mathcal N(G))^\bot ) \). Then, \( \begin{bmatrix} R & S \end{bmatrix} \) is an invertible \( (n+1)\times (n+1) \) matrix and
\begin{align*}
\begin{bmatrix}
R^T \\ S^T 
\end{bmatrix}
(\lambda F+\mu G)
\begin{bmatrix}
R & S
\end{bmatrix}
& =
\lambda 
\begin{bmatrix}
R^T F R & R^T F S\\
S^T F R & S^T F S
\end{bmatrix}
+ \mu 
\begin{bmatrix}
R^T G R & R^T G S\\
S^T G R & S^T G S
\end{bmatrix}
\\
& =
\begin{bmatrix}
0 & 0\\
0 & \lambda S^T F S + \mu S^T G S
\end{bmatrix},
\end{align*}
which indicates that
\begin{equation}
\label{(F,G)=(SFS,SGS)}
II_\succeq (F,G) = II_\succeq (S^T F S, S^T G S),
\end{equation}
and it can be shown that
\begin{equation}
\label{N(SFG)N(SGS)}
\mathcal N(S^T FS)\cap \mathcal N(S^T GS)=\{0\}.
\end{equation}
Lemma \ref{int(II)} shows that the interior of \( II_\succeq (F,G) \) is the ``positive definite matrix pencil'' of \( F \) and \( G \).

\begin{lemma}
\label{int(II)}
Let \( S\in \mathbb{R} ^{(n+1)\times k} \) be the matrix base of \( (\mathcal N(F)\cap \mathcal N(G))^\bot \) with \( k=\dim ((\mathcal N(F)\cap \mathcal N(G))^\bot ) \). Then, 
\[
\operatorname{int} (II_\succeq (F,G)) = \operatorname{int} (II_\succeq (S^T F S, S^T G S)) = \{(\lambda ,\mu ) \mid \lambda S^T F S + \mu S^T G S \succ 0\}.
\]
\end{lemma}

\begin{proof}
From \eqref{(F,G)=(SFS,SGS)}, \( \operatorname{int} (II_\succeq (F,G)) = \operatorname{int} (II_\succeq (S^T F S, S^T G S)) \). To further prove that \( \operatorname{int} (II_\succeq (S^T F S, S^T G S)) = \{(\lambda ,\mu ) \mid \lambda S^T F S + \mu S^T G S \succ 0\} \), we notice that \( \{(\lambda ,\mu ) \mid \lambda S^T F S + \mu S^T G S \succ 0\} \) is an open subset of \( II_\succeq (S^T F S, S^T G S) \) so that
\[
\{(\lambda ,\mu ) \mid \lambda S^T F S + \mu S^T G S \succ 0\}\subset \operatorname{int} (II_\succeq (S^T F S, S^T G S)).
\]

Conversely, Suppose, to the contrary, that there were a pair \( (\bar{\lambda },\bar{\mu })\in \operatorname{int}(II_\succeq (F,G)) \) and \( w_0\in \mathbb{R} ^k \setminus \{0\} \) such that 
\begin{equation}
\label{SFG+SGS}
w_0^T (\bar{\lambda } S^T F S + \bar{\mu } S^T G S) w_0 \leq 0,
\end{equation}
where \( S \) is an \( (n+1)\times k \) matrix as defined above. Since \( (\bar{\lambda },\bar{\mu }) \) is an interior point, there exists an open ball \( \mathcal{B} (\bar{\lambda },\bar{\mu }) \subset II_\succeq (S^T F S, S^T G S) \) such that
\begin{equation}
\label{(SFG,SGS)}
(\lambda ,\mu )\in \mathcal{B} (\bar{\lambda },\bar{\mu }) \Longrightarrow w ^T (\lambda S^T F S + \mu S^T G S)w \geq 0, \; \forall \, w\in \mathbb{R} ^k.
\end{equation}
Combining \eqref{SFG+SGS} and \eqref{(SFG,SGS)}, we obtain
\[
w_0^T (\bar{\lambda } S^T F S + \bar{\mu } S^T G S) w_0 = 0.
\]
It implies that the linear functional \( K(\lambda ,\mu )=w_0^T (\lambda S^T F S + \mu S^T G S)w_0 \) attains the minimum value \( 0 \) at \( (\bar{\lambda },\bar{\mu }) \) on \( \mathcal{B} (\bar{\lambda },\bar{\mu }) \). Therefore, \( K \) is identically zero. That is,
\begin{equation}
\label{L=0}
w_0^T (\lambda S^T F S + \mu S^T G S) w_0 = 0 ,\;\forall \,(\lambda ,\mu )\in \mathbb{R} ^2.
\end{equation}
Incorporating \eqref{L=0} with \eqref{(SFG,SGS)}, we have, for each \( (\lambda ,\mu )\in \mathcal{B} (\bar{\lambda },\bar{\mu }) \), the quadratic form \( w^T ( \lambda S^T FS + \mu S^T GS)w \) attains the minimum value \( 0 \) at \( w_0\neq 0 \) over \( \mathbb{R} ^k \). Necessarily, there is \( (\lambda S^T F S + \mu S^T G S) w_0 = 0 \) for each \( (\lambda ,\mu )\in \mathcal{B} (\bar{\lambda },\bar{\mu }) \). Consequently,
\begin{align*}
&\begin{bmatrix}
(S^T FS)w_0 & (S^T GS)w_0
\end{bmatrix}_{k\times 2}
\begin{bmatrix}
\lambda \\ \mu 
\end{bmatrix}
=0, \; \forall \, (\lambda ,\mu )\in \mathcal{B} (\bar{\lambda },\bar{\mu })\\
\Longrightarrow {} & \mathcal{B} (\bar{\lambda },\bar{\mu })\subset \mathcal N
\begin{bmatrix}
(S^T FS)w_0 & (S^T GS)w_0
\end{bmatrix}_{k\times 2}\\
\Longrightarrow {}&
\begin{bmatrix}
(S^T FS)w_0 & (S^T GS)w_0
\end{bmatrix}
=0.
\end{align*}
It follows that \( (S^T FS)w_0=(S^T GS)w_0=0 \) with \( w_0\neq 0 \), contradicting \eqref{N(SFG)N(SGS)}. The proof is thus complete.
\end{proof}

Since \( II_\succeq (F,G) \) is closed, by \eqref{(F,G)=(SFS,SGS)} and Lemma \ref{int(II)}, an element \( (\lambda ,\mu ) \) is in the boundary set \( \partial (II_\succeq (F,G)) \) if and only if \( \lambda S^T FS+\mu S^T GS\succeq 0 \) and \( \lambda S^T FS+\mu S^T GS\nsucc 0 \). Namely,
\[
\partial (II_\succeq (F,G))=\{(\lambda ,\mu )\mid \lambda S^T FS+\mu S^T GS\succeq 0,\det(\lambda S^T FS+\mu S^T GS)=0\}.
\]
Since \( \{F,G\} \) are linear independent, by \eqref{pted<=>l.i.}, \( II_\succeq (F,G) \) is a pointed cone. Then, its boundary \( \partial (II_\succeq (F,G)) \) consists of at most two non-parallel extreme rays, except for the special case \( II_\succeq (F,G)=\{(0,0)\} \). To characterize such rays, it suffices to look into those \( (\lambda ,\mu )\in II_\succeq (F,G) \) that sit on the unit circle. Specifically, we consider the set
\begin{align*}
\partial ^\circ (II_\succeq (F,G))&=\partial (II_\succeq (F,G))\cap \{(\lambda ,\mu )\mid \lambda ^2+\mu ^2=1\}\\
&=\left\{(\lambda ,\mu )\, \middle | \,
\begin{aligned}
&\lambda S^T FS+\mu S^T GS\succeq 0,\\
&\det(\lambda S^T FS+\mu S^T GS)=0,\lambda ^2+\mu ^2=1
\end{aligned}
\right\}
\end{align*}
with a superscript ``\( \circ \)'' to indicate the intersection of \( \partial (II_\succeq (F,G)) \) with the unit circle.
\begin{itemize}
    \item If \( \partial ^\circ (II_\succeq (F,G))=\emptyset \), then \( \partial (II_\succeq (F,G))=II_\succeq (F,G)=\{(0,0)\} \).
    \item If \( \partial ^\circ (II_\succeq (F,G))=\{(\hat{\lambda },\hat{\mu })\} \) is a singleton, then
    \[
    \partial (II_\succeq (F,G))=II_\succeq (F,G)=\{(t\hat{\lambda },t\hat{\mu })\mid t\geq 0\}
    \]
    is a ray.
    \item If \( \partial ^\circ (II_\succeq (F,G))=\{(\hat{\lambda },\hat{\mu }),(\tilde{\lambda },\tilde{\mu })\} \), then
    \[
    \partial (II_\succeq (F,G))=\{(t\hat{\lambda },t\hat{\mu })\mid t\geq 0\}\cup \{(s\tilde{\lambda },s\tilde{\mu })\mid s\geq 0\}
    \]
    is the union of two rays, and \( II_\succeq (F,G) \) can be represented as conic combinations of the two rays as follows.
    \begin{equation}
    \label{II=_/}
    II_\succeq (F,G)=\{(t\hat{\lambda }+s\tilde{\lambda },t\hat{\mu }+s\tilde{\mu })\mid t\geq 0,~ s\geq 0\},
    \end{equation}
    which is an angular sector with angle less than \( \pi \).
\end{itemize}
To characterize \( \partial ^\circ (II_\succeq (F,G)) \), let us first deal with a larger candidate set
\[
\Lambda = \{(\lambda ,\mu ) \mid \det(\lambda S^T FS+\mu S^T GS)=0, \lambda ^2 + \mu ^2 = 1\}
\]
and handle the positive semi-definiteness of \( \lambda S^T FS+\mu S^T GS \) separately. There are two possibilities:
\begin{itemize}
    \item \( \lambda =0 \). In this case, \( \mu =\pm 1 \) and \( \det(S^T GS)=0 \).
    \item \( \lambda \neq 0 \). By setting \( \xi =\frac{\mu }{\lambda } \), we can turn \( \det(\lambda S^T FS+\mu S^T GS) \) into a polynomial of the variable \( \xi \) with degree at most \( k=\dim((\mathcal N(F)\cap \mathcal N(G))^\bot ) \):
\begin{equation}
\label{Psi=det=0}
\Psi (\xi )=\det (S^T FS + \xi S^T GS)=0.
\end{equation}
Collect all the real roots of \( \Psi (\xi ) \) into
\[
\Xi =\{\xi \in \mathbb{R} \mid \Psi (\xi )=0\}
\]
and for each \( \xi \in \Xi \), by solving \( \frac{\mu }{\lambda }=\xi \) together with \( \lambda ^2+\mu ^2 =1 \), namely, \( \mu =\xi \lambda \) and \( \lambda ^2(1+\xi ^2)=1 \), we obtain
\[
(\lambda ,\mu )=\pm \left (\frac{1}{\sqrt{1+\xi ^2}},\frac{\xi }{\sqrt{1+\xi ^2}} \right ).
\]
\end{itemize}
To conclude,
\begin{equation}
\label{Lambda express}
\Lambda = \bigcup _{\xi \in \Xi }\left \{\left (\frac{1}{\sqrt{1+\xi  ^2}},\frac{\xi  }{\sqrt{1+\xi  ^2}} \right ),\left (\frac{-1}{\sqrt{1+\xi  ^2}},\frac{-\xi  }{\sqrt{1+\xi  ^2}} \right ) \right \}\cup \Lambda _0,
\end{equation}
where
\[
\Lambda _0=
\begin{cases}
\emptyset  & \text{if } \det (S^T GS)\neq 0 ;\\
\{(0,1),(0,-1)\}& \text{if } \det (S^T GS)=0.
\end{cases}
\]
Moreover, according to Lemma \ref{I complex} \cite[2026]{nguyen2023Po4} below, \( \Psi (\xi ) \) in \eqref{Psi=det=0} is not a constant polynomial and it has only real roots if \( I_\succeq (F, G) = \{\xi \in \mathbb{R} \mid F+\xi G \succeq 0\} \neq \emptyset \). In this case, \( \Lambda \) in \eqref{Lambda express} must be a finite set.

\begin{lemma}[\mbox{\cite[2026]{nguyen2023Po4}}]
\label{I complex}
If \( I_\succeq (F, G)=I_\succeq (S^T FS, S^T GS) \) is not empty, then \( \Psi (\xi )=\det (S^T FS + \xi S^T GS) \) is not a constant polynomial and has only real roots.
\end{lemma}

With the following result, if \( \Psi (\xi ) \) happens to be a constant polynomial or, to appear at least one nonreal root, we can immediately conclude that the special case \( II_\succeq (F,G)=\{(0,0)\} \) occurs.

\begin{lemma}
If \( II_\succeq (F, G)=II_\succeq (S^T FS, S^T GS)\neq \{(0,0)\} \), then \( \Psi (\xi )=\det (S^T FS + \xi S^T GS) \) is not a constant polynomial and has only real roots.
\end{lemma}

\begin{proof}
To shorten the notation, let \( P=S^T FS\in \mathcal S^k \) and \( Q=S^T GS\in \mathcal S^k \). Suppose that \( (0,0)\neq (\tilde{\lambda } ,\tilde{\mu } )\in II_\succeq (P,Q) \). Namely, \( \tilde{\lambda } P + \tilde{\mu } Q \succeq 0 \) with \( \tilde{\lambda }^2+\tilde{\mu }^2\neq 0 \). We only prove the case \( \tilde{\lambda } \neq 0 \), as the other case \( \tilde{\mu }\neq 0 \) can be similarly argued.
\begin{itemize}
\item If \( \tilde{\lambda }>0 \), then \( P + \frac{\tilde{\mu }}{\tilde{\lambda }}Q \succeq 0 \) and \( \frac{\tilde{\mu }}{\tilde{\lambda }} \in I_\succeq (P,Q)\neq \emptyset \). By Lemma \ref{I complex}, \( \det(P+\xi Q) \) is not a constant polynomial and has only real roots.
\item If \( \tilde{\lambda }<0 \), then \( -P - \frac{\tilde{\mu }}{\tilde{\lambda }}Q \succeq 0 \) and \( \frac{\tilde{\mu }}{|\tilde{\lambda }|} \in I_\succeq (-P,Q)\neq \emptyset \). By Lemma \ref{I complex}, \( \det(-P+\gamma Q)=(-1)^k \det(P-\gamma Q) \) is not a constant polynomial and has only real roots. After change of variable \( \xi =-\gamma \), \( \det(P+\xi Q) \) is also non-constant and has only real roots, too.
\end{itemize}
The proof is thus completed.
\end{proof}

Suppose that \( \Psi (\xi ) \) is a non‑constant polynomial with only real roots and denote all the distinct roots by
\[
\Xi =\{\xi _1>\xi _2>\cdots >\xi _m\}
\]
where \( 1\leq m\leq k=\dim((\mathcal N(F)\cap \mathcal N(G))^\bot ) \).

Due to the symmetry in \eqref{Lambda express}, \( \Lambda \) contains either \( 2m \) elements if \( \det (S^T GS)\neq 0 \) or \( 2m+2 \) elements if \( \det (S^T GS)=0 \), all on the unit circle. We enumerate all the elements of \( \Lambda \) by labelling them consecutively as \( (\lambda _1,\mu _1),(\lambda _2,\mu _2),\ldots ,(\lambda _{2m},\mu _{2m}) \) or \( (\lambda _1,\mu _1),(\lambda _2,\mu _2),\ldots ,(\lambda _{2m+2},\mu _{2m+2}) \) according to the following way. 
\begin{itemize}
    \item If \( \det (S^T GS)\neq 0 \), define
\begin{align}
(\lambda _i,\mu _i)&=\left (\frac{1}{\sqrt{1+\xi _i^2}},\frac{\xi _i}{\sqrt{1+\xi _i^2}} \right ),~ \forall \,1\leq i\leq m; \label{det=/=0,i=1...m}\\
(\lambda _i,\mu _i)&=\left (\frac{-1}{\sqrt{1+\xi _{i-m}^2}},\frac{-\xi _{i-m}}{\sqrt{1+\xi _{i-m}^2}} \right ),~\forall \,m+1\leq i\leq 2m. 
\end{align}
\item If \( \det (S^T GS)=0 \), define
\begin{align}
(\lambda _i,\mu _i)&=\left (\frac{1}{\sqrt{1+\xi _i^2}},\frac{\xi _i}{\sqrt{1+\xi _i^2}} \right ),~\forall \,1\leq i\leq m; 
\\
(\lambda _i,\mu _i)&=(0,-1),~i=m+1; 
\\
(\lambda _i,\mu _i)&=\left (\frac{-1}{\sqrt{1+\xi _{i-(m+1)}^2}},\frac{-\xi _{i-(m+1)}}{\sqrt{1+\xi _{i-(m+1)}^2}} \right ),~\forall \,m+2\leq i\leq 2m+1; 
\\
(\lambda _i,\mu _i)&=(0,1),~i=2m+2. \label{det=0,i=2m+2}
\end{align}
\end{itemize}
Since \( h(\xi )=\frac{\xi }{\sqrt{1+\xi ^2}} \) with \( h'(\xi )=(1+\xi ^2)^{-\frac 3 2}>0 \) is an increasing function, the way we arrange \( (\lambda _i,\mu _i) \) as in \eqref{det=/=0,i=1...m}--\eqref{det=0,i=2m+2} orients them in a clockwise manner on the unit circle, with \( (\lambda _i,\mu _i) \), \( i=1,2,\ldots ,m \text{ (or }m+1\text{)} \) in the I and the IV quadrant, whereas all the rest on the III and then the II quadrant. Among the \( 2m \) (or \( 2m+2 \)) elements in \( \Lambda \), there could be 0, 1, or at most 2 of them that can make \( \lambda F+\mu G \) positive semi-definite, but not positive definite, depending on whether \( II_\succeq (F,G)=\{(0,0)\} \), \( II_\succeq (F,G) \) is a ray, or an angular sector, respectively.

In Lemma \ref{consecutive Lambda} below, we show that, when \( \partial ^\circ (II_\succeq (F,G)) \) has exactly two elements \( (\hat{\lambda },\hat{\mu }) \) and \( (\tilde{\lambda },\tilde{\mu }) \) so that \( II_\succeq (F,G) \) is an angular sector with angle less than \( \pi \) (see \eqref{II=_/}), the two points \( (\hat{\lambda },\hat{\mu }) \) and \( (\tilde{\lambda },\tilde{\mu }) \) must be a pair of neighboring elements of \( \Lambda \) on the unit circle.

\begin{lemma}
\label{consecutive Lambda}
Suppose that \( \partial ^\circ (II_\succeq (F,G)) \) contains exactly two elements and that \( \det(S^T GS)\neq 0 \). Then they must be two consecutive points on the unit circle in \( \Lambda \): either \( \{(\lambda _i,\mu _i),(\lambda _{i+1},\mu _{i+1})\} \) for some \( i=1,2,\ldots ,2m-1 \) or \( \{(\lambda _{2m},\mu _{2m}),(\lambda _1,\mu _1)\} \).
\end{lemma}

\begin{proof}
The argument is simple. Let \( (\lambda _p,\mu _p),(\lambda _q,\mu _q)\in \partial ^\circ (II_\succeq (F,G)) \) such that
\[
\lambda _iS^T FS+\mu _iS^T GS\succeq 0, ~ \lambda _iS^T FS+\mu _iS^T GS\nsucc 0,~  \lambda _i^2+\mu _i^2=1. \quad (i=p,q)
\]
Suppose that \( (\lambda _p,\mu _p) \) and \( (\lambda _q,\mu _q) \) were not adjacent to each other on the unit circle. There would be some other \( (\lambda _s,\mu _s)\in \Lambda \), which is a conic combination of \( (\lambda _p,\mu _p) \) and \( (\lambda _q,\mu _q) \). Then, \( (\lambda _s,\mu _s)\in \operatorname{int}(II_\succeq (F,G)) \), and by Lemma \ref{int(II)}, \( \lambda _sS^T FS+\mu _sS^T GS\succ 0 \), which contradicts \( \det(\lambda _sS^T FS+\mu _sS^T GS)=0 \) since  \( (\lambda _s,\mu _s)\in \Lambda \). The proof is complete.
\end{proof}

An analogue of Lemma~\ref{consecutive Lambda} for the case \(\det(S^\top GS)=0\) follows from the same argument, with the two additional candidates \((0,1)\) and \((0,-1)\) included in \(\Lambda\).

The cone \(II_\succeq (F,G)\) can be determined by the following procedure.
\begin{enumerate}
    \item Compute all distinct real roots of \( \Psi (\xi )=\det(S^T FS+\xi S^T GS)=0 \) and order them as \(\xi _1>\xi _2>\cdots >\xi _m \).
    \item Based on whether \(\det(S^T GS)\neq 0 \) or not, utilize \(\xi _1>\xi _2>\cdots >\xi _m \) to construct candidates \( (\lambda _i,\mu _i) \), \( i=1,2,\ldots ,2m \) (or \( i=1,2,\ldots ,2m+2 \)), by  \eqref{det=/=0,i=1...m}--\eqref{det=0,i=2m+2}, respectively, for boundary points of \( II_\succeq (F,G) \).
    \item Starting from \( (\lambda _1,\mu _1) \), examine sequentially the positive semidefiniteness of matrix pencils:
\begin{equation}
\label{psd condition}
    \lambda _i S^T FS+\mu _i S^T GS \succeq 0
\end{equation}
    for \( i=1,2,\ldots ,2m \) (or \( i=1,2,\ldots ,2m+2 \)).
\begin{enumerate}[label=\theenumi-\roman*.]
    \item If \eqref{psd condition} fails for every \( (\lambda _i,\mu _i) \), then \( II_\succeq (F,G)=\{(0,0)\} \).
    \item Otherwise, let \( (\lambda _{i_0},\mu _{i_0}) \) denote the pair satisfying \eqref{psd condition}. By Lemma \ref{consecutive Lambda}, it remains to check its neighbors \( (\lambda _{i_0-1},\mu _{i_0-1}) \) and \( (\lambda _{i_0+1},\mu _{i_0+1}) \).
\begin{itemize}
    \item If neither neighboring candidate satisfies \eqref{psd condition}, then \( II_\succeq (F,G)=\{(t\lambda _{i_0},t\mu _{i_0})\mid t\geq 0\} \).
    \item  Otherwise, suppose\((\lambda _{i_1},\mu _{i_1}) \) is the one among \( \{(\lambda _{i_0-1},\mu _{i_0-1}), (\lambda _{i_0+1},\mu _{i_0+1})\} \) that make the matrix pencil in \eqref{psd condition} positive semi-definite, then \( II_\succeq (F,G)=\{(t\lambda _{i_0}+s\lambda _{i_1},t\mu _{i_0}+s\mu _{i_1})\mid t\geq 0,~s\geq 0\} \).
\end{itemize}
\end{enumerate}

\end{enumerate}

For the pseudocode for computing \(II_\succeq (F,G)\), see Algorithm~\ref{alg:II(F,G)} in Appendix~\ref{'en-s'ng-hu8at}.

\subsection{Computing \( I_\succeq (F,G) \)}
\label{I(F,G) refined}

To check \( \{f\leq 0\}\cap \{g=0\}=\emptyset \), it requires to find a \( \xi \in \mathbb{R} \) such that \( f(x)+\xi g(x)\geq 0,~\forall x\in \mathbb{R} ^n \), which further reduces to compute the matrix pencil \( I_\succeq (F,G) \) in Theorem \ref{(F2-1)} and in Theorem \ref{(F1-1)}. We point out that the way to compute \( I_\succeq (F,G) \) has been well addressed in Section 2.3 \cite{nguyen2023Po4} and Algorithm A.1 \cite{nguyen2023Po4} for the purpose of determining the optimal value and its attainability for a given QP1QC problem. In this subsection, we recall the part of their procedure, and then introduce a refinement that reduces the computation of the matrix pencil \( I_\succeq (F,G) \).

Nguyen et al.\ \cite{nguyen2023Po4} proved that, if \( \Psi (\xi )=\det(S^T FS+\xi S^T GS) \) is a constant polynomial or has at least one non-real root, then \( I_\succeq (F,G)=\emptyset \). Suppose that \( \Psi (\xi ) \) is a nonconstant polynomial and has only real roots. Denote its distinct roots, arranged in decreasing order, by
\[
\Xi =\{\xi _1>\xi _2>\cdots >\xi _m\}.
\]
If none of these roots satisfies \( S^T FS+\xi _iS^T GS\succeq 0 \), then \( I_\succeq (F,G)=\emptyset \). Otherwise, \( I_\succeq (F,G) \) can be determined by
\[
I_\succeq (F,G)=
\begin{cases}
(-\infty ,\xi _{\max}], & \text{if } G\preceq 0;\\
[\xi _{\min},\infty ), & \text{if } G\succeq 0;\\
[\xi _{\min},\xi _{\max}], & \text{otherwise},
\end{cases}
\]
where \( \xi _{\max}=\max\{\xi _i\mid S^T FS+\xi _iS^T GS\succeq 0\} \) and \( \xi _{\min}=\min\{\xi _i\mid S^T FS+\xi _iS^T GS\succeq 0\} \).

In \cite{nguyen2023Po4}, \( \xi _{\max} \) is computed by examining \( \xi _1,\xi _2,\ldots ,\xi _m \) in descending order and selecting the first root \( \xi _i \) for which \( S^T FS+\xi _iS^T GS\succeq 0 \). Similarly, \( \xi _{\min} \) is computed by examining \( \xi _m,\xi _{m-1},\ldots ,\xi _1 \) in ascending order and selecting the first root \( \xi _i \) satisfying \( S^T FS+\xi _iS^T GS\succeq 0 \).

However, when \( G \) is indefinite and \( I_\succeq (F,G)\neq \emptyset \), once \( \xi _{\max} \) has been determined, a separate search for the other endpoint \( \xi _{\min} \) is unnecessary. By Theorem 2.2(iii) in \cite{nguyen2023Po4} (see also Remark 4.2 \cite{hsia2014revisit}), if \( I_\succeq (F,G)=[\xi _{\min},\xi _{\max}] \) is a nondegenerate interval, then
\[
S^T FS+\xi S^T GS\succ 0,~\forall \xi \in \operatorname{int}(I_\succeq (F,G)).
\]
Consequently, no point in the interior of \( I_\succeq (F,G) \) can be a root of \( \Psi (\xi )=\det(S^T FS+\xi S^T GS) \). As the root set \( \Xi =\{\xi _1>\xi _2>\cdots >\xi _m\} \) is ordered, the two endpoints must therefore be consecutive roots with \( \xi _{\min}<\xi _{\max} \). Thus, if \( \xi _{\max} \) is computed to be \( \xi _{i_0} \), then
\begin{equation}
\label{xi_min}
\xi _{\min}=
\begin{cases}
\xi _{i_0+1} & \text{if } \xi _{i_0}>\xi _m \text{ and } S^T FS+\xi _{i_0+1}S^T GS\succeq 0;\\
\xi _{i_0} & \text{if } \xi _{i_0}>\xi _m \text{ and } S^T FS+\xi _{i_0+1}S^T GS\nsucceq 0;\\
\xi _{i_0} & \text{if } \xi _{i_0}=\xi _m.
\end{cases}
\end{equation}

In the next section, we will see that using \eqref{xi_min} to compute the matrix pencil \( I_\succeq (F,G)=[\xi _{\min},\xi _{\max}] \) can save significant computational cost compared with searching \( \xi _{\min} \) independently from the bottom up as in \cite{nguyen2023Po4}.

\section{Numerical experiment}
\label{num exp}

For the numerical experiments on [Calabi](nonhomo version) and [Finsler](nonhomo version), we adopt the benchmark instances constructed in \cite{nguyen2023Po4}. Their instances are randomly generated but all guaranteed to have the known ground truth to the following two cases:
\begin{itemize}
\item[(a)] \( \inf\{f(x)\mid g(x)\leq 0\}>0 \);
\item[(b)] \( \inf\{f(x)\mid g(x)\leq 0\}=0 \) but the optimal value \( 0 \) is unattainable.
\end{itemize}
Both cases lead to \( \{f=0\}\cap \{g=0\}=\emptyset \) and also \( \{f\leq 0\}\cap \{g=0\}=\emptyset \). Therefore, the same benchmark data are suitable for testing [Calabi](nonhomo version) as well as [Finsler](nonhomo version).

To ensure a fair comparison, we generated all test instances using the MATLAB code provided by the authors of \cite{nguyen2023Po4}. The experiments in \cite{nguyen2023Po4} considered problem dimensions \( n=10,20,\ldots,100 \), but we extended the test set to larger instances with \( n=50,100,\ldots,500 \) for evaluating the scalability.

Throughout the experiments, a numerical error tolerance of \( 10^{-5} \) was adopted. For two real numbers \( r \) and \( s \), we regard \( r>s \) if \( r>s+10^{-5} \), \( r=s \) if \( |r-s|\leq 10^{-5} \), and \( r<s \) if \( r<s-10^{-5} \). The correct identification rate for each dimension is defined as the number of instances correctly identified over 200 instances generated for each dimension, with 100 apiece for case (a) and case (b).

All experiments were conducted in MATLAB R2024A on a workstation equipped with an Intel i7-10700 CPU and 32 GB RAM. For each dimension, we report the average runtime over the 200 benchmark instances, and also the correct identification rate. The results are documented in Table~\ref{C0-time,acc} for [Calabi](nonhomo version), and in Table~\ref{F0-time,acc} for [Finsler](nonhomo version).

The speedup factor of our method over the method in \cite{nguyen2023Po4} is defined as the average runtime of the method in \cite{nguyen2023Po4} divided by the average runtime of our method. Thus, a speedup factor greater than 1 indicates that our method is faster.

\begin{table}[!ht]
\renewcommand{\arraystretch}{1}
    \centering
    \caption{Average runtime (in seconds) and correct identification rate for verifying \( \{f=0\}\cap \{g=0\}=\emptyset \)} \label{C0-time,acc}
    \begin{tabular}{|c|ccc|cc|}
    \Xhline{2\arrayrulewidth}
        \multirow{2}{*}{\makecell{Dimension\\ \(n\)}} &\multicolumn{3}{c|}{Average Runtime}  & \multicolumn{2}{c|}{Correct Identification Rate}\\ 
        & Our method&  \cite{nguyen2023Po4}'s method& Speedup factor& Our method& \cite{nguyen2023Po4}'s method\\ \hline
        50 & 0.016& 0.071& 4.47 & 1 & 0.97\\ 
        100 & 0.107& 0.637& 5.94 & 0.98& 0.93\\ 
        150 & 0.406& 2.444& 6.03 & 0.98& 0.88\\ 
        200 & 0.959& 5.522& 5.76 & 0.985& 0.905\\ 
        250 & 1.765& 10.414& 5.90 & 0.965& 0.895\\ 
        300 & 3.111& 17.558& 5.64 & 0.99& 0.86\\ 
        350 & 4.952& 29.022& 5.86 & 0.975& 0.845\\ 
        400 & 7.139& 42.147& 5.90 & 0.98& 0.88\\ 
        450 & 10.362& 62.766& 6.06 & 0.98& 0.82\\ 
        500 & 18.861& 91.615& 4.86 & 0.97& 0.81\\ 
        \Xhline{2\arrayrulewidth}
    \end{tabular}
\end{table}

\begin{table}[!ht]
\renewcommand{\arraystretch}{1}
    \centering
    \caption{Average runtime (in seconds) and correct identification rate for verifying \( \{f\leq 0\}\cap \{g=0\}=\emptyset \)} \label{F0-time,acc}
    \begin{tabular}{|ccc|}
        \Xhline{2\arrayrulewidth}
        Dimension \(n\)& Average Runtime & \makecell{Correct Identi-\\fication Rate}\\ \hline
        50 & 0.016& 1 \\ 
        100 & 0.110& 1 \\ 
        150 & 0.409& 0.985\\ 
        200 & 0.941& 0.985\\ 
        250 & 1.726& 0.98\\ 
        300 & 2.990& 0.99\\ 
        350 & 4.766& 0.98\\ 
        400 & 6.877& 0.99\\ 
        450 & 10.178& 0.99\\ 
        500 & 18.483& 0.975\\
        \Xhline{2\arrayrulewidth}
    \end{tabular}
\end{table}

From Table~\ref{C0-time,acc}, it shows that our algorithm for [Calabi](nonhomo version) outperforms the four-QP1QC approach in \cite{nguyen2023Po4}, achieving a speedup factor of at least \( 4.47 \). On the other hand, the correct identification rate climbs from over \( 80\% \) in \cite{nguyen2023Po4} to more than \( 95\% \) by our method. The substantial reduction in computation time and the significant increase in accuracy fully justify our efforts in characterizing the mathematical essence of the problem.

Table~\ref{F0-time,acc} reports the numerical results for verifying \( \{f\leq 0\}\cap \{g=0\}=\emptyset \) by our algorithm. Comparing with those in Table~\ref{C0-time,acc}, we observe that the results are relatively close in both average runtime and correct identification rate. Since the primary difference lies in computing a two-dimensional matrix pencil \( II_\succeq (A,B) \) for [Calabi](nonhomo version) and a one-dimensional matrix pencil \( I_\succeq (A,B) \) for [Finsler](nonhomo version), the highly similar computational results suggest that this architectural difference has less impact than it superficially appears.

Let us look into these matrix-pencil computations. Suppose that \( \Psi (\xi )=\det(S^T FS+\xi S^T GS) \) has \( m \) distinct real roots and no nonreal one. In computing the one-dimensional pencil \( I_\succeq (F,G) \), there require at most \( m \) positive-semidefiniteness tests for the matrices \( S^T FS+\xi _iS^T GS \) to search for \( \xi _{\max} \). Once \( \xi _{\max} \) is found, \( \xi _{\min} \) can be determined by at most one additional positive-semidefiniteness test.

For the two dimensional matrix pencil \( II_\succeq (F,G) \), according to \eqref{psd condition}, at most \( 2m+2 \) positive-semidefiniteness tests for the matrices \( \lambda _iS^T FS+\mu _iS^T GS \) may be needed. However, due to the symmetry appearing in \eqref{det=/=0,i=1...m}--\eqref{det=0,i=2m+2}, among the \( 2m+2 \) matrix positive-semidefiniteness checks, the matrices in the first \( m+1 \) tests and those in the last \( m+1 \) tests differ only by a negative sign. Consequently, the required number of eigenvalue decompositions is actually \( m+1 \), rather than \( 2m+2 \). This explains why the computation costs for \( I_\succeq (F,G) \) and \( II_\succeq (F,G) \) are roughly the same, which is clearly evidenced by the numerical results in Table~\ref{pencil-time}(a).

Lastly, in Table~\ref{pencil-time}(b), we also verified that, in computing the one-dimensional matrix pencil \( I_\succeq (F,G) \), switching to \eqref{xi_min} for \( \xi _{\min} \) calculation achieves a speedup factor of approximately \( 1.3 \) over searching \( \xi _{\min} \) separately from the bottom up as in \cite{nguyen2023Po4}.

\begin{table}[!ht]
\renewcommand{\arraystretch}{1}
    \centering
    \refstepcounter{table}
    \label{pencil-time}

    \begin{minipage}[t]{0.42\textwidth}
        \centering
        {\normalsize Table~\thetable(a) Average runtime (in seconds) for computing
        \( II_\succeq (A,B) \) and \( I_\succeq (A,B) \)}
    \end{minipage}
    \hfill
    \begin{minipage}[t]{0.57\textwidth}
        \centering
        {\normalsize Table~\thetable(b) Comparison of average runtime (in seconds) for computing
        \( I_\succeq (A,B) \) using different methods}
    \end{minipage}
    \\[1em]
    \begin{minipage}{0.42\textwidth}
    \begin{tabular}{|c|cc|}
    \Xhline{2\arrayrulewidth}
    \makecell{Dimen-\\ sion \(n\)}
    &
    \makecell{\(II_\succeq (A,B)\),\\
    our \\ method }
    &
    \makecell{\(I_\succeq (A,B)\),\\
    our \\ method}
    \\ \hline
    50  & 0.012& 0.013\\ 
    100 & 0.100& 0.102\\ 
    150 & 0.388& 0.389\\ 
    200 & 0.894& 0.899\\ 
    250 & 1.666& 1.672\\ 
    300 & 2.948& 2.884\\ 
    350 & 4.737& 4.665\\ 
    400 & 6.809& 6.684\\ 
    450 & 10.083& 9.982\\ 
    500 & 15.281& 15.062\\
    \Xhline{2\arrayrulewidth}
    \end{tabular}
    \end{minipage}
    \hfill
    \begin{minipage}{0.57\textwidth}
    \begin{tabular}{|c|cc|c|}
    \Xhline{2\arrayrulewidth}
    \makecell{Dimen- \\ sion \(n\)}&
    \makecell{\(I_\succeq (A,B)\),\\
    our \\ method }
    &
    \makecell{\(I_\succeq (A,B)\),\\
    \cite{nguyen2023Po4}'s \\ method}
    &
    \makecell{Speedup \\factor}\\ \hline
    50  & 0.013& 0.014& 1.12\\ 
    100 & 0.102& 0.134& 1.32\\ 
    150 & 0.389& 0.516& 1.32\\ 
    200 & 0.899& 1.178& 1.31\\ 
    250 & 1.672& 2.221& 1.33\\ 
    300 & 2.884& 3.832& 1.33\\ 
    350 & 4.665& 6.188& 1.33\\ 
    400 & 6.684& 9.228& 1.38\\ 
    450 & 9.982& 13.502& 1.35\\ 
    500 & 15.062& 19.407& 1.29\\
    \Xhline{2\arrayrulewidth}
    \end{tabular}
    \end{minipage}
\end{table}

\section{Concluding remarks}

In this paper, we complete the two remaining cases in the family of nonhomogeneous quadratic unsolvability conditions by establishing the nonhomogeneous Calabi theorem and the nonhomogeneous strict Finsler lemma. We characterize \[ \{f=0\}\cap\{g=0\}=\emptyset \quad\text{and}\quad \{f\leq 0\}\cap\{g=0\}=\emptyset \] through three mutually exclusive alternatives. For the nonhomogeneous Calabi theorem, the three alternatives are determined geometrically by the position of the origin \( (0,0) \) relative to the joint numerical range \( (f,g)(\mathbb{R}^n) \) and its convex hull \( \operatorname{conv}((f,g)(\mathbb{R}^n)) \). The topological aspect enters through the position of the origin relative to the closure \( \operatorname{cl}(\operatorname{conv}((f,g)(\mathbb{R}^n))) \), which brings limiting behavior at infinity into the analysis. A parallel classification with three mutually exclusive alternatives is derived for the nonhomogeneous strict Finsler lemma. Numerical experiments confirm that our methods for checking these conditions are computationally efficient and accurate.

For applications, the optimal value \( \alpha \) of \( \min\{f(x)\mid g(x)\leq 0\} \) is obtained from the S-Lemma, the alternative theorem for \( \{f<0\}\cap\{g\leq 0\}=\emptyset \). Our algorithm then decides whether \( \alpha \) is attained through \( \{f-\alpha=0\}\cap\{g\leq 0\}=\emptyset \). Similarly, the optimal value \( \beta \) of \( \min\{f(x)\mid g(x)=0\} \) follows from the S-Lemma with equality, the alternative theorem for \( \{f<0\}\cap\{g=0\}=\emptyset \), while its attainability can be characterized by \( \{f-\beta=0\}\cap\{g=0\}=\emptyset \).

For quadratic fractional programming \[ (\mathrm{QFP})\qquad \begin{array}{rl} \min & \dfrac{h(x)}{g(x)}\\ \mathrm{s.t.} & f(x)\mathbin{\star}0, \end{array} \qquad \star\in\{\leq,=\}, \] whether the objective function is well defined over the feasible set is now answered by our results through checking \( \{f\mathbin{\star}0\}\cap\{g=0\}=\emptyset \). Once a finite optimal value \( \nu \) is known, it is unattainable if and only if \( \{f\mathbin{\star}0\}\cap\{h-\nu g=0\}=\emptyset \), which is again a direct application of our results.

Classical optimization theory often addresses attainability indirectly through verifiable sufficient conditions such as compactness, coercivity, and level-boundedness. The intrinsic difficulty of attainability lies in the geometry of the feasible set at infinity and the asymptotic behavior of the objective function along unbounded directions. A direct study of attainability therefore requires a deeper understanding of the mathematics behind the optimization model. In this paper, we turn the attainability of quadratic optimization problems from a technical existence issue into computable necessary and sufficient conditions, opening a new direction for the study of optimization at infinity.

\begingroup
\setlength{\intextsep}{8pt}
\begin{appendices}
\section{Algorithms.}
\label{'en-s'ng-hu8at}

\begin{algorithm}[H]
\caption{Checking (F3): \( \{g=0\} \) separates \( \{f\leq 0\} \).}
\label{alg:F3}
\begin{algorithmic}
\Input \( f(x)=x^T Ax+2a^T x+a_0 \), \( g(x)=x^T Bx+2b^T x+b_0 \)
\Output Whether or not \( \{g=0\} \) separates \( \{f\leq 0\} \).

\If {\( B\neq 0 \) \textbf{ or } \( b=0 \)}
    \State \Return False;
\EndIf

    \State \( V\leftarrow \) a matrix basis of \( \mathcal{N} (b^T ) \); \Comment{\( V\in \mathbb{R} ^{n\times (n-1)} \)}
    \State \( x_0 \leftarrow \frac{-b_0}{2b^T b}b \);
    \State \( w\leftarrow V^T (Ax_0 + a) \); \Comment{\( w\in \mathbb{R} ^{n-1} \)}

    \If{\(
    \begin{cases}
    b\in \mathcal{R} (A) \textbf{ and}\\
    b\neq 0 \textbf{ and}\\
    A \text{ has exactly one negative eigenvalue} \textbf{ and}\\
    b^T A^\dagger b \leq 0 \textbf{ and}\\
    w\in \mathcal{R} (V^T AV) \textbf{ and}\\
    f(x_0)-w^T (V^T AV)^\dagger w >0
    \end{cases}
    \)}
        \State \Return True;
\EndIf

\State \Return False;
\end{algorithmic}
\end{algorithm}

\begin{algorithm}[H]
\caption{Checking (C3): the separation property of \( f \) and \( g \).}
\label{alg:C3}
\begin{algorithmic}
\Input \( f(x)=x^T Ax+2a^T x+a_0 \), \( g(x)=x^T Bx+2b^T x+b_0 \)
\Output Whether \( f \) and \( g \) have the separation property

\If {\( A=0 \)}
    \If {\( B\neq 0 \)}
        \State \( \gamma _0\leftarrow 1 \), \( \delta _0\leftarrow 0 \);
    \Else
        \State \Return False;
    \EndIf
\Else \Comment{\( A\neq 0 \)}
    \If {\( B=\zeta A \) for some \( \zeta \in \mathbb{R} \)}
        \State \( \gamma _0\leftarrow -\zeta \), \( \delta _0\leftarrow 1 \);
\algstore{alg(C3)}
\end{algorithmic}
\end{algorithm}

\begin{algorithm}[H]
\ContinuedFloat
\caption{(continued)}
\begin{algorithmic}
\algrestore{alg(C3)}
    \Else
        \State \Return False;
    \EndIf
\EndIf
\For {\( (\gamma ,\delta )\in \{(\gamma _0,\delta _0),(-\gamma _0,-\delta _0)\} \)}
    \State \( c \leftarrow \gamma a + \delta b \);
    \State \( d \leftarrow \delta a - \gamma b \);
    \State \( D \leftarrow \delta A - \gamma B \);
    \State \( V\leftarrow \) a matrix basis of \( \mathcal{N} (c^T ) \); \Comment{\( V\in \mathbb{R} ^{n\times (n-1)} \)}
    \State \( x_0 \leftarrow \frac{-\gamma a_0-\delta b_0}{2c^T c}c \);
    \State \( w\leftarrow V^T (Dx_0 + d) \); \Comment{\( w\in \mathbb{R} ^{n-1} \)}
    \If{\(
    \begin{cases}
   c\neq 0,~ c\in \mathcal{R} (D) \textbf{ and}\\
    D \text{ has exactly one negative eigenvalue} \textbf{ and}\\
    c^T D^\dagger c \leq 0 \textbf{ and}\\
    w\in \mathcal{R} (V^T DV) \textbf{ and}\\
    \delta f(x_0)-\gamma g(x_0)-w^T (V^T DV)^\dagger w >0
    \end{cases}
    \)}
        \State \Return True;
    \EndIf
\EndFor
\State \Return False;
\end{algorithmic}
\end{algorithm}

\begin{algorithm}[H]
\caption{Checking (C1) and (C2).}
\label{alg:C12}
\begin{algorithmic}
\Input \( f(x)=x^T Ax+2a^T x+a_0 \), \( g(x)=x^T Bx+2b^T x+b_0 \)
\Output Whether either (C1) or (C2) holds for \( f \) and \( g \)

\State \( F\leftarrow \begin{bmatrix}A & a \\ a^T & a_0 \end{bmatrix} \); \( G\leftarrow \begin{bmatrix}B & b \\ b^T & b_0 \end{bmatrix} \);
\State Compute \( II_\succeq (F,G) \) by Algorithm \ref{alg:II(F,G)};
\If {\( II_\succeq (F,G) \) is \( \{(0,0)\} \)}
    \State \Return Both (C1) and (C2) fail;
\EndIf
\State \( (\lambda ,\mu )\leftarrow \) a relative interior point of \( II_\succeq (F,G) \);
\If {\( \lambda a_0+\mu b_0-( \lambda a+\mu b)^T ( \lambda A+\mu B)^\dagger ( \lambda a+\mu b) > 0 \)}
        \State \Return (C1) holds but (C2) fails; \Comment{The only case where (C1) holds}
\EndIf
\If {\( II_\succeq (F,G) \) is a ray \textbf{ and }\( \lambda a_0+\mu b_0-( \lambda a+\mu b)^T ( \lambda A+\mu B)^\dagger ( \lambda a+\mu b) = 0 \)}
    \State \( C\leftarrow \mu A-\lambda B \);
    \State  \( c\leftarrow \mu a-\lambda b \);
    \State \( c_0\leftarrow \mu a_0-\lambda b_0 \);
    \State \( U\leftarrow \) a matrix basis of \( \mathcal{N}(\lambda A + \mu B) \);
    \State \( u\leftarrow -(\lambda A+\mu B)^\dagger (\lambda a+\mu b) \); 
    \If {\( U^T (Cu+c)\notin \mathcal{R} (U^T CU)\)}
        \State \Return Both (C1) and (C2) fail;
    \EndIf
    \If {\( U^T CU \succeq 0 \) \textbf{ and } \( (\mu f(u)-\lambda g(u))- (Cu+c)^T U(U^T CU)^\dagger U^T (Cu+c) >0 \)}
        \State \Return (C1) fails but (C2) holds;
    \EndIf
    \If {\( U^T CU \preceq 0 \) \textbf{ and } \( (\mu f(u)-\lambda g(u))- (Cu+c)^T U(U^T CU)^\dagger U^T (Cu+c) <0 \)}
        \State \Return (C1) fails but (C2) holds;
    \EndIf
    \State \Return Both (C1) and (C2) fail;
\Else
    \State \Return Both (C1) and (C2) fail; 
\EndIf
\end{algorithmic}
\end{algorithm}

\begin{algorithm}[H]
\caption{Checking (F1) and (F2).}
\label{alg:F12}
\begin{algorithmic}
\Input \( f(x)=x^T Ax+2a^T x+a_0 \), \( g(x)=x^T Bx+2b^T x+b_0 \)
\Output Whether either (F1) or (F2) holds for \( f \) and \( g \)

\State \( F\leftarrow \begin{bmatrix}A & a \\ a^T & a_0 \end{bmatrix} \); \( G\leftarrow \begin{bmatrix}B & b \\ b^T & b_0 \end{bmatrix} \);
\State Compute \( I_\succeq (F,G) \) using the procedure in \cite{nguyen2023Po4};
\If{\( I_\succeq (F,G)=\emptyset \)}
    \State \Return Both (F1) and (F2) fail;
\Else
    \State \( \xi \leftarrow \) a relative interior point of \( I_\succeq (F,G) \);
\EndIf

\If {\( I_\succeq (F,G) \) is a singleton}
    \If {\( a_0+\xi b_0-(a+\xi b)^T (A+\xi B)^\dagger (a+\xi b) > 0 \)}
        \State \Return (F1) holds but (F2) fails; \Comment{The only case where (F1) holds}
    \ElsIf {\( a_0+\xi b_0-(a+\xi b)^T (A+\xi B)^\dagger (a+\xi b) = 0 \)}
        \State \( U\leftarrow \) a matrix basis of \( \mathcal{N}(A + \xi B) \);
        \State \( u\leftarrow -(A+\xi B)^\dagger (a+\xi b) \); 
        \If {\( U^T (Bu+b)\notin \mathcal{R} (U^T BU)\)}
            \State \Return Both (F1) and (F2) fail;
        \EndIf
        \If {\( U^T BU \succeq 0 \) \textbf{and} \( g(u)- (Bu+b)^T U(U^T BU)^\dagger U^T (Bu+b) >0 \)}
            \State \Return (F1) fails but (F2) holds;
        \EndIf
        \If {\( U^T BU \preceq 0 \) \textbf{and} \( g(u)- (Bu+b)^T U(U^T BU)^\dagger U^T (Bu+b) <0 \)}
            \State \Return (F1) fails but (F2) holds;
        \EndIf
        \State \Return Both (F1) and (F2) fail;
    \Else
        \State \Return Both (F1) and (F2) fail;
    \EndIf
\Else \Comment{\( I_\succeq (F,G)\) is a closed interval}
     \If {\( a_0+\xi b_0-(a+\xi b)^T (A+\xi B)^\dagger (a+\xi b) > 0 \)}
        \State \Return (F1) holds but (F2) fails;
    \Else
        \State \Return Both (F1) and (F2) fail;
    \EndIf
\EndIf
\end{algorithmic}
\end{algorithm}

\begin{algorithm}[H]
\caption{Characterizing \( II_\succeq (F,G) \).}
\label{alg:II(F,G)}
\begin{algorithmic}
\Input \( F,G\in \mathcal S^{n+1} \).
\Output The matrix pencil \( II_\succeq (F,G) \).

\If{\( F=G=0 \)}
\State \Return \( II_\succeq (F,G)\leftarrow \mathbb{R} ^2 \);
\ElsIf{\( \lambda F+\mu G=0 \) for some \( (\lambda ,\mu )\neq (0,0) \)}
\If{\( \mu F-\lambda G\succeq 0 \)}
\State \Return \( II_\succeq (F,G)\leftarrow \{(r,s)\in \mathbb{R} ^2\mid \mu r-\lambda s\geq 0\} \);
\ElsIf{\( \mu F-\lambda G\preceq 0 \)}
\State \Return \( II_\succeq (F,G)\leftarrow \{(r,s)\in \mathbb{R} ^2\mid \mu r-\lambda s\leq 0\} \);
\Else
\State \Return \( II_\succeq (F,G)\leftarrow \{(r,s)\in \mathbb{R} ^2\mid \mu r-\lambda s=0\} \);
\EndIf
\EndIf
\State \Comment{\( \{F,G\} \) is linearly independent}
\State \(R\leftarrow \) a matrix basis for \(\mathcal N(F)\cap \mathcal N(G)\).
\State \(S \leftarrow \) a matrix basis for \(\mathcal N(R^T )\).
\State \( \Psi (\xi )\leftarrow \det (S^T FS + \xi S^T GS)\);
\If {\( \Psi \) is a constant polynomial \textbf{ or } \( \Psi \) has at least one nonreal root}
\State \Return \( II_\succeq (F,G)\leftarrow \{(0,0)\} \);
\EndIf
\State \( \{\xi _1>\xi _2>\ldots >\xi _m\} \leftarrow \) the collection of all roots \( \xi \) of \( \Psi (\xi ) \);
\If {\( \det (S^T GS)\neq 0 \)}
\State \( l\leftarrow m; \)
\Else 
\State \( l\leftarrow m+1; \)
\State \( (\lambda _l,\mu _l)\leftarrow (0,-1) \);
\State \( (\lambda _{2l},\mu _{2l})\leftarrow (0,1) \);
\EndIf
\State \( (\lambda _i,\mu _i)\leftarrow \left (\frac{1}{\sqrt{1+\xi _i^2}},\frac{\xi _i}{\sqrt{1+\xi _i^2}} \right ),~\forall i=1,2,\ldots ,m \);
\State \( (\lambda _i,\mu _i)\leftarrow \left (\frac{-1}{\sqrt{1+\xi _{i-l}^2}},\frac{-\xi _{i-l}}{\sqrt{1+\xi _{i-l}^2}} \right )=(-\lambda _{i-l},-\mu _{i-l}),~ \forall i=l+1,l+2,\ldots ,l+m \);
\If{\( \lambda _iF+\mu _iG\nsucceq 0,~\forall i=1,2,\ldots ,2l \)}
\State \Return \( II_\succeq (F,G)\leftarrow \{(0,0)\} \);
\EndIf
\State \( i_0 \leftarrow \min \{i\in \{1,2,\ldots ,2l\}\mid \lambda _iF+\mu _iG\succeq 0\} \);
\If{\( i_0=1 \)}
\If{\( \lambda _2F+\mu _2G\succeq 0 \)}
\State \Return \( II_\succeq (F,G)\leftarrow \{(t\lambda _1 + t' \lambda _2, t\mu _1 + t'\mu _2) \mid t\geq 0,t'\geq 0\} \);
\ElsIf{\( \lambda _{2l}F+\mu _{2l}G\succeq 0 \)}
\State \Return \( II_\succeq (F,G)\leftarrow \{(t\lambda _1 + t' \lambda _{2l}, t\mu _1 + t'\mu _{2l}) \mid t\geq 0,t'\geq 0\} \);
\Else
\State \Return \( II_\succeq (F,G)\leftarrow \{(t\lambda _1, t\mu _1) \mid t\geq 0\} \);
\EndIf
\ElsIf{\( 1<i_0<2l \)}
\If{\( \lambda _{i_0+1}F+\mu _{i_0+1}G\succeq 0 \)}
\State \Return \( II_\succeq (F,G)\leftarrow \{(t\lambda _{i_0} + t' \lambda _{i_0+1}, t\mu _{i_0} + t'\mu _{i_0+1}) \mid t\geq 0,t'\geq 0\} \);
\Else
\State \Return \( II_\succeq (F,G)\leftarrow \{(t\lambda _{i_0}, t\mu _{i_0}) \mid t\geq 0\} \);
\EndIf
\Else \Comment{\( i_0=2l \)}
\State \Return \( II_\succeq (F,G)\leftarrow \{(t\lambda _{2l}, t\mu _{2l} )\mid t\geq 0\} \);
\EndIf
\end{algorithmic}
\end{algorithm}

\end{appendices}
\endgroup

\end{document}